\documentclass{amsart}

\usepackage{verbatim}
\usepackage{xcolor}
\usepackage{tikz}
\usepackage{tikz-cd}
\usepackage{arydshln}
\usepackage{caption}
\usepackage{subcaption}
\usepackage{graphicx}
\usepackage{gensymb}
\usepackage{mathrsfs}
\usepackage{appendix}

\let\stdthebibliography\thebibliography
\let\stdendthebibliography\endthebibliography
\renewenvironment*{thebibliography}[1]{%
  \stdthebibliography{ABC+20}}
  {\stdendthebibliography}

\tikzset{
	symbol/.style={
		draw=none,
		every to/.append style={
			edge node={node [sloped, allow upside down, auto=false]{$#1$}}}
	}
}

\usepackage{amsmath}
\usepackage{bbm}
\usepackage{hyperref}

\newtheorem{theorem}{Theorem}[section]
\newtheorem{proposition}[theorem]{Proposition}
\newtheorem{lemma}[theorem]{Lemma}
\newtheorem{corollary}[theorem]{Corollary}

\theoremstyle{definition}
\newtheorem{definition}[theorem]{Definition}

\theoremstyle{remark}
\newtheorem{remark}[theorem]{Remark}
\newtheorem{claim}[theorem]{Claim}

\numberwithin{equation}{section}

\newcommand{\agamma}{\vec{\gamma}}

\newcommand{\RR}{\mathbb{R}}

\DeclareMathOperator{\Mod}{Mod}
\newcommand{\PML}{\mathcal{PML}}
\newcommand{\ML}{\mathcal{ML}}
\newcommand{\T}{\mathcal{T}}
\newcommand{\M}{\mathcal{M}}
\DeclareMathOperator{\Stab}{Stab}
\newcommand{\cN}{\mathcal{N}}
\newcommand{\cO}{\mathcal{O}}
\newcommand{\PoM}{\mathcal{P}^1\mathcal{M}}
\newcommand{\MRG}{\mathcal{MRG}}
\newcommand{\bfL}{\mathbf{L}}
\newcommand{\RGmetric}{\mathbf{x}}
\newcommand{\vg}{\vec{\gamma}}
\newcommand{\mrg}{\mathbf{x}}
\newcommand{\mrgy}{\mathbf{y}}
\newcommand{\WP}{\mathrm{wp}}
\DeclareMathOperator{\arctanh}{arctanh}
\DeclareMathOperator{\sech}{sech}

\begin{document}

\title[The shapes of complementary subsurfaces]{The shapes of complementary subsurfaces to simple closed hyperbolic multi-geodesics}

\author{Francisco Arana--Herrera}

\email{farana@ias.edu}

\address{Institute for Advanced Study, 1 Einstein Dr, Princeton, NJ 08540, USA}
	
\author{Aaron Calderon}

\email{aaroncalderon@uchicago.edu}

\address{Department of Mathematics, The University of Chicago, 5734 S. University Ave, Chicago, IL 60637, USA}


\date{\today}

\begin{abstract}
Cutting a hyperbolic surface $X$ along a simple closed multi-geodesic results in a hyperbolic structure on the complementary subsurface.
We study the distribution of the shapes of these subsurfaces in moduli space as boundary lengths go to infinity, showing that they equidistribute to the Kontsevich measure on a corresponding moduli space of metric ribbon graphs.
In particular, random subsurfaces look like random ribbon graphs, a law which does not depend on the initial choice of $X$.
This result strengthens Mirzakhani's famous simple closed multi-geodesic counting theorems for hyperbolic surfaces.
\end{abstract}

\maketitle

\thispagestyle{empty}

\tableofcontents

\section{Introduction}

In \cite{Mir08b}, Mirzakhani showed that the number of simple closed multi-geodesics of a given topological type and length $\leq L$ on a closed, connected, oriented hyperbolic surface of genus $g \geq 2$ is asymptotic to a polynomial of degree $6g-6$. Each one of these multi-geodesics gives rise to a (potentially disconnected) compact, oriented hyperbolic surface with totally geodesic boundary obtained by cutting the original hyperbolic surface along the multi-geodesic.
The geometry of these subsurfaces can be encoded by their spine, represented as a metric ribbon graph \cite{Luo}.

In this paper we show that these metric ribbon graphs, appropriately rescaled, equidistribute to the Kontsevich measure (see \S\ref{sec:prelim} for a definition) on the corresponding moduli space, thus strengthening Mirzkahani's original result.
See Theorem \ref{theo:main} below.
We emphasize that the limiting distribution of shapes of subsurfaces does not depend on our initial choice of hyperbolic surface.

\begin{figure}
    \centering
    \includegraphics[scale=.5]{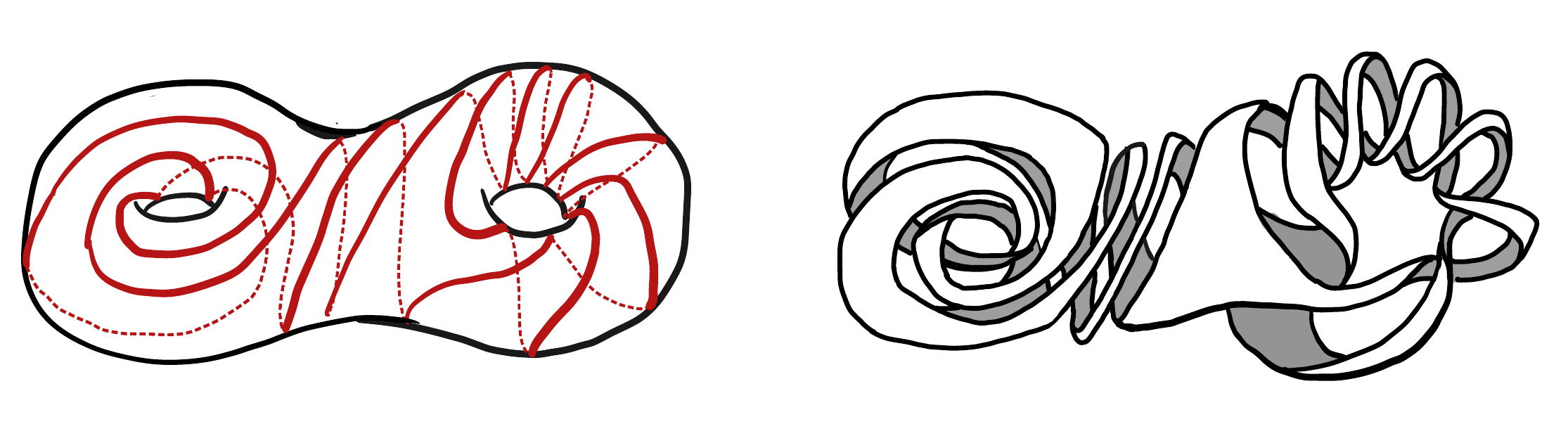}
    \caption{A long curve and its complementary subsurface.}
    \label{fig:genus2}
\end{figure}

The first result along these lines can be found in work of Mirzakhani \cite{Mir16}, where she studies the distribution of the lengths of individual components of pants decompositions.
In that case, the geometry of the complementary subsurfaces, pairs of pants, is completely determined by the lengths of their boundary components.
More general results concerning the distribution of individual lengths of simple closed multi-curves can be found in \cite{Liu19,Ara20a,ES20}.
None of these results can be directly used to study the geometry of complementary subsurfaces.

Another motivation for our result comes from homogeneous dynamics.  There is a general analogy between moduli spaces of Riemann surfaces and spaces of unimodular lattices (two theories that coincide when the surfaces have genus 1 and the lattices have rank 2) which has yielded extraordinary results; see for example \cite{MW02, EM, EMM}. In \cite{AES16a,AES16b,ERW19}, Aka, Einsiedler, Shapira, Rühr, and Wirth studied the equidistribution properties of integer points on spheres and their orthogonal lattices, generalizing famous results of Linnik \cite{Linnik} and Duke \cite{Duke}. The main result of this paper can be seen as a generalization in the same spirit in the context of moduli spaces of Riemann surfaces.

The main difficulty of the problem at hand comes from the fact that on moduli spaces of metric ribbon graphs there is no access to an ergodic flow that can be used to study the equidistribution of points via standard methods. Our approach is based on a series of intricate reductions that reduce the problem to an equidistribution question concerning the dynamics of the earthquake flow on moduli spaces of hyperbolic surfaces. These reductions combine the coarse geometry of train tracks with Margulis's well-known averaging and unfolding techniques \cite{Mar04}.

Our result also provides a surprising
new perspective on the structure and distribution of simple closed multi-geodesics on hyperbolic surfaces.
By work of Birman and Series \cite{BS85} and Erlandsson and Souto \cite{ES21}, simple closed geodesics on a hyperbolic surface $X$ are sparsely distributed on $X$ and in its unit tangent bundle.
By work of Huang, Ohshika, and Papadopoulos \cite{HOP}, the way that simple closed geodesics wind around $X$ (specifically, the shape of the ball of unit-length measured laminations on $X$) completely determines $X$ as a point in $\T_g$.
The main result of this paper guarantees that, regardless of this sparsity and independently of the geometry of the initial hyperbolic surface, simple closed multi-geodesics wind around the surface in such a way that there is no bias in the shape of their complementary subsurfaces.

This result also provides a new procedure for sampling random metric ribbon graphs. In particular, the following fundamental principle can be deduced as a direct consequence of our work: the geometry of random metric ribbon graphs is completely reflected in the hyperbolic geometry of any single hyperbolic surface.

\subsection*{Statement of the main theorem.} To avoid overwhelming the reader with lots of notation, for the moment we only state our main result in the (highly non-trivial) case of non-separating simple closed curves.
The statement of the general case, and the description of the necessary background, is postponed to \S \ref{sec:general}.

For the rest of this discussion fix an integer $g \geq 2$ and denote by $S_g$ a closed, connected, oriented surface of genus $g$. Let $\mathrm{Mod}_g$ be the mapping class group of $S_g$, $\mathcal{T}_g$ be the Teichmüller space of marked hyperbolic structures on $S_g$, and $\mathcal{M}_g$ be the moduli space of hyperbolic structures on $S_g$. Free homotopy classes of unoriented simple closed curves on $S_g$ will be refered to as simple closed curves. 
Given a simple closed curve $\alpha$ on $S_g$ and a marked hyperbolic structure $X \in \mathcal{T}_g$, denote by $\ell_\alpha(X) > 0$ the length of the unique geodesic representative of $\alpha$ with respect to $X$.

Let $\gamma$ be a non-separating simple closed curve on $S_g$ and $X \in \mathcal{T}_g$ be a marked hyperbolic structure on $S_g$. For every $L > 0$ consider the counting function
\[
s(X,\gamma,L) := \# \{\alpha \in \mathrm{Mod}_g \cdot \gamma \ | \ \ell_\alpha(X) \leq L \}.
\]
This function does not depend on the marking of $X \in \mathcal{T}_g$ but only on its underlying hyperbolic strucure $X \in \mathcal{M}_g$. Indeed, it is equal to the number of non-separating simple closed geodesics on $X$ of length $\leq L$.
By Mirzakhani's seminal work \cite{Mir08b}, there exist constants $c(\gamma) > 0$, $B(X) > 0$, and $b_g > 0$ such that
\begin{equation*}
\lim_{L \to \infty} \frac{s(X,\gamma,L)}{L^{6g-6}} = \frac{c(\gamma) \cdot B(X)}{b_g}.
\end{equation*}

Denote by $\mathcal{MRG}_{g-1,2}(1,1)$ the moduli space of metric ribbon graphs of genus $g-1$ with two boundary components, each of length $1$. Given a non-separating simple closed curve $\alpha$ on $S_g$ and a marked hyperbolic structure $X \in \mathcal{T}_g$, denote by $\mathrm{RSC}_\alpha(X) \in \mathcal{MRG}_{g-1,2}(1,1)$ the metric ribbon graph obtained by $(\mathrm{C})$ cutting $X$ along the geodesic representative of $\alpha$, $(\mathrm{S})$ constructing the spine of the resulting hyperbolic surface with totally geodesic boundary, and $(\mathrm{R})$ rescaling this spine so that each of the boundary components has length $1$.
The resulting metric ribbon graph $\mathrm{RSC}_\alpha(X) \in \mathcal{MRG}_{g-1,2}(1,1)$ encodes the geometry of the complementary subsurface of $\alpha$ in $X$.
See \S \ref{sec:prelim} for more details on this construction.

On $\mathcal{MRG}_{g-1,2}(1,1)$ consider the counting measure
\[
\eta_{X,\gamma}^L := \sum_{\alpha \in \mathrm{Mod}_g \cdot \gamma} \mathbbm{1}_{[0,L]}(\ell_\alpha(X)) \cdot \delta_{\mathrm{RSC}_\alpha(X)}.
\]
Just like the counting function $s(X,\gamma,L)$, this does not depend on the marking but only on the underlying hyperbolic structure.
Denote by $\eta_\mathrm{Kon}$ the measure on $\mathcal{MRG}_{g-1,2}(1,1)$ arising from the top power of Kontsevich's symplectic form (see Section \ref{sec:prelim}) and by $c_g := \eta_\mathrm{Kon}(\mathcal{MRG}_{g-1,2}(1,1))> 0$ its total mass.
The following theorem, which shows that the complementary subsurfaces of simple closed non-separating geodesics equidistribute over $\mathcal{MRG}_{g-1,2}(1,1)$, is an instance of the main result of this paper. For the general version see Theorem \ref{theo:main_2_new} and also Theorem \ref{theo:main_3_new} for an even stronger version corcerning simutanenous equidistribution.

\begin{theorem}
\label{theo:main}
Let $\gamma$ be a non-separating simple closed curve on $S_g$ and $X \in \mathcal{M}_g$. Then, with respect to the weak-$\star$ topology for measures on $\mathcal{MRG}_{g-1,2}(1,1)$,
\[
\lim_{L \to \infty} \frac{\eta_{X,\gamma}^L}{s(X,\gamma,L)} = \frac{\eta_{\mathrm{Kon}}}{c_g}.
\]
\end{theorem}

\begin{remark}
As will be discussed in Corollary \ref{cor:identity}, the constants $c(\gamma)$ and $c_g$ are related by the following identity:
\[
c(\gamma) = \frac{c_g}{12g-12}.
\]
\end{remark}

\subsection*{Main ideas of the proof.} To prove Theorem \ref{theo:main} we consider the following equivalent reformulation.
Let $f \colon \mathcal{MRG}_{g-1,2}(1,1) \to \mathbf{R}_{\geq 0}$ be a non-zero, non-negative, continuous, compactly supported function.
Then for any non-separating simple closed curve $\gamma$, any $X \in \mathcal{T}_g$, and any $L > 0$ define the $f$-weighted counting function 
\begin{align*}
c(X,\gamma,f,L) &:= \int_{\mathcal{MRG}_{g-1,2}(1,1)} f(x) \thinspace d\eta_{X,\gamma}^L(x)\\
&\phantom{:}= \sum_{\alpha \in \mathrm{Mod}_g \cdot \gamma} \mathbbm{1}_{[0,L]}(\ell_\alpha(X)) \cdot f(\mathrm{RSC}_\alpha(X)).
\end{align*}
Again, this counting function is independent of markings (and is also independent of our choice of $\gamma$ among all non-separating simple closed curves).
With this notation in place, Theorem \ref{theo:main} admits the following equivalent reformulation.

\begin{theorem}
\label{theo:red_intro}
Let $\gamma$ be a non-separating simple closed curve on $S_g$, $X \in \mathcal{M}_g$ be a hyperbolic structure on $S_g$, and $f \colon \mathcal{MRG}_{g-1,2}(1,1) \to \mathbf{R}_{\geq 0}$ be a non-zero, non-negative, continuous, compactly supported function. Then,
\[
\lim_{L \to \infty} \frac{c(X,\gamma,f,L)}{s(X,\gamma,L)} = \frac{1}{c_g} \int_{\mathcal{MRG}_{g-1,2}(1,1)} f(x) \thinspace d\eta_\mathrm{Kon}(x).
\]
\end{theorem}

It is interesting to note that, rather than reducing a counting problem to an equidistribution question, as is more usual, the approach in this paper is to reduce the initial equidistribution question to a counting problem.

Once in this setting, we apply Margulis's ``averaging and unfolding'' techniques \cite{Mar04} to reduce the counting problem at hand to an equidistribution question over $\M_g$.
The whole point of this reduction is to reduce the original equidistribution question over moduli spaces of ribbon graphs to an equidistribution question over the dynamically richer moduli spaces of hyperbolic surfaces.

There is an important issue that arises during the ``averaging'' step: one needs uniform control over how the metric ribbon graph $\mathrm{RSC}_\alpha(X)$ varies as $X \in \mathcal{T}_g$ does. To this end we show that for most non-separating simple closed curves $\alpha$ we can attain this control; see Lemma \ref{lem:edgespersist} and Proposition \ref{prop:weightsvary}. We also show that the number of non-separating simple closed curves for which we do not attain this control is negligible for the counting problem at hand; see Proposition \ref{prop:asymp_triv}.

This control allows us to perform the averaging and unfolding step of our argument, after which our original problem reduces to a question regarding the equidistribution of certain subsets, which we call ``RSC-horoballs,'' in the moduli space of hyperbolic surfaces.
To tackle this question we use the ergodicity of the earthquake flow, a result proved by Mirzakhani in \cite{Mir08a}.
This step is reminiscent of ideas in \cite{Ara19b}, but important modifications need to be made given the more intricate nature of the horoballs at hand. For precise statements see Theorem \ref{theo:horoball_equid_1} and Corollary \ref{cor:equid}.

An important role in the proof is played by a result of Mondello \cite{Mo09} and Do \cite{Do10} showing that the Weil--Petersson volume form on moduli spaces of hyperbolic surfaces with totally geodesic boundary converges to the Kontsevich volume form on the corresponding moduli space of metric ribbon graphs as the lengths of the boundary components become infinitely large.
These results are used to establish an important relationship between the total mass of RSC-horoball measures and the Kontsevich measure; see Proposition \ref{prop:total_mass}.

\subsection*{Organization of the paper.}
In \S\ref{sec:prelim} we cover the preliminaries needed to understand the proofs of the main results of the paper, discussing moduli spaces of hyperbolic surfaces and ribbon graphs, as well as measures on them, and important links between them.
In \S\ref{sec:asymptriv} we show that most simple closed multi-geodesics of a given topological type on a hyperbolic surface have complementary subsurfaces whose spines are trivalent with long edges.
In \S\ref{sec:weightsvary} we show that the weights of the edges of such spines vary uniformly as the base hyperbolic surface varies over moduli space.
In \S\ref{sec:RSCequi} define RSC-horoballs and show that they equidistribute over moduli space.
In \S\ref{sec:avg_unfold} we use the previous results to prove the main result of this paper for non-separating simple closed geodesics.
In \S\ref{sec:general} we discuss the generalization of this result for general simple closed multi-geodesics as well as a stronger simultaneous equidistribution result.

\subsection*{Acknowledgments.}
The authors would like to thank Alex Wright and Yair Minsky for enlightening conversations on several topics related to this paper.
The authors would also like to thank Yair Minsky for supplying a more direct proof of Lemma \ref{lem:edgespersist} than was originally included and Scott Wolpert for clarifying the etymology of some of the cited results.
This work was completed while FAH was a member of the Institute for Advanced Study (IAS) and he is very grateful to the IAS for its hospitality.
This material is based upon work funded by the National Science Foundation: FAH was supported by grant DMS-1926686 and AC was supported by grants
DGE-1122492, 
DMS-2005328,
and DMS-2202703.

\section{Hyperbolic surfaces and ribbon graphs}
\label{sec:prelim}

\subsection*{Outline of this section.}
We begin this section with a brief introduction to moduli spaces of hyperbolic surfaces and their Weil-Petersson volumes. We then briefly recall the theory of measured geodesic laminations and the Thurston measure. We proceed to discuss moduli spaces of metric ribbon graphs and the Kontsevich measure. We finish by discussing the different relations between the moduli spaces of hyperbolic surfaces and metric ribbon graphs as well as the relations between the Weil-Petersson and Kontsevich measures.

\subsection*{Moduli of hyperbolic surfaces.} For a connected, oriented, closed surface $S_g$ of genus $g \ge 2$, we denote by $\mathcal{T}_g$ the Teichmüller space of marked hyperbolic structures on $S_g$.
The mapping class group of $S_g$, denoted $\mathrm{Mod}_g$, is the group of orientation preserving homeomorphisms of $S_g$ up to homotopy.
This group acts naturally on $\mathcal{T}_g$ by changing the markings.
This action is properly discontinuous, and its quotient orbifold $\mathcal{M}_g := \mathcal{T}_g / \mathrm{Mod}_g$ is the moduli space of genus $g$ hyperbolic surfaces.
By uniformization, this space can be canonically identified with the moduli space of genus $g$ Riemann surfaces.

For a compact surface $S_{g,b}$ of genus $g$ with $b$ totally geodesic boundary components, we can repeat the above discussion.
In this setting, we consider (marked) hyperbolic structures up to homotopies that fix each boundary component {\em setwise}.
\footnote{There is another standard definition in which one considers structures up to homotopies {\em pointwise} fixing the boundary.
This pointwise-fixed version is a $b$-dimensional torus bundle over the space we consider here.}
To define the mapping class group $\Mod_{g,b}$ we consider only consider those orientation-preserving homeomorphisms that also preserve each boundary component, and similarly consider them up to homotopy setwise fixing each boundary.
Both the Teichm{\"u}ller space $\T_{g,b}$ and moduli space $\M_{g,b}$ are fibered by slices with the same boundary lengths; for any vector $\bfL \in \mathbb{R}_{>0}^b$, we use $\mathcal{M}_{g,b}(\mathbf{L})$ to denote the moduli space of hyperbolic structures on $S_{g,b}$ with labeled geodesic boundary components of length $\mathbf{L}$. 

\subsection*{The Weil-Petersson symplectic form.} The Teichmüller space $\mathcal{T}_{g}$ can be endowed with a $3g-3$ dimensional complex structure. This complex structure admits a natural Kähler Hermitian structure. The associated symplectic form $\omega_{\text{wp}}$ is called the Weil-Petersson symplectic form. The Weil-Petersson volume form $v_\mathrm{wp}$ is the top exterior power of this symplectic form, 
\[v_{\text{wp}} := \frac{1}{(3g-3)!}\bigwedge^{3g-3} \omega_{\text{wp}}.\]
The Weil-Petersson measure $\mu_{\text{wp}}$ on $\mathcal{T}_{g}$ is the measure induced by $v_{\text{wp}}$. By work of Wolpert \cite{Wol85}, this measure coincides with Lebesgue measure on any set of Fenchel-Nielsen coordinates. The Weil-Petersson measure $\widehat{\mu}_{\text{wp}}$ on $\mathcal{M}_{g}$ is the local pushforward of $\mu_{\text{wp}}$ under the map $\mathcal{T}_{g} \to \mathcal{M}_{g}$.
See \cite{Hub16} for more details.

Again, when the surface has boundary one can also define a Weil--Petersson 2-form.
Restricted to each slice $\mathcal{T}_{g,b}(\mathbf{L})$ with fixed boundary lengths this form turns out to be symplectic, and we use $\widehat{\mu}_{\WP}^{\bfL}$ to denote the measure on $\mathcal{M}_{g,b}(\mathbf{L})$ coming from the associated symplectic volume form.

\subsection*{Measured geodesic laminations.} Fix a marked hyperbolic structure $X \in \mathcal{T}_g$. A geodesic lamination on $X$ is a closed subset of $X$ that can be written as a disjoint union of simple geodesics. A measured geodesic lamination is a geodesic lamination endowed with a fully supported invariant transverse measure. The transverse measure assigns a finite Borel measure to every arc transverse to the lamination. This assignment is invariant under splitting of arcs and homotopies preserving the leaves of the lamination. We denote by $\mathcal{ML}_X$ the space of measured geodesic laminations on $X$.

The different spaces of measured geodesic laminations $\mathcal{ML}_X$ obtained as $X$ varies over $\mathcal{T}_g$ can be canonically identified to each other. We denote by $\mathcal{ML}_g$ any such space and refer to it as the space of measured geodesic laminations on $S_g$. This space carries a natural $(\mathbb{R}_{>0}, \times)$ action that scales transverse measures. Simple closed multi-curves embed naturally into $\mathcal{ML}_g$ by considering their geodesic representatives endowed with transverse dirac measures. We denote by $\mathcal{ML}_g(\mathbb{Z}) \subseteq \mathcal{ML}_g$ the set of integral simple closed multi-curves on $S_g$. By work of Thurston \cite{T80}, every measured geodesic lamination $\lambda \in \mathcal{ML}_g$ has a well defined hyperbolic length $\ell_\lambda(X) > 0$ with respect to any marked hyperbolic structure $X \in \mathcal{T}_g$.

\begin{sloppypar}
\subsection*{The Thurston measure.} 
Train track coordinates induce a $6g-6$ dimensional piecewise integral linear (PIL) structure on the space $\mathcal{ML}_{g}$ of measured geodesic laminations on $S_{g}$; see \cite[\S 3.1]{PH92} for details. By work of Masur, there exists a unique (up to scaling) non-zero, locally finite, $\text{Mod}_{g}$-invariant, Lebesgue class measure on $\mathcal{ML}_{g}$ \cite[Theorem 2]{Mas85}. Several different definitions of such measure (equal up to scaling) can be found in the literature. We will consider the definition coming from the symplectic structure of $\mathcal{ML}_{g}$.
\end{sloppypar}

More precisely, consider the $\text{Mod}_{g}$-invariant symplectic form $\omega_{\text{Thu}}$ on the PIL manifold $\mathcal{ML}_{g}$ induced by train track coordinates; see \cite[\S 3.2]{PH92} for an explicit definition. This symplectic form is known as the \textit{Thurston symplectic form}. The top exterior power $$v_\text{Thu} := \frac{1}{(3g-3)!} \bigwedge_{i=1}^{3g-3} \omega_{\text{Thu}}$$ induces a non-zero, locally finite, $\text{Mod}_{g}$-invariant, Lebesgue class measure $\mu_{\text{Thu}}$ on $\mathcal{ML}_{g}$. We refer to this measure as the \textit{Thurston measure}.

\subsection*{Ribbon graphs and dual arc systems}
A ribbon graph is a (simplicial) graph equipped with a cyclic ordering of the edges incident to each vertex.
It is useful to think of a ribbon graph $\Gamma$ as encoding a deformation retraction of a surface with boundary; this process can be reversed by replacing each edge of $\Gamma$ with a ribbon and connecting the borders of the ribbons according to the cyclic ordering.
The genus and number of boundary components of $\Gamma$ are the values for the resulting topological surface. 

One may also equip a ribbon graph with a metric $\RGmetric \in \RR_{>0}^{e(\Gamma)}$ that assigns a length to each of its edges; we denote by $\MRG_{g,b}$ the space of all metric ribbon graphs with genus $g$ and $b$ distinctly-labeled boundary components, all of whose vertices have valence at least three.
Using meromorphic quadratic differentials, Jenkins \cite{Jenkins} and Strebel \cite{Strebel} proved that $\MRG_{g,b}$ is homeomorphic to the usual moduli space $\M_{g,b}$.
Other proofs were given by Penner \cite{Pen} and Bowditch and Epstein \cite{BowdEpst} using cusped hyperbolic surfaces; see also Theorem \ref{thm:Luo} below for a similar statement using the geometry of hyperbolic surfaces with boundary.

It is often helpful to consider the dual arcs to a ribbon graph.
Namely, if $\Sigma$ is a surface of genus $g$ with $b$ boundaries that deformation retracts onto a given ribbon graph $\Gamma$, then the fibers of this retraction break up into a finite union of isotopy classes of properly embedded arcs. Since $\Gamma$ is a spine for $\Sigma$, this arc system must fill $\Sigma$, i.e., cut it into disks.
As such, the combinatorics of $\Gamma$ are completely captured by the dual filling arc system.
Similarly, if $\Gamma$ is equipped with a metric $\RGmetric$, we may assign the weight $\RGmetric(e)$ to the arc dual to the edge $e$ of $\Gamma$; the length of a boundary component of $\Gamma$ then corresponds to the sum of the weights of the arcs incident to the corresponding boundary component of $\Sigma$ (counted with multiplicity).
This discussion shows that there is a homeomorphism
\[\MRG_{g,b} \cong |{\mathscr{A}_\text{fill}}(\Sigma)|_{\RR} / \Mod_{g,b}\]
where $|{\mathscr{A}_\text{fill}}(\Sigma)|_{\RR}$ denotes the space of all positively-weighted, filling arc systems on a surface $\Sigma$ of genus $g$ with $b$ boundary components.
For more on this correspondence, see \cite{Mond_handbook}.

The dual viewpoint is useful because it allows us more insight into the combinatorial structure of this moduli space.
For example, the space of all arc systems (the {\em arc complex}) ${\mathscr{A}}(\Sigma)$ has a natural simplicial structure with vertices isotopy classes of arcs and simplices corresponding to collections of pairwise disjoint arcs. Let $|{\mathscr{A}}(\Sigma)|$ denote its geometric realization.
The space of all filling arc systems $|{\mathscr{A}_\text{fill}}(\Sigma)|$ can then be realized as a subspace of $|{\mathscr{A}}(\Sigma)|$, and from the combinatorial structure of the arc complex it inherits a faceted structure.
The maximal-dimension facets of $|{\mathscr{A}_\text{fill}}(\Sigma)|$ correspond to maximal (filling) arc systems, equivalently, trivalent ribbon graphs, and higher codimension facets correspond to arc systems which are not maximal but which still fill, equivalently, ribbon graphs with higher-valence vertices.
With this description, it is apparent that the faceted structure of $|{\mathscr{A}_\text{fill}}(\Sigma)|$ is locally finite, as there are only finitely many completions of any given filling arc system.
\footnote{One should note that the simplicial topology on the arc complex and the topology coming from its geometric realization are not the same \cite{BowdEpst}, but restricted to $|{\mathscr{A}_\text{fill}}(\Sigma)|$ the two agree because of local finiteness.}

The space of weighted, filling arc systems $|{\mathscr{A}_\text{fill}}(\Sigma)|_{\RR}$ can then be identified with $|{\mathscr{A}_\text{fill}}(\Sigma)| \times \RR_{>0}$, and so it too has a faceted structure.
Moreover, it also has a natural $(\RR_{>0}, \times)$ action giving by multiplying each weight by the same factor.
This action is clearly $\Mod_{g,b}$ equivariant, and so descends to an action on the combinatorial moduli space.
Denote by $\mathrm{R}_t$ the rescaling map
\[\mathrm{R}_t:\MRG_{g,b} \to \MRG_{g,b}\]
that {\em divides} each length by a factor of $t$.

For a given tuple $\bfL = (L_1, \ldots, L_b)$ of positive numbers, we define $\MRG_{g,b}(\bfL)$ to be the moduli space of all ribbon graphs of genus $g$ which have $b$ boundary components with lengths $L_1, \ldots, L_b$.
The slices $\MRG_{g,b}(\bfL)$ piece together to form a fibration
$\MRG_{g,b} \to \RR_{>0}^{b},$
and the rescaling map restricts to a homeomorphism
\[\mathrm{R}_t: \MRG_{g,b}(t\bfL) \xrightarrow{\sim} \MRG_{g,b}(\bfL) \]
between slices with homothetic length vectors.
In later sections, we sometimes use $\mathrm{R}$ without a subscript when the target of the rescaling is fixed.

The slice $\MRG_{g,b}(\bfL)$ inherits a faceted structure from that of $\MRG_{g,b}$.
Moreover, unless $g=0$, $b=3$, and the length vector $\bfL$ satisfies 
\[L_i + L_j = L_k \text{ for } \{i,j,k\} = \{1,2,3\},\]
the top-dimensional facets of $\MRG_{g,b}(\bfL)$ correspond to trivalent ribbon graphs.
\footnote{This can be deduced by a purely combinatorial argument or by invoking the results of \cite{GT_residue}.}
In the exceptional case, the moduli space $\MRG_{0,3}(\bfL)$ is a single point corresponding to a ribbon graph with a single vertex of valence $4$.

\subsection*{The Kontsevich measure.}
In his solution of Witten's conjecture \cite{Kon92}, Kontsevich defined a piecewise $2$-form $\omega_{\mathrm{Kon}}$ on $\MRG_{g,b}$ that computes intersection numbers between tautological classes.
While this form is not symplectic (and indeed, is not even globally well-defined) on the entirety of $\MRG_{g,b}$, it turns out that when restricted to a slice $\MRG_{g,b}(\bfL)$ with fixed boundary lengths $\omega_{\mathrm{Kon}}$ is symplectic on every maximal facet.
Thus, the Kontsevich form gives rise to volume forms
\[\frac{1}{(3g-3+b)!}
\bigwedge\nolimits^{3g-3+b}
\omega_{\mathrm{Kon}}\]
on each maximal facet, which can be pasted together into a volume form on the entire slice $\MRG_{g,b}(\bfL)$. We will use $\eta_{\mathrm{Kon}}^{\bfL}$ to denote the measure associated to this volume form, which is called the {\em Kontsevich measure} on $\MRG_{g,b}(\bfL)$.

\begin{remark}
To be completely precise, one should really define the Kontsevich measure on a manifold cover of $\MRG_{g,b}$ (or better yet, on the Teichm{\"u}ller space of metric ribbon graphs) then take a local pushforward to $\MRG_{g,b}$, as was done for the Weil--Petersson measure.
Defining $\eta_{\mathrm{Kon}}^{\bfL}$ in this manner gives the correct weighting by automorphisms of the ribbon graph; compare with the discussion in Section \ref{sec:general}, especially \eqref{eqn:WPcutandglue} and Remark \ref{rmk:K_normalization}.
\end{remark}

For later reference, we also record an expression for the cohomology class of $\MRG_{g,b}(\bfL)$ represented by the Kontsevich form:
\begin{equation}\label{eqn:Konts_int}
[\omega_{\mathrm{Kon}}|_{\MRG_{g,b}(\bfL)}] = \frac{1}{2}\sum L_i^2 \psi_i
\end{equation}
where $\psi_i$ denotes the first Chern class of the circle bundle associated to the $i^\text{th}$ boundary component of the surface.
We direct the reader to \cite[Section 6]{CMS_Ksympred} for a formal definition of $\omega_{\mathrm{Kon}}$ and a thorough discussion of its symplectic properties.

In Appendix C of \cite{Kon92}, Kontsevich showed that each $\eta_{\mathrm{Kon}}^{\bfL}$ is a constant multiple of the Euclidean volume form in the $\RGmetric$ coordinates (restricted to each slice).
Integrating against $\bfL$, the volume forms on these slices fit together into a volume form defined on the entire combinatorial moduli space, yielding the following relationship between the Kontsevich and Euclidean volume forms:
\begin{equation}\label{eqn:Konts=Euclid}
d \eta_{\mathrm{Kon}}^{\bfL} \wedge \bigwedge_{i=1}^b d L_i = 2^{2g-2+b} \bigwedge_{j=1}^n d x_j
\end{equation}
as measures on $\MRG_{g,b}$.
The statement recorded above appears as Lemma 3.8 of \cite{ABC}, but was essentially proved in both \cite{Kon92} and \cite{CMS_Ksympred}.

We note that \eqref{eqn:Konts_int} and \eqref{eqn:Konts=Euclid} both imply that the Kontsevich measure scales homogeneously with respect to the rescaling map $\mathrm{R}_t$:
\begin{equation}\label{eqn:Konts_rescale}
    (\mathrm{R}_{t})_* \eta_{\mathrm{Kon}}^{t\bfL} 
    = {t^{6g-6+2b}} \eta_{\mathrm{Kon}}^{\bfL}.
\end{equation}

\begin{remark}
There are two different standard definitions of the Kontsevich form, differing by a factor of $2$.
In this paper, we follow the convention of \cite{ABC} and \cite{Do10}, which results in a different power of $2$ when comparing the Kontsevich and Lebesgue measures than is computed in \cite{Kon92} and \cite{CMS_Ksympred}.
See Theorem \ref{theo:Do} below for an explanation of why this normalization makes sense in our context.
\end{remark}

\subsection*{Spines and the orthogeodesic foliation.}
The tool that allows us to pass between the moduli spaces of hyperbolic surfaces and metric ribbon graphs is the {\em orthogeodesic foliation}. We give a summary of this construction below; for more details, the reader is directed to \cite[\S 5]{CF} and \cite[\S 2]{Mo09}.

Let $Y$ denote a finite-area hyperbolic surface with boundary; we allow $\partial Y$ to consist of both closed geodesics as well as hyperbolic crowns. Then the orthogeodesic foliation $\cO_{\partial Y}(Y)$ of $Y$ rel boundary is the (singular) foliation of $Y$ whose leaves are fibers of the closest-point projection map to $\partial Y$. See Figure \ref{fig:orthofol}.
If $X$ is a closed or cusped hyperbolic surface equipped with a geodesic lamination $\lambda$, then the orthogeodesic foliations of the (metric completions of the) components of $X \setminus \lambda$ glue together into a global singular foliation $\cO_{\lambda}(X)$ of $X$.

\begin{figure}[ht]
    \centering
    \includegraphics[scale=.7]{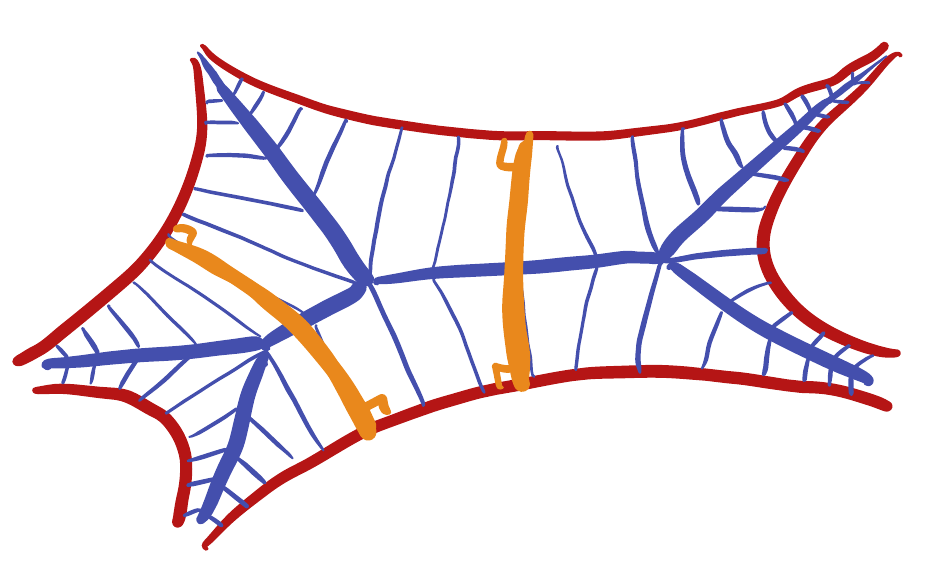}
    \caption{The orthogeodesic foliation of an ideal hyperbolic pentagon, together with its spine and dual arc system.}
    \label{fig:orthofol}
\end{figure}

We observe that if leaves of $\lambda$ on $X$ are close, then they are joined by a segment of $\cO_{\lambda}(X)$, as the following lemma shows.

\begin{lemma}\label{lem:closeimpliesarc}
Suppose that $\lambda$ is a geodesic lamination on $\mathbb{H}^2$ and $\ell_1$ and $\ell_2$ are leaves of $\lambda$ that are not separated from each other by any other leaves of $\lambda$.
Then if $\ell_1$ and $\ell_2$ are at distance less than $\log \sqrt{3}$, they are joined by a segment of the orthogeodesic foliation $\mathcal{O}_{\lambda}(\mathbb{H}^2)$.
\end{lemma}
\begin{proof}
The radius of the inscribed circle in any ideal triangle in $\mathbb{H}^2$ is $\log \sqrt{3}$; in particular, this implies that if three geodesics of $\mathbb{H}^2$ all have distance $d < \log \sqrt{3}$ from a point, then some of them must intersect.

So now let $x$ be a point at at some distance $d < \log \sqrt{3}$ from both $\ell_1$ and $\ell_2$.
Consider the closest-point projection from $x$ to $\lambda$; if it does not map to $\ell_1$ or $\ell_2$, then there is some closer leaf of $\lambda$.
However, this leaf must be disjoint, and thus must separate $\ell_1$ from $\ell_2$, in contradiction to our assumption.

Thus, the closest-point projection from $x$ to $\lambda$ maps to both $\ell_1$ and $\ell_2$, demonstrating the desired segment of the orthogeodesic foliation.
\end{proof}

The combinatorics of the orthogedeodesic foliation can be used to capture the geometry of a hyperbolic structure. For the remainder of the section, we restrict ourselves to compact hyperbolic surfaces $Y$ with boundary; for the general picture, the reader is directed to \cite[Section 5]{CF}.

For any point of $Y$, its {\em valence} is the number of closest-point projections to $\partial Y$ that it has.
The set of points which have valence at least 2 form an embedded piecewise-geodesic graph which is a spine for $Y$.
Remembering that this graph is a deformation retract of $Y$ equips it with a natural ribbon structure, and we can assign each edge $e$ the weight $|e|_Y$ given by the length of either of the closest point projections of $e$ to $\partial Y$; see Figure \ref{fig:symmetry}.
Equivalently, the arcs of $\cO_{\partial Y}(Y)$ break up into a finite union of isotopy classes, each of which contains a unique shortest (orthogeodesic) representative connecting $\partial Y$ to itself. We may then weight each such arc with the length along $\partial Y$ of the band of parallel arcs of $\cO_{\partial Y}(Y)$.
This yields a map 
\[\mathrm{S}: \M_{g,b} \to \MRG_{g,b} \cong 
|{\mathscr{A}_\text{fill}}(\Sigma)|_{\RR}/\Mod(\Sigma)\]
taking $Y$ to its metric ribbon graph spine, equivalently, the weighted, filling system of dual orthogeodesic arcs.

\begin{figure}[h]
    \centering
    \includegraphics[scale=.7]{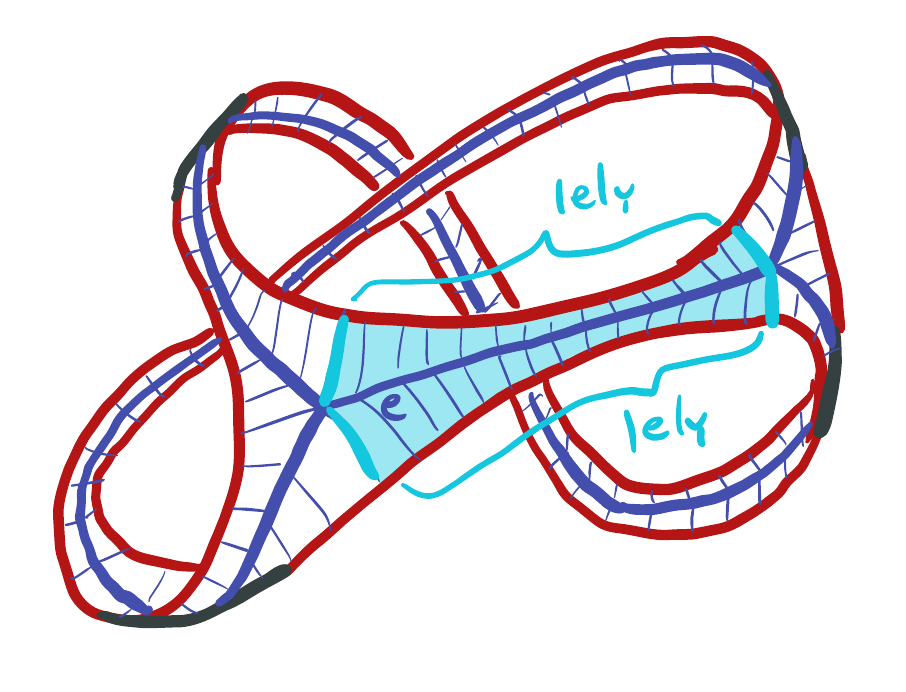}
    \caption{The metric spine of a genus 1 surface $Y$ with a totally geodesic boundary. The two projections (in blue) of the edge $e$ to the boundary $\partial Y$ have the same length $|e|_Y$.}
    \label{fig:symmetry}
\end{figure}

Luo proved that this combinatorial data completely determines the hyperbolic metric $Y$ \cite[Theorem 1.2 and Corollary 1.4]{Luo}; see also \cite{Do10} and \cite{Mo09}.

\begin{theorem}\label{thm:Luo}
For any compact surface with boundary $\Sigma = \Sigma_{g,b}$, the spine map S is a homeomorphism.
Moreover, for any $\bfL \in \RR_{>0}^n$, it restricts to a homeomorphism
\[\mathrm{S}: \M_{g,b}(\bfL) \xrightarrow{\sim} \MRG_{g,b}(\bfL)\]
between slices with fixed boundary lengths.
\end{theorem}

When there is no risk of confusion, we will use $\mathrm{S}(Y)$ to refer to both the locus of points in $Y$ which have valence at least $2$ as well as the abstract metric ribbon graph/dual arc system obtained as above.
We reserve $|e|_Y$ to denote the length of an edge of the abstract ribbon graph/the weight of a dual arc $a$, so that it always refers to a length measured along $\partial Y$, while we use $\ell_Y(a)$ to denote the length of the (unique) orthogeodesic arc in the isotopy class of $a$.

The following basic estimate is a ``collar lemma'' for arcs of the dual arc system, which allows us to show that long edges of the spine have short dual arcs and vice versa.
The proof follows from basic facts about the hyperbolic geometry of trirectangles; see \cite[Theorem 2.3.1]{Bus92} and also \cite[Lemma 6.6]{CF} (it can also be deduced by doubling $Y$ along its boundary and invoking the usual collar lemma).

We use $f(x) \sim g(x)$ to denote that the ratio of $f(x)$ and $g(x)$ tends to $1$.

\begin{lemma}\label{lem:lengthvsweight}
For any hyperbolic structure $Y$ on $\Sigma$ and any edge $e$ of $S(Y)$ with dual orthogeodesic arc $a$,
\[|e|_Y \sim \log( 1/ \ell_Y(a))\]
as $\ell_Y(a)$ becomes small or as $|e|_Y$ becomes large.
\end{lemma}

\subsection*{Asymptotics of Weil--Petersson volumes}
The spine map gives a clear way to relate the asymptotic behavior of Weil--Petersson volumes of moduli spaces of hyperbolic surfaces to the Kontsevich volumes of spaces of ribbon graphs.

Using symplectic reduction and the normal form theorem, together with Wolpert's computation \cite{Wolp_WPclass}
of the cohomology class represented by the Weil--Petersson form on $\M_{g,b}(0,\ldots, 0)$, i.e., the moduli space of hyperbolic metrics with $b$ cusps, in \cite{Mir07c} Mirzakhani gave an expression for the class of the Weil--Petersson symplectic form on $\M_{g,b}(\bfL)$:
\begin{align*}
[\omega_{\WP}|_{\M_{g,b}(\bfL)}]
& = 2\pi^2 \kappa_1 + \frac{1}{2}\sum L_i^2 \psi_i.
\end{align*}
The class $\kappa_1$ is the first Mumford tautological class and the identification of the $2\pi^2 \kappa_1$ term is due to Wolpert \cite{Wolp_WPclass}; what is relevant in this paper is just that it does not depend on $\bfL$.
As a consequence, she deduced that the Weil--Petersson volume of $\M_{g,b}(\bfL)$ is a polynomial in the $L_i$'s \cite[Theorem 1.1]{Mir07c}.

Comparing this formula with \eqref{eqn:Konts_int}, it also follows that the leading asymptotics of the Weil--Petersson volume of $\M_{g,b}(t \bfL)$ and the Kontsevich volume of $\MRG_{g,b}(t\bfL)$ are the same as $t \to \infty$.
Since the Kontsevich volume form rescales homogeneously \eqref{eqn:Konts_rescale}, this coefficient is just the Kontsevich volume of $\MRG_{g,b}(\bfL)$.
We record this as follows; compare with \cite[p. 18]{Do10}.

\begin{corollary}\label{cor:WPtop=Kont}
The top degree part of the Weil--Petersson volume polynomial of $\M_{g,b}(\bfL)$ equals the Kontsevich volume of $\MRG_{g,b}(\bfL)$.
\end{corollary}

In what follows, we will need a more precise asymptotic convergence result in order to compare the Weil--Petersson and Kontsevich measures, not just their total masses.
The following result, a direct consequence of independent work of Mondello \cite[Corollary 4.4]{Mo09} and Do \cite[Theorem 2]{Do10}, is crucial for our arguments in Section \ref{sec:RSCequi}.

\begin{theorem} \label{theo:Do}
Fix an $\bfL \in \RR^b_{>0}$. Then
\[
\lim_{t \to \infty} \frac{{(\mathrm{R}_t\mathrm{S})}_*\widehat{\mu}_{\mathrm{wp}}^{t\bfL}}{t^{6g-6+2b}} = \eta^{\bfL}_\mathrm{Kon}
\]
with respect to the weak-$\star$ topology for measures on $\MRG_{g,b}(\bfL)$.
\end{theorem}

\subsection*{Notation.}
For the reader's convenience we have collected here a list of some of the notation conventions we use throughout the paper, including some that we have not yet introduced.
We have tried to remain consistent with the conventions of \cite{Ara19b,Ara20a} as much as possible.

\begin{itemize}
    \item $\mu$ for measures on Teichmüller and moduli spaces of Riemann surfaces. In particular, $\widehat{\mu}_\mathrm{wp}$ denotes the Weil-Petersson measure on $\mathcal{M}_g$.
    \item $\nu$ for measures on bundles of unit length measured geodesic laminations over Teichmüller and moduli spaces of Riemann surfaces. In particular, $\widehat{\nu}_\mathrm{Mir}$ denotes the Mirzakhani measure on $\mathcal{P}^1\mathcal{M}_g$; see \S \ref{sec:RSCequi}.
    \item $\eta$ for measures on moduli spaces of metric ribbon graphs. In particular, $\eta_\mathrm{Kon}$ denotes the Kontsevich measure on $\mathcal{MRG}_{g-1,2}(1,1)$.
    \item $\mu_{\gamma,h}^L$ for RSC-horoball measures on $\mathcal{T}_g$. We denote by $\smash{\widetilde{\mu}_{\gamma,h}^L}$ the local pushforward of $\mu_{\gamma,h}^L$ to the intermediate cover $\mathcal{T}_g/\mathrm{Stab}(\gamma)$ and by $\smash{\widehat{\mu}_{\gamma,h}^L}$ the pushforward of $\smash{\widetilde{\mu}_{\gamma,h}^L}$ to $\mathcal{M}_g$; see \S \ref{sec:RSCequi}.
    \item $\vg := (\vg_1, \ldots, \vg_k)$ for an ordered, oriented simple closed multi-curve on a topological surface $S_g$.
    \item $S_g \setminus \vg$ for the metric completion of the corresponding cut surface. We denote its components by $\Sigma_1, \ldots, \Sigma_m$, indexed according to the order and orientation of the components of $\vg$.
    \item $\Gamma$ for ribbon graphs. Their edges are denoted by $e_1, \ldots, e_E$ and any metric structure has coordinates $\RGmetric = (x_1, \ldots, x_E)$.
    \item $\MRG(S+g \setminus \vg)$ for the product of moduli spaces of complementary subsurfaces subject to gluing conditions; see \S\ref{sec:general}
    \item $\MRG(S_g \setminus \vg; \bfL)$ for the slice of $\MRG(S_g \setminus \vg)$ with fixed lengths of boundary components.
    \item $\MRG(S_g \setminus \vg; \Delta)$ for the subset of $\MRG(S_g \setminus \vg)$ with lengths of boundary components in the standard simplex $\Delta$ of $\RR^k$.
\end{itemize}

\section{Nonmaximal facets and power saving}\label{sec:asymptriv}

\subsection*{Outline of this section.} The purpose of the next two sections is to prove that the weights of the edges of the spine $SC_{\alpha}(X)$ vary uniformly in $\alpha$ as $X$ varies in a neighborhood of Teichmüller space; this is what allows us to eventually run our ``averaging and unfolding'' arguments in \S \ref{sec:avg_unfold}. This goal will be complicated by the faceted structure of the combinatorial moduli space. As such, in this section we first prove that we need only consider those multi-curves whose corresponding spines are ``deep enough'' in maximal facets; see Proposition \ref{prop:asymp_triv}.

Our proof holds for all simple closed multi-curves and the results are phrased in generality; for a precise description of the space $\MRG(S_g \setminus \agamma)$, the reader is directed to Section \ref{sec:general}.
The reader may freely restrict to the case of a single non-separating simple closed curve and $\MRG_{g-1, 2}$ with no loss of intuition.

\subsection*{Statement of the main result.}

For any $X \in \T_g$, any simple closed multi-curve $\gamma$, and any $K>0$, let $F(X, \gamma, K)$ denote the set of $\alpha \in \Mod_g \cdot \gamma$ for which $SC_{\alpha}(X)$ is either not trivalent or has an edge of weight at most $K$. 
Equivalently, $F(X, \gamma, K)$ is the set of $\alpha$ for which $SC_{\alpha}(X)$ lies in a $K$-neighborhood of a lower-dimensional facet of $\MRG(S_g \setminus \agamma)$.
The following is the main result of this section.

\begin{proposition}\label{prop:asymp_triv}
For every $X \in \T_g$, every multi-curve $\gamma$, and every $K >0$,
\[
\frac{\# \{ \alpha \in F(X, \gamma, K) \mid \ell_X(\alpha) \le L\}} {L^{6g-6}} \to 0
\text{ as } L \to \infty.
\]
\end{proposition}

\begin{remark}
   The proof of Proposition \ref{prop:asymp_triv} actually shows that
    \[
    \#\{ \alpha \in F(X, \gamma, K) \ | \ \ell_X(\alpha) \le L\} = O\left(L^{6g-7}\right),
    \]
    corresponding to a power saving of an entire degree. 
\end{remark}

The main idea of the proof is to exploit the connection (mediated by geometric train tracks) between certain subspaces in $\ML_g$ and facets of the moduli space of metric ribbon graphs.
As we show below, the elements of $F(X, \gamma, K)$ are carried on non-maximal train tracks, bounding the growth of $F(X, \gamma, K)$ in terms of the maximal dimension of the weight space of any of these train tracks.

\subsection*{Geometric train tracks}
The uniform $\delta$-neighborhood $\cN_\delta(\lambda)$ of any geodesic lamination $\lambda$ is foliated by (the restrictions of) leaves of the orthogeodesic foliation $\cO_\lambda(X)$. 
If the collapse map extends to a $C^1$ homotopy equivalence of the surface $X$ (equivalently, if each leaf of $\cO_\lambda(X)|_{\cN_\delta(X)}$ is just an interval), then the leaf space is called a {\em train track}, and we say that $\cN_\delta(X)$ is a {\em geometric train track neighborhood.}
The train track $\tau$ can also be thought of as a graph embedded in the surface with an assignment of tangential data at each vertex.
Its edges (or {\em branches}) correspond to ``rectangles'' foliated by parallel leaves of $\cO_{\lambda}(X)|_{\cN_{\delta}(\lambda)}$, while its vertices (or {\em switches}) correspond to leaves where these rectangles are conjoined.

We begin by recording a uniform estimate on the width of a $\delta$-neighborhood of an arbitrary geodesic lamination. See also \cite[Lemma 14.5]{CF}.

\begin{lemma}\label{lem:ttwidth}
For any $X \in \T_g$ there exists a constant $\delta_0 = \delta_0(X) > 0$, uniform on the thick part of $\mathcal{T}_g$, such that if $\lambda$ is a geodesic lamination of $X$ and $0 < \delta < \delta_0$, then the length of any segment of $\cO_{\lambda}(X)|_{\cN_{\delta}(\lambda)}$ is $O_X(\delta)$, with the implicit constant being uniform on the thick part of $\mathcal{T}_g$.
\end{lemma}

\begin{proof}[Proof sketch]
Let $t$ be such a segment.
Since $t \cap \lambda$ has measure $0$, we can compute the length of $t$ by summing the length of the pieces of $t \setminus \lambda$.
But now we note that the pieces of $t \setminus \lambda$ fall into finitely many isotopy classes of arcs on $X \setminus \lambda$, and each subsequent time that an isotopy class occurs it must do so a definite distance (the injectivity radius of $X$) further into the thin part of $X \setminus \lambda$.
In particular, this means that the length of $t$ is bounded by a sum of finitely many geometric series whose first terms are all at most $2 \delta$.
\end{proof}

From Lemma \ref{lem:ttwidth} we get that the defining parameter for geometric train track neighborhoods can be taken to be uniformly large in the base lamination.

\begin{lemma}\label{lem:unifttparam}
For any $X \in \T_g$ there exists a constant $\delta_0 = \delta_0(X) > 0$, uniform in the thick part of $\mathcal{T}_g$, so that for any $0 < \delta < \delta_0$ and any geodesic lamination $\lambda$ on $X$, the uniform $\delta$-neighborhood $\cN_\delta(\lambda) \subseteq X$ is a geometric train track neighborhood.
\end{lemma}
\begin{proof}
So long as the orthogeodesic foliation $\cO_\lambda(X)$ has no closed leaves, one can take any $\delta < \log \sqrt{3}$ (this cutoff ensures that $\cN_{\delta}(\lambda)$ does not contain any vertices of the spine).
Otherwise, if it does, one can use Lemma \ref{lem:ttwidth} to ensure that the length of any segment of $\cO_{\lambda}(X)$ in $\cN_\delta(\lambda)$ is also less than the systole of $X$ and so $\cO_\lambda(X)|_{\cN_\delta(\lambda)}$ has no closed leaves.
\end{proof}

\subsection*{Maximal laminations and facets}
A lamination $\lambda$ is {\em maximal} if it cuts the surface into $4g-4$ ideal hyperbolic triangles.
Being sufficiently Hausdorff-close to a maximal lamination implies that the complementary subsurface $X \setminus \alpha$ should look like a union of ideal triangles. In particular, its spine should be trivalent and all of its edges should have large weight, as the following result shows.

\begin{lemma}\label{lem:Hclosetomax_largeweight}
Let $X \in \mathcal{T}_g$ and $\delta_0 = \delta_0(X) > 0$ be as in Lemma \ref{lem:unifttparam}. Consider a maximal geodesic lamination $\lambda$ on $X$. Then, for any $0 < \delta < \delta_0(X)$  and any multi-geodesic $\alpha$ on $X$ such that 
\[d_X^H(\lambda, \alpha) < \delta,\]
the spine $SC_\alpha(X)$ is trivalent and each of its edges has weight $\Omega_X(\log(1/\delta))$,
where the implicit constant is uniform as $X$ varies in the thick part of $\T_g$.
\end{lemma}
\begin{proof}
We begin by demonstrating that there is a correspondence between the vertices of $SC_\alpha(X)$ and the complementary components of $X \setminus \lambda$.
Compare with the discussion of ``invisible arc systems'' in \cite{CF2}.

Consider the regular $\delta$-neighborhood $\cN_\delta(\lambda) \subseteq X$; it may be foliated by segments of the orthogeodesic foliation of $X$ with respect to $\lambda$, all of which have length at most $O_X(\delta)$
(Lemma \ref{lem:ttwidth}).
These segments break up into finitely many isotopy classes of disjoint arcs $\{a_i\}$ running from $\alpha$ to itself, each of which has a representative of length at most $O_X(\delta)$. Taking $\delta > 0$ small enough, Lemma \ref{lem:closeimpliesarc} implies that the arc system $\{a_i\}$ is a subset of the dual arc system to the spine $SC_\alpha(X)$.

\begin{figure}[hb]
    \centering
    \includegraphics[scale=.6]{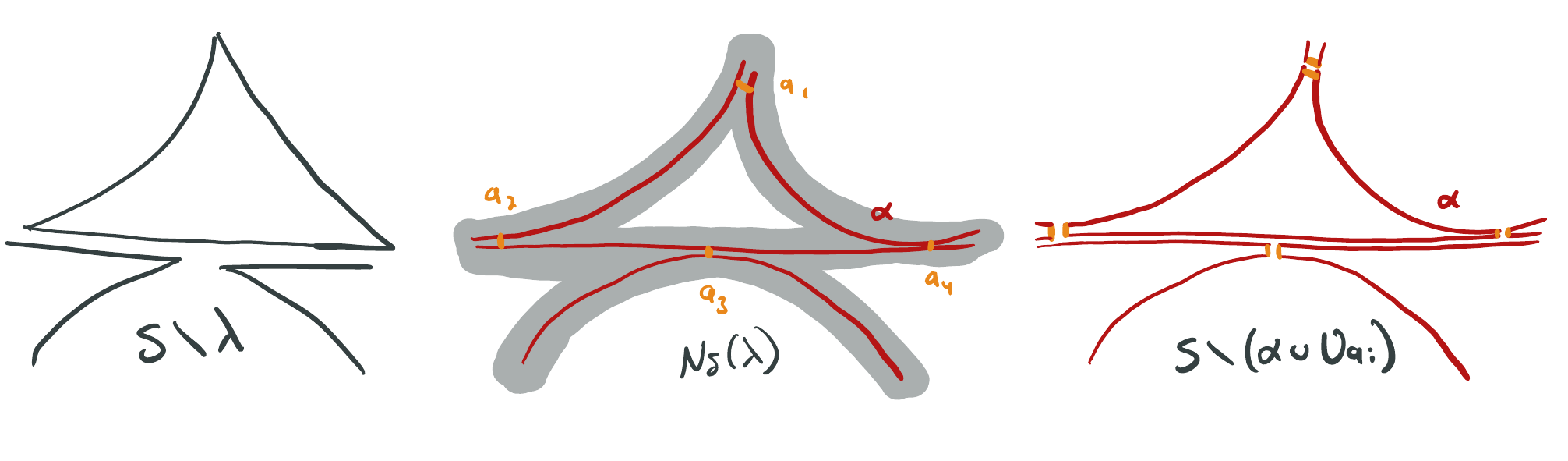}
    \caption{The correspondence between plaques of $X \setminus \lambda$ and hexagons of $X \setminus (\alpha \cup \bigcup_i a_i)$.}
    \label{fig:maxarcs}
\end{figure}

In particular, the components of $X \setminus (\alpha \cup \bigcup_i a_i)$ correspond to the components of $X \setminus \cN_\delta(\lambda)$, which in turn correspond to the plaques of $X \setminus \lambda$. See Figure \ref{fig:maxarcs}.
Since $\lambda$ is maximal, the components of $X \setminus (\alpha \cup \bigcup_i a_i)$ are all right-angled hexagons. In particular, the arc system $\{a_i\}$ is maximal, so must be the entire dual arc system to $SC_\alpha(X)$.

Now the dual orthogeodesic arcs to $SC_\alpha(X)$ are of minimal length in their isotopy class, and since there are representatives in each class of length $O_X(\delta)$, the orthogeodesic representatives are also of length $O_X(\delta)$.
Applying Lemma \ref{lem:lengthvsweight} then gives the desired lower bound on the weights of the edges of $SC_\alpha(X)$.
\end{proof}

In fact, inspection of the proof above reveals that we have actually proved a stronger statement.
Recall that a train track on a closed surface is {\em maximal} if its complementary regions are all triangles.

\begin{lemma}\label{lem:Hclosetomaxtt_largeweight}
Let $X \in \mathcal{T}_g$, $\lambda$ be a geodesic lamination on $X$ and fix $\delta > 0$. Suppose that the uniform $\delta$-neighborhood $\cN_\delta(\lambda) \subseteq X$ defines a maximal train track on $X$. Then, for any multi-geodesic $\alpha$ on $X$ such that 
\[d_X^H(\lambda, \alpha) < \delta,\]
the spine $SC_\alpha(X)$ is trivalent and each of its edges has weight $\Omega_X(\log(1/\delta))$,
where the implicit constant is uniform as $X$ varies in the thick part of $\T_g$.
\end{lemma}
\begin{proof}
The statement of Lemma \ref{lem:ttwidth} is uniform over all geodesic laminations $\lambda$, no matter the topological type, and the proof of Lemma \ref{lem:Hclosetomax_largeweight} above needs only that the components of $X \setminus \cN_\delta(\lambda)$ are triangles.
\end{proof}

\subsection*{Train tracks and lattice point counting}
In the previous paragraph we observed that the geometry of the uniform $\delta$-neighborhood $\cN_\delta(\lambda) \subseteq X$ imposes constraints on the geometry of the spine $SC_\alpha(X)$ for any multi-geodesic $\alpha$ contained within the neighborhood.
We now show that the topology of $\cN_\delta(\lambda)$ controls how many curves are contained within, or more generally carried, by it.

Before we begin, we first recall some facts about combinatorics of train tracks. The reader is directed to \cite{PH92} for a more thorough introduction to these concepts. Each switch of a train track $\tau$ cuts out a hyperplane in the space $\mathbb{R}^{\text{edges}(\tau)}$ by imposing the condition that the sum of the edge weights on one side of the switch is equal to the sum of the edge weights on the other. The intersection of all of these hyperplanes is called the {\em weight space} $W(\tau)$ of the train track. We record its dimension below; for a proof, see \cite[Section 2.1]{PH92}.

\begin{lemma}\label{lem:ttdim}
Let $\tau$ be a train track on a closed, genus $g \geq 2$ surface $S_g$. Then
\[\dim W(\tau) = -\chi(\tau) + n_0(\tau),\]
where $\chi(\tau)$ is the Euler characteristic of $\tau$ and $n_0(\tau)$ is the number of orientable components of $\tau$. In particular, if $\tau$ is not maximal, then $\dim W(\tau) \le 6g-7$.
\end{lemma}

The intersection of $W(\tau)$ with the octant $\mathbb{R}_{\ge 0}^{\text{edges}(\tau)}$ defines a finite-sided polyhedron $P(\tau)$. Any weight system $w \in P(\tau)$ corresponds to a measured geodesic lamination carried on $\tau$ and, in particular, any system of integral weights in $P(\tau)$ corresponds to an integrally weighted multi-curve carried on $\tau$. We denote by $\lambda \prec \tau$ the statement that the measured geodesic lamination $\lambda$ is carried by the train track $\tau$. This linear structure allows us to bound the number of integral multi-curves contained in a geometric train track neighborhood in terms of integral points of the weight space.

\begin{lemma}\label{lem:ttgrowthrate}
Let $X \in \mathcal{T}_g$ and $\tau$ be a train track on $X$. Then,
\[\# \{ \alpha \in \ML_g(\mathbb{Z}) \mid
\alpha  \prec \tau, \ \ell_X(\alpha) \le L \}
= O\left(L^{\dim W(\tau)}\right).\]
\end{lemma}

\begin{proof}
 Denote by $\|\cdot \|$ the Euclidean norm on $\mathbb{R}^{\text{edges}(\tau)}$.
 As the polyhedron of weight systems $P(\tau)$ is projectively compact and the hyperbolic length function $\ell_X$ is continuous on $\ML_g$, there exists a constant $C > 0$ such that for every $\lambda \prec \tau$ with weights $w \in P(\tau)$,
 \[
 \|w\| \leq C \cdot \ell_X(\lambda).
 \]
 This reduces the problem of bounding the counting function of interest to a standard lattice point count on Euclidean space. We conclude
    \[
    \# \{ \alpha \in \ML_g(\mathbb{Z}) \mid
    \alpha  \prec \tau, \ \ell_X(\alpha) \le L \}
    = O\left(L^{\dim W(\tau)}\right). \qedhere
    \]
\end{proof}

We are now ready to prove Proposition \ref{prop:asymp_triv}.
We begin by recalling that a geodesic lamination $\lambda$ is {\em approximable} if it can be arbitrarily approximated in the Hausdorff metric by simple closed multi-geodesics (on any hyperbolic surface) \cite{OP_bij}.
This definition generalizes the notion of {\em chain-recurrence,} which is equivalent to the condition of approximability by simple closed geodesics.
Importantly, the space of chain-recurrent geodesic laminations is compact with respect to the Hausdorff topology \cite[Proposition 6.2]{Thu98}, and the same proof shows that the space of approximable laminations $\mathcal{AL}$ is compact.

\begin{proof}[Proof of Proposition \ref{prop:asymp_triv}]
Using Lemma \ref{lem:Hclosetomaxtt_largeweight}, choose a constant $0 < \delta < \delta_0(X)$ so that if $\alpha$ is a multi-geodesic $\delta$-close to a maximal geodesic lamination $\lambda$ on $X$, then $SC_\alpha(X)$ is trivalent with all edge weights greater than $K$. By compactness of $\mathcal{AL}$, there exists a finite set $\{\lambda_i\}$ of (possibly non-maximal) laminations so that any approximable lamination, and in particular any multi-curve, is $\delta$-Hausdorff close to some $\lambda_i$ on $X$.

For each $i \in I$, let $\tau_i$ denote the train track obtained from $\cN_\delta(\lambda_i)$ by collapsing the leaves of $\cO_{\lambda_i}(X) |_{\cN_\delta(\lambda_i)}$.
Because the $\delta$-neighborhoods of the $\lambda_i$ cover $\mathcal{AL}$, the geometric train tracks $\tau_i$ cover $\ML_X$: given any measured lamination $\lambda$, its support is contained in one of the $\delta$-neighborhoods of the $\lambda_i$ and so the collapse map
$\cN_\delta(\lambda_i) \to \tau_i$
demonstrates that $\lambda$ is carried by $\tau_i$.

Let $J \subset I$ denote the indices of those train tracks $\tau_j$ which are not maximal; 
for every $\alpha \in F(X, \gamma, K)$, Lemma \ref{lem:Hclosetomaxtt_largeweight} implies that $\alpha$ is carried on some $\tau_j$ for $j \in J$.
Applying Lemma \ref{lem:ttgrowthrate}, we therefore get that
\begin{align*}
\# \{ \alpha \in F(X, \gamma, K) & \mid
\ell_X(\alpha) \le L \} \\
& \le 
\sum_{J} \# \{\alpha \in \ML(\mathbb{Z}) \mid
\alpha \prec \tau_j, \ \ell_X(\alpha) \le L \} \\
& = \sum_{J} O\left(L^{\dim W(\tau_j)}\right) = O\left(L^{6g-7}\right),
\end{align*}
where the last equality is a consequence of Lemma \ref{lem:ttdim}. This finishes the proof.
\end{proof}

\section{Variation of weights}\label{sec:weightsvary}

\subsection*{Outline of this section.}
We now derive uniform estimates on the geometry of $SC_\alpha(X)$ and $SC_\alpha(X')$ for all $X'$ close to $X$ and asymptotically all $\alpha$.
In particular, we show that so long $SC_\alpha(X)$ is deep in a maximal facet (an asymptotically generic assumption by Proposition \ref{prop:asymp_triv}), the spines $SC_\alpha(X)$ and $SC_\alpha(X')$ are both combinatorially (Lemma \ref{lem:edgespersist}) and geometrically (Proposition \ref{prop:weightsvary}) comparable.

Throughout this section, we work in Teichm{\"u}ller space so that the components of $X \setminus \alpha$ and $X' \setminus \alpha$ (and hence their spines) are equipped with induced markings.
This allows us to compare the combinatorics and geometry of the spines directly, not just up to the action of the mapping class group.

\subsection*{Bi-Lipschitz comparisons}
Recall that a map $f: X \to X'$ between metric spaces is said to be {\em $L$-bi-Lipschitz} if it distorts distances by a factor of at most $L$. That is, for every pair of points $x, y \in X$, we have that
\[\frac{1}{L} d_{X}(x,y)
\le
d_{X'}(f(x),f(y))
\le 
L d_X(x,y).\]
Throughout this section, we say that two marked hyperbolic surfaces $(X, f)$ and $(X', f') \in \T_g$ are {\em $\varepsilon$-bi-Lipschitz close} if there exists an $e^{\varepsilon}$-bi-Lipschitz diffeomorphism from $X$ to $X'$ in the isotopy class of $f' \circ f^{-1}$.
It is a standard (though nontrivial) fact that the topology defined by $\varepsilon$-bi-Lipschitz closeness is the same as the usual topology on $\T_g$; see \cite[pg. 268]{Th_book} as well as \cite{DE_extensions}.

Now that we have made precise what it means for $X$ and $X'$ to be close, we state precisely the main results of this section.
The first thing we must show is that for any $X'$ close to $X$, the spines $SC_\alpha(X')$ and $SC_\alpha(X)$ have the same topological type so long as one of them is deep enough in a maximal facet.

\begin{lemma}\label{lem:edgespersist}
For every $X \in \T_g$ and every multi-curve $\gamma$, there exist constants $K_1 = K_1(X) > 0$ and $\varepsilon_0 = \varepsilon_0(X) > 0$ so that 
for any multi-curve 
$\alpha \in \Mod_g \cdot \gamma \setminus F(\gamma, X, K_1)$ and any $X'$ that is $\varepsilon_0$ bi-Lipschitz close to $X$,
then the metric ribbon graphs $SC_\alpha(X)$ and $SC_\alpha(X')$ have the same topological type.
\end{lemma}

Together with Proposition \ref{prop:asymp_triv}, this implies that for all but an asymptotically trivial proportion of the multi-curves on $X$, there is a correspondence between the edges of the ribbon graph spines $SC_\alpha(X)$ and $SC_\alpha(X')$. In particular, we can compare their weights.

With this in mind, the main result of this section is the following:

\begin{proposition}\label{prop:weightsvary}
For every $X \in \T_g$ there exists $K_2 = K_2(X) > K_1 > 0$ such that 
for any $X'$ that is $\varepsilon < \varepsilon_0$ bi-Lipschitz close to $X$ and any
$\alpha \in \Mod_g \cdot \gamma \setminus F(X, \gamma, K_2)$, the following holds.
For every edge $e$ of the spine $SC_\alpha(X)$,
\[e^{-\varepsilon} |e|_X - K_2 \le |e|_{X'} \le e^\varepsilon |e|_X + K_2\]
\end{proposition}

\begin{remark}
In fact, the constants $K_2$ and $\varepsilon_0$ in Proposition \ref{prop:weightsvary} can be shown to depend only on the thickness of $X$.
\end{remark}

Throughout this section, we use $\alpha_X$ to denote the geodesic realization of a simple closed curve $\alpha$ with respect to the hyperbolic metric $X$.

\begin{remark}
Stronger, quantitative statements are also true, but in order to state them we would need to be careful about adjacency of facets of the moduli space of ribbon graphs and be much more particular about the geometry of $\alpha$ on $X$.
Compare \cite{CF2}.
Since we will not need such detailed results, we content ourselves with the ``soft'' methods and coarse estimates recorded in this section.
\end{remark}

\subsection*{Geodesics map near geodesics}
We first show how we can leverage the fact that the geometry of $X$ is comparable with that of $X'$ to show that the way $\alpha_X$ wraps around $X$ is comparable to how $\alpha_{X'}$ wraps around $X'$.

Fix an $e^{\varepsilon}$-bi-Lipschitz map $f: X \to X'$.
By the Morse Lemma (see Lemma \ref{Morse Lemma}), the geodesic $\alpha_X$ is sent some bounded distance away from $\alpha_{X'}$.
We need finer control on this distance, so we prove a version of the Morse Lemma that allows us to ensure that $f(\alpha_X)$ and $\alpha_{X'}$ are arbitrarily close.

\begin{proposition}\label{prop:geos_to_geos}
For any small enough $\delta>0$ there exists an $\varepsilon = \varepsilon(\delta)>0$ so that for any $e^\varepsilon$-bi-Lipschitz map $f: \mathbb{H}^2 \rightarrow \mathbb{H}^2$ and any geodesic $g \subset \mathbb{H}^2$, we have
\[d_{\mathbb{H}^2}^H( f(g), g') \le \delta\]
where $g'$ denotes the geodesic with the same endpoints as $f(g)$ and $d_{\mathbb{H}^2}^H$ denotes the Hausdorff distance between closed sets in $\mathbb{H}^2$.
\end{proposition}

This Proposition is an immediate consequence of the equicontinuity of $e^{\varepsilon}$-Lipschitz maps and the Arzela--Ascoli theorem.
In the interest of the overall flow of the paper, we have deferred a formal proof to an \hyperref[appendix]{Appendix}.

Let $\varepsilon = \varepsilon(\delta) > 0$ be as in Proposition \ref{prop:geos_to_geos}. For later use, we define the function
\begin{equation}
\delta(\varepsilon) :=
\inf\{ \delta \mid \varepsilon \le \varepsilon(\delta)\}.
\end{equation}
Proposition \ref{prop:geos_to_geos} can then be rephrased as stating that every $e^{\varepsilon}$-bi-Lipschitz map takes every geodesic $g$ within $\delta(\varepsilon)$ of the geodesic $g'$ with the same endpoints as $f(g)$, and that $\delta(\varepsilon) \to 0$ as $\varepsilon \to 0$.

Assuming the statement above we can prove Lemma \ref{lem:edgespersist}: since the $X$-weight of every edge of $\alpha$ is large, each of its dual arcs on $X$ is short. Since $X$ and $X'$ are bi-Lipschitz equivalent these arcs remain short on $X'$, so they must also appear as dual arcs for $\alpha$ on $X'$.

\begin{proof}[Proof of Lemma \ref{lem:edgespersist}]
Suppose that $e$ is an edge of $SC_\alpha(X)$ whose dual arc $a_e$ joins leaves $\ell_1$ and $\ell_2$ of $\widetilde{\alpha}_X \subset \widetilde{X}$.
By Lemma \ref{lem:lengthvsweight}, we know that its dual arc $a_e$ has length comparable to $e^{-|e|_X}$ so long as the weight $|e|_X$ is large enough.

Let $\ell_1'$ and $\ell_2'$ denote the corresponding leaves of $\widetilde{\alpha}_{X'} \subset \widetilde{X}'$ and let $f:X \to X'$ be a $e^{\varepsilon}$-bi-Lipschitz map.
By Proposition \ref{prop:geos_to_geos}, we know that 
$f(\ell_i)$ is $\delta(\varepsilon)$ close to $\ell_i'$ for $i=1,2$. In particular, this implies that there is a point of $\ell_1'$ that is $\delta$ close to $f(a_e)$, and the same for $\ell_2'$.

We can therefore build a path from $\ell_1'$ to $\ell_2'$ by way of $f(a_e)$ that, by Proposition \ref{prop:geos_to_geos} and the bi-Lipschitz quality of $f$, has length at most
\[2\delta(\varepsilon) + e^{\varepsilon -|e|_X}.\]
Forcing $\varepsilon$ to be small enough (smaller than some $\varepsilon_0$) and $|e|_X$ to be large enough (larger than some $K_1(X)$), we can ensure that this quantity is smaller than the universal cutoff $\log \sqrt{3}$. 
Lemma \ref{lem:closeimpliesarc} then implies that the leaves $\ell_1'$ and $\ell_2'$ are connected by an arc of $\mathcal{O}_\alpha(X)$ which is necessarily in the same isotopy class as $a_e$ rel $\widetilde{\alpha}_X$.
\end{proof}

\begin{remark}
Careful inspection of the proof above reveals that if we take $\varepsilon_0$ smaller, we can also take $K_1$ smaller (though at a certain point, this breaks down because the estimate from Lemma \ref{lem:lengthvsweight} does not work for small edge weights; compare also \cite[Lemma 6.7]{CF}).
This reflects the fact that as one takes $X'$ closer and closer to $X$ in Teichm{\"u}ller space, more and more curves look similar on the two surfaces.
\end{remark}

\subsection*{Centers map near centers}
While we were able to give a rough estimate of the length of the dual arcs of $SC_\alpha(X')$ in terms of $\varepsilon$ and the lengths of the corresponding arcs on $X$ above, passing this through Lemma \ref{lem:lengthvsweight} exponentiates the error in our estimate, resulting in bounds on the weights of the edges of $SC_\alpha(X')$ that are much weaker than what we want.

Instead, to prove Proposition \ref{prop:weightsvary} we shift our focus from the dual arcs to $SC_\alpha(X)$ and the thin parts of $X \setminus \alpha_X$ to the vertices and thick parts.
Our aim is to show that not only does $f$ map $\alpha_X$ near $\alpha_{X'}$, but it also takes vertices of $SC_\alpha(X)$ near the corresponding vertices of $SC_\alpha(X')$. 

We first record a bound, uniform in both $X$ and $\alpha$, on the distance between vertices of $SC_\alpha(X)$ and the geodesic $\alpha_X$.

\begin{lemma}\label{lem:upperbdribs}
For any $s>0$, there is a constant $R=R(s)>0$ so that for any $s$-thick $X \in \T_g$ and any multi-curve $\alpha$, any vertex $u$ of $SC_\alpha(X)$ is at most $R(s)$ away from $\alpha_X$.
\end{lemma}
\begin{proof}
We first prove this for a fixed $X$, then use Proposition \ref{prop:geos_to_geos} to bootstrap up to a uniform estimate on the thick part.

Fix $X \in \mathcal{T}_g$ and fix some small $\zeta \in (0, \log \sqrt[4]{3})$.
Consider the cover of the space of approximable laminations $\mathcal{AL}$ by radius $\zeta$ neighborhoods in the Hausdorff distance on $X$.
By compactness, there is a finite set $\{\lambda_i\}$ of approximable laminations so that any $\lambda \in \mathcal{AL}$ (and in particular, any multi-curve $\alpha$) has Hausdorff distance at most $\zeta$ from some $\lambda_i$ on $X$.
Let $R_X$ denote the largest radius of any circle inscribed in any of the $\lambda_i$; we note that $R_X$ is finite because the maximal size circles are all centered at vertices of the spines $SC_{\lambda_i}(X)$ and there are finitely many of these (even if $X \setminus \lambda_i$ is a crowned hyperbolic surface).

Now let $\alpha$ be any multi-curve (or more generally, any approximable lamination), let $u$ be a vertex of $SC_\alpha(X)$, and let $r$ be the distance from $u$ to $\alpha_X$ (equivalently, the radius of the inscribed circle centered at $u$).
Note that $r \ge \log \sqrt{3}$, since $\log \sqrt{3}$ is the radius of the circle inscribed in an ideal triangle (compare Lemma \ref{lem:closeimpliesarc}).

By our choice of cover, there is some $\lambda_i$ which is $\zeta$-close to $\alpha$, and so the $r - 2\zeta >0$ ball centered at $u$ does not meet $\lambda_i$; if it did, this would give a path from $u$ to $\alpha$ of length at most $r-\zeta$, contradicting the definition of $r$.
In particular, the $r - 2\zeta$ ball centered at $u$ is contained inside some ball inscribed in $\lambda_i$, so
\[r < R_X + 2\zeta.\]
Since $\alpha$ and $u$ were arbitrary, this completes the proof for our fixed $X$.

To upgrade this statement to hold over the entire thick part, consider an $X'$ that is $\varepsilon$-bi-Lipschitz close to $X$ and fix an $e^{\varepsilon}$-bi-Lipschitz diffeomorphism $f:X' \to X$.
For any ball $B$ embedded in $X' \setminus \alpha_{X'}$ of radius $r$, its image $f(B)$ contains a ball of radius at least $e^{-\varepsilon} r$ that is disjoint from $f(\alpha_{X'})$.
Now by Proposition \ref{prop:geos_to_geos}, this implies $f(B)$ contains a ball of radius $e^{-\varepsilon} r - \delta(\varepsilon)$ disjoint from $\alpha_X$.
In particular, since the balls centered at vertices of the spine have the maximal radius among all balls embedded in $X \setminus \alpha$, this implies that
\[e^{-\varepsilon} r - \delta(\varepsilon) < R_X + 2 \zeta\]
using our bound from above. Hence we get a uniform bound for $r$ over the entire $\varepsilon$-bi-Lipschitz neighborhood of $X$.

We may therefore cover the $s$-thick part of $\M_g$ with finitely many $\varepsilon$-bi-Lipschitz balls and run this argument for each.
The maximum constant we get thus bounds the distance from any vertex of $SC_\alpha(X)$ to $\alpha$ for any $\alpha$ and any $s$-thick $X$.
\end{proof}

We can now prove that $f$ maps vertices of $SC_\alpha(X)$ near those of $SC_\alpha(X')$ by showing a general statement for geodesics in $\mathbb{H}^2$.

Given a triple of pairwise disjoint geodesics in $\mathbb{H}^2$, none of which separates the others, there is a unique inscribed circle.
The {\em center} of the triple is the center of this circle, and the {\em basepoints} of the triple are the points of tangency of this circle with each geodesic.

\begin{lemma}\label{lem:centers_to_centers}
Let $f: \mathbb{H}^2 \rightarrow \mathbb{H}^2$ be an $e^\varepsilon$-bi-Lipschitz map for some $\varepsilon > 0$.
Suppose that $(g_1, g_2, g_3)$ is a triple of pairwise disjoint geodesics, none of which separates the other two, and let $u$ denote the center of the triple.
Set $r$ to be the distance from $u$ to any of the $g_i$, equivalently, the radius of the inscribed circle.

For each $i$, let $g_i'$ denote the geodesic with the same endpoints as the quasigeodesic $f(g_i)$.
Then there exists a constant $D = D(r, \varepsilon)$ so that $f(u)$ is at most $D$ away from the center $u'$ of the triple $(g_1', g_2', g_3')$.
\end{lemma}
\begin{proof}
Because $f$ is $e^\varepsilon$-bi-Lipschitz, we have that $f(u)$ is at most $r e^\varepsilon$ away from each $f(g_i)$.
Invoking Proposition \ref{prop:geos_to_geos}, this implies that it has distance at most
$r e^\varepsilon + \delta(\varepsilon)$
from each $g_i'$.
This in turn implies that $u'$ is at least that close to each $g_i'$, since $u'$ is the center of the triple.

The desired statement is then an immediate corollary of the following:
\begin{claim}
For any $C > 0$, there is a $D(C)$ so that
for any triple of pairwise disjoint geodesics $(g_1, g_2, g_3)$ in $\mathbb{H}^2$, none of which separates the others,
the diameter of 
\begin{equation}\label{eq:tripleint}
\cN_C(g_1) \cap \cN_C(g_2) \cap \cN_C(g_3)
\end{equation}
is at most $D(C)$, where $\cN_C(g_i)$ denotes the regular $C$-neighborhood of $g_i$.
\end{claim}

\begin{figure}
    \centering
    \includegraphics[scale=.5]{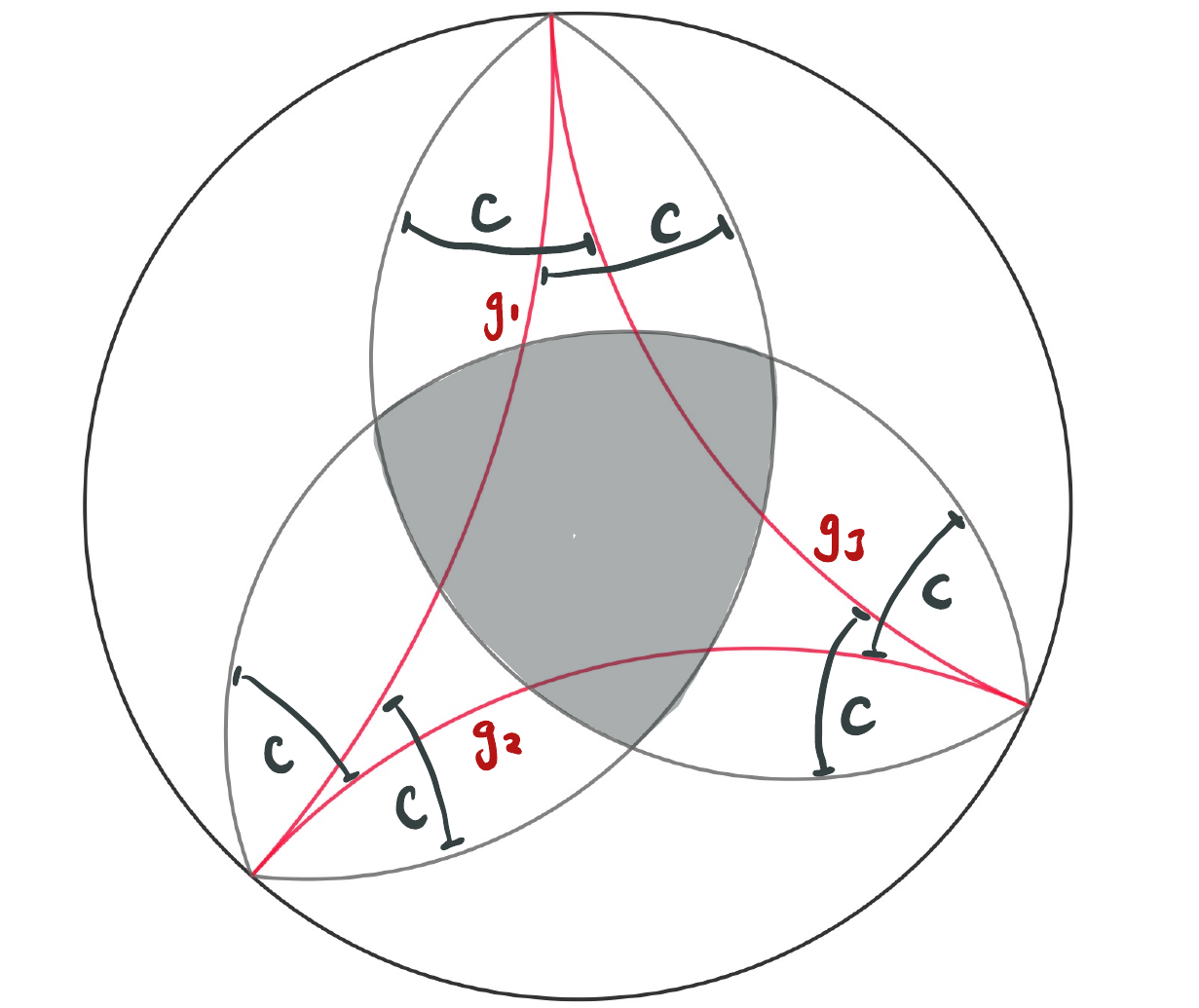}
    \caption{The intersection of three hypercyclic neighborhoods.}
    \label{fig:triple_hypercycle}
\end{figure}

\begin{proof}[Proof of Claim]
For any $C$, the diameter of the triple intersection \eqref{eq:tripleint} is maximized when $(g_1, g_2, g_3)$ forms an ideal triangle.
In this case, inspection reveals that \eqref{eq:tripleint} is compact, and so has bounded diameter. 
Compare Figure \ref{fig:triple_hypercycle}.
\end{proof}

Since $f(u)$ and $u'$ are both contained within a
$r e^\varepsilon + \delta(\varepsilon)$
neighborhood of the $g_i'$, the claim gives us a bound of
\[
D(r e^\varepsilon + \delta(\varepsilon))
\]
(which in particular depends only on $r$ and $\varepsilon$) on the distance between them.
\end{proof}

Vertices of the spine $SC_\alpha(X)$ may be identified with centers of tuples of lifts of $\alpha_X$ to $\smash{\widetilde{X}}$; in particular, the distance from the vertex to $\alpha_X$ is the radius of the inscribed circle.
The uniform bound of Lemma \ref{lem:upperbdribs} on $r$ therefore gives a uniform bound on the $D$ guaranteed by Lemma \ref{lem:centers_to_centers}, hence any $\varepsilon$-bi-Lipschitz map $f:X \to X'$ sends vertices of $SC_\alpha(X)$ uniformly (in $\alpha$) near those of $SC_{\alpha}(X')$.

We are now ready to prove that an $e^\varepsilon$-bi-Lipschitz map scales the edge weights of $SC_\alpha(X)$ by at most $e^{\varepsilon}$, up to a uniform additive error.
The main idea of our proof is that since vertices of $SC_\alpha(X)$ map near vertices of $SC_\alpha(X')$, basepoints on $\alpha_X$ map near basepoints on $\alpha_{X'}$.

\begin{proof}[Proof of Proposition \ref{prop:weightsvary}]
We begin by noting that if $X$ is $s$-thick and $X'$ is $e^\varepsilon$-bi-Lipschitz close to $X$, then $X'$ is $e^{-\varepsilon} s$-thick.

Now for any $\alpha \in F(X, \gamma, K_1)$, we know by Lemma \ref{lem:edgespersist} that the spines $SC_\alpha(X)$ and $SC_\alpha(X')$ have the same combinatorial type.
Let $e$ be any edge of $SC_\alpha(X)$ and choose a lift to the universal cover. Let $\ell$ be either of the lifts of $\alpha$ which meet the dual arc to $e$, and denote by $p$ and $q$ the basepoints on $\ell$ corresponding to the endpoints $u$ and $v$ of $e$.
By definition, the weight $|e|_X$ is the length of the subsegment of $\ell$ that runs from $p$ to $q$.
Since $SC_\alpha(X)$ and $SC_\alpha(X')$ have the same combinatorial type, we can specify a corresponding edge $e'$ of $SC_\alpha(X')$ with endpoints $u'$ and $v'$, a lift $\ell'$ of $\alpha_{X'}$, and basepoints $p'$ and $q'$ on $\ell'$.
The weight $|e|_{X'}$ is similarly the distance between $p'$ and $q'$ along $\ell'$.
See Figure \ref{fig:weightestimate}.

\begin{figure}[ht]
    \centering
    \includegraphics[scale=.65]{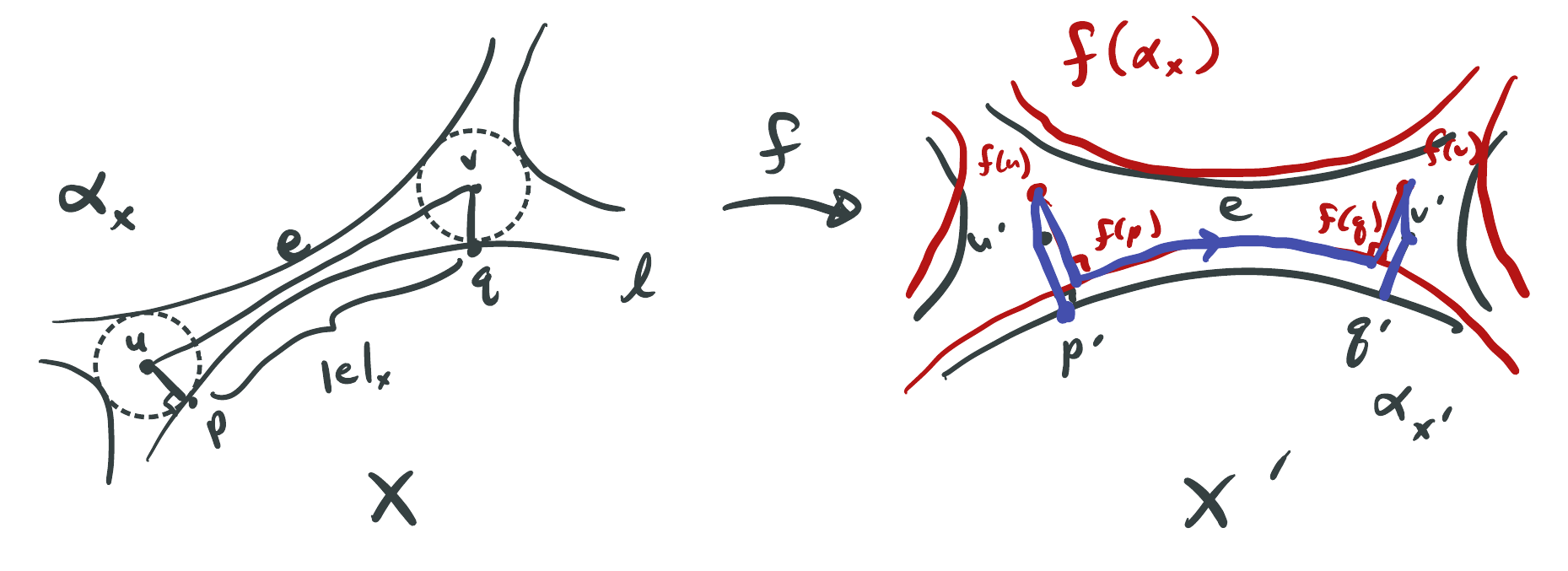}
    \caption{Edges, centers, and basepoints under a bi-Lipschitz map. The path used to derive the estimate of Proposition \ref{prop:weightsvary} is highlighted in the right-hand figure.}
    \label{fig:weightestimate}
\end{figure}

We claim that we can additively bound the distance from $f(p)$ to $p'$. Indeed, by Lemma \ref{lem:upperbdribs} the basepoint $p$ has distance at most $R(s)$ to the center $u$ and likewise $p'$ is at most $R({e^{-\varepsilon}s})$ away from $u'$.
Now by Lemma \ref{lem:centers_to_centers}, $f(u)$ is $D(R(s), \varepsilon)$ away from $u'$, hence we can build a path
$f(p) \to f(u) \to u' \to p'$
with total length at most
\[E(\varepsilon, s) := e^{\varepsilon} R(s) + D(R(s), \varepsilon) + R(e^{-\varepsilon}s).\]
By the same logic, the distance between $f(q)$ and $q'$ is at most $E(\varepsilon, s)$.

Therefore, by traveling from $p'$ to $f(p)$, then from $f(p)$ along $f(\ell)$ to $f(q)$, and finally from $f(q)$ to $q'$, we can build a path from $p'$ to $q'$.
The bi-Lipschitz quality of $f$ plus our bound from the previous equation together imply that
\[|e|_{X'} = d_{X'}(p', q') 
\le e^{\varepsilon}|e|_X + 2E(\varepsilon, s).\]
To prove the reverse direction of the inequality, we note that the argument above depends only on the bi-Lipschitz constant and the thickness of $X$ and $X'$. Thus we can repeat the steps with the roles reversed, arriving at the other desired inequality (with an additive error of $2e^{-\varepsilon} E(\varepsilon, s)$, in fact).

We now observe that for any $\varepsilon < \varepsilon_0$, since every $\varepsilon$-bi-Lipschitz map is in particular $\varepsilon_0$-bi-Lipschitz, and every $e^{-\varepsilon_0}s$-thick surface is $e^{-\varepsilon}s$-thick, we have that
\[E(\varepsilon, s) \le E(\varepsilon_0, s).\]
Setting $K_2 = \max( K_1, 2E(\varepsilon_0, s) )$ therefore yields the desired result.
\end{proof}

\section{Equidistribution of expanding RSC-horoballs}\label{sec:RSCequi}

\subsection*{Outline of this section.} In this section we show that expanding horoballs on moduli spaces of hyperbolic surfaces defined in terms of the cut-spine-rescale construction equidistribute with respect to the Mirzakhani measure. The proof makes crucial use of the results in \cite{Ara19b} with some important modifications highlighted throughout. We begin with a brief review of some aspects of the ergodic theory of the earthquake flow following Mirzakhani \cite{Mir08a}.

In this section and the next, we restrict exclusively to the case when $\gamma$ is a non-separating simple closed curve; the general case is discussed in Section \ref{sec:general}.

\subsection*{The Mirzakhani measure.} For the rest of this section fix an integer $g \geq 2$ and a connected, oriented, closed surface $S_g$ of genus $g$. Consider the bundle $\mathcal{P}^1 \mathcal{T}_{g}$ of unit length measured geodesic laminations over $\mathcal{T}_{g}$. More precisely,
\[
\mathcal{P}^1 \mathcal{T}_{g} := \{(X,\lambda) \in \mathcal{T}_{g} \times \mathcal{ML}_{g} \ | \ \ell_{\lambda}(X) = 1\}.
\]
Recall that $\mu_\mathrm{Thu}$ denotes the Thurston measure on $\mathcal{ML}_g$, as introduced in \S 2. For every $X \in \mathcal{T}_{g}$ consider the measure $\mu_{\text{Thu}}^X$ on the fiber $\mathcal{P}^1_X\mathcal{T}_{g}$ of  $\mathcal{P}^1\mathcal{T}_{g}$ above $X$ which to every Borel measurable subset $A \subseteq \mathcal{P}^1_X\mathcal{T}_{g}$ assigns the value
\begin{equation*}
\mu_{\text{Thu}}^X(A) := \mu_{\text{Thu}}([0,1] \cdot A).
\end{equation*}
The \textit{Mirzakhani measure} $\nu_{\text{Mir}}$ on $\mathcal{P}^1\mathcal{T}_{g}$ is defined by the disintegration formula
\[
d\nu_{\text{Mir}}(X,\lambda) := d\mu_{\text{Thu}}^X(\lambda) \thinspace  d\mu_{\text{wp}}(X).
\]

The mapping class group $\text{Mod}_{g}$ acts diagonally on $\mathcal{P}^1 \mathcal{T}_{g}$ in a properly discontinuous way preserving $\nu_{\text{Mir}}$. The quotient $\mathcal{P}^1 \mathcal{M}_{g} := \mathcal{P}^1\mathcal{T}_{g}/\text{Mod}_{g}$ is the bundle of unit length measured geodesic laminations over $\mathcal{M}_{g}$. Locally pushing forward $\nu_{\text{Mir}}$ through the quotient map $\mathcal{P}^1\mathcal{T}_{g} \to \mathcal{P}^1\mathcal{M}_{g}$ yields a measure $\widehat{\nu}_{\text{Mir}}$ on $\mathcal{P}^1\mathcal{M}_{g}$, also called the {Mirzakhani measure}. The pushforward of $\widehat{\nu}_{\text{Mir}}$ under the bundle map $\pi \colon \mathcal{P}^1\mathcal{M}_{g} \to \mathcal{M}_{g}$ is given by
\[
d\pi_*(\widehat{\nu}_{\text{Mir}})(X) = B(X) \thinspace d\widehat{\mu}_\text{wp}(X),
\]
where $B \colon \mathcal{M}_{g} \to \mathbf{R}_+$ is the \textit{Mirzakhani function} defined for every $X \in \mathcal{M}_g$ by
\[
B(X) := \mu_\mathrm{Thu}\left(\left\lbrace \lambda \in \mathcal{ML}_g \ | \ \ell_\lambda(X) \leq 1 \right\rbrace \right) = \mu_{\text{Thu}}^X(\mathcal{P}^1_X\mathcal{T}_{g}).
\]
By work of Mirzakhani \cite[Theorem 3.3]{Mir08b}, the function $B \colon \mathcal{M}_g \to \mathbf{R}_+$ is integrable with respect to the Weil-Petersson measure $\widehat{\mu}_\mathrm{wp}$. We denote the total mass of $\mathcal{P}^1 \mathcal{M}_{g}$ with respect to $\widehat{\nu}_{\text{Mir}}$ by
\[
b_g := \widehat{\nu}_{\text{Mir}}(\mathcal{P}^1 \mathcal{M}_{g}) = \int_{\mathcal{M}_{g}} B(X) \thinspace d\widehat{\mu}_{\text{wp}}(X);
\]
observe that this is one of the constants appearing in Mirzakhani's asymptotic counting formula \eqref{eq:mir}.

\subsection*{The earthquake flow.}Given a marked hyperbolic structure $X \in \mathcal{T}_{g}$, a measured geodesic lamination $\lambda \in \mathcal{ML}_{g}$, and a real number $t \in \mathbf{R}$, denote by $\text{tw}_\lambda^t(X) \in \mathcal{T}_{g}$ the time $t$ \textit{twist deformation} (or \textit{earthquake deformation}) of $X$ along $\lambda$; see \S 2 in \cite{Ker83} for a definition.
By work of Wolpert \cite{Wol83}, twist deformations are the Hamiltonian flows of hyperbolic length functions with respect to the Weil-Petersson symplectic form $\omega_{\text{wp}}$ on $\mathcal{T}_{g}$. In particular, the Weil-Petersson symplectic form $\omega_{\text{wp}}$, the Weil-Petersson volume form $v_\text{wp}$, and the Weil-Petersson measure $\mu_\text{wp}$ on $\mathcal{T}_{g}$ are invariant under twist deformations.

The \textit{earthquake flow} $\{\text{tw}^t \colon \mathcal{P}^1\mathcal{T}_{g} \to \mathcal{P}^1\mathcal{T}_{g}\}_{t \in \mathbf{R}}$ on the bundle $\mathcal{P}^1 \mathcal{T}_{g}$ of unit length measured geodesic laminations over Teichmüller space is defined for every $t \in \mathbf{R}$ and every $(X,\lambda) \in \mathcal{P}^1 \mathcal{T}_{g}$ by
\[
\text{tw}^t(X,\lambda) := (\text{tw}^t_\lambda(X),\lambda) \in \mathcal{P}^1 \mathcal{T}_{g}.
\]
The earthquake flow preserves the Mirzakhani measure $\nu_\mathrm{Mir}$ on $\mathcal{P}^1\mathcal{T}_g$. As the mapping class group action on $\mathcal{P}^1 \mathcal{T}_{g}$ commutes with the earthquake flow, the bundle $\mathcal{P}^1 \mathcal{M}_{g} = \mathcal{P}^1 \mathcal{T}_{g,n}/\text{Mod}_{g}$ of unit length measured geodesic laminations over moduli space also carries an {earthquake flow}
$\{\text{tw}^t \colon \mathcal{P}^1\mathcal{M}_{g} \allowbreak \to \mathcal{P}^1\mathcal{M}_{g}\}_{t \in \mathbf{R}}$ which preserves the Mirzakhani measure $\widehat{\nu}_\mathrm{Mir}$.

The following result of Mirzakhani is fundamental for our approach.

\begin{theorem}
    \label{theo:ergodic}
    \cite[Corollary 1.2]{Mir08a}
    The earthquake flow on the bundle $\mathcal{P}^1\mathcal{M}_g$ is ergodic with respect to the Mirzakhani measure $\widehat{\nu}_\mathrm{Mir}$.
\end{theorem}

\subsection*{RSC-horoball measures.}
We now introduce the main construction of this section.
Let $\gamma$ be a non-separating simple closed curve on $S_g$ and $h \colon \mathcal{MRG}_{g-1,2}(1,1) \to \mathbf{R}_{\geq 0}$ be a non-zero, continuous, compactly supported function. For every $L > 0$ consider the {\em RSC-horoball measure} $\mu_{\gamma,h}^L$ on $\mathcal{T}_g$ given by
\begin{equation}
\label{eq:RSC_meas}
d\mu_{\gamma,h}^L(X) := \mathbbm{1}_{[0,L]}(\ell_\gamma(X)) \thinspace h(\mathrm{RSC}_\gamma(X)) \thinspace d\mu_\mathrm{wp}(X),
\end{equation}
i.e., this is the measure obtained by restricting the Weil-Petersson measure to the horoball of depth $L$ at $\gamma$ and weighting it by $h$. This measure is $\mathrm{Stab}(\gamma) \subseteq \mathrm{Mod}_g$ invariant. Consider the sequence of covers
\[
\mathcal{T}_g \to \mathcal{T}_g/\mathrm{Stab}(\gamma) \to \mathcal{M}_g.
\]
Let $\widetilde{\mu}_{\gamma,h}^L$ be the local pushforward of $\mu_{\gamma,h}^L$ to $\mathcal{T}_g/\mathrm{Stab}(\gamma)$ and $\widehat{\mu}_{\gamma,h}^L$ be the pushworward of $\widetilde{\mu}_{\gamma,h}^L$ to $\mathcal{M}_g$.

These measures can be lifted to the corresponding bundles of unit length measured geodesic laminations in the following way. On $\mathcal{P}^1\mathcal{T}_g$ consider the measures $\nu_{\gamma,h}^L$ given by
\[
d\nu_{\gamma,h}^L(X,\lambda) :=  d \delta_{\gamma/\ell_{\gamma}(X)}(\lambda) \thinspace d\mu_{\gamma,h}^L(X).
\]
This measure is also $\mathrm{Stab}(\gamma) \subseteq \mathrm{Mod}_g$ invariant. Let $\widetilde{\nu}_{\gamma,h}^L$ be the local pushforward of $\nu_{\gamma,h}^L$ to $\mathcal{P}^1\mathcal{T}_g/\mathrm{Stab}(\gamma)$ and $\widehat{\nu}_{\gamma,h}^L$ be the pushforward of $\widetilde{\nu}_{\gamma,h}^L$ to $\mathcal{P}^1\mathcal{M}_g$. Denote by $m_{\gamma,h}^L$ the total mass of these measures:
\[m_{\gamma,h}^L := \widehat{\nu}_{\gamma,h}^L(\mathcal{P}^1\mathcal{M}_g) = \widehat{\mu}_{\gamma,h}^L(\mathcal{M}_g).\]

\subsection*{Equidistribution of expanding RSC-horoballs.} The following theorem is the main result of this section; its proof will occupy the bulk of this section.

\begin{theorem}
    \label{theo:horoball_equid_1}
    Let $\gamma$ be a non-separating simple closed curve on $S_g$ and consider a non-zero, continuous, compactly supported function $h \colon \mathcal{MRG}_{g-1,2}(1,1) \to \mathbf{R}_{\geq 0}$. Then, with respect to the weak-$\star$ topology for measures on $\mathcal{P}^1\mathcal{M}_g$,
    \[
    \lim_{L \to \infty} \frac{\widehat{\nu}_{\gamma,h}^L}{m_{\gamma,h}^L} = \frac{\widehat{\nu}_\mathrm{Mir}}{b_g}.
    \]
\end{theorem}

Throughout the rest of the section, $\gamma$ and $h$ will be fixed as in the statement of Theorem \ref{theo:horoball_equid_1}. Taking pushforwards to $\M_g$ in Theorem \ref{theo:horoball_equid_1} we deduce the following important consequence.

\begin{corollary}
    \label{cor:equid}
    With all notation as in Theorem \ref{theo:horoball_equid_1},
    \[
    \lim_{L \to \infty} \frac{\widehat{\mu}_{\gamma,h}^L}{m_{\gamma,h}^L} = \frac{B(X) \thinspace d\widehat{\mu}_{\mathrm{wp}}(X)}{b_g}
    \]
    with respect to the weak-$\star$ topology for measures on $\mathcal{M}_g$.
\end{corollary}

To prove Theorem \ref{theo:horoball_equid_1} we study the limit points of the sequence of probability measures $(\widehat{\nu}_{\gamma,h}^L/m_{\gamma,h}^L)_{L>0}$ on $\mathcal{P}^1\mathcal{M}_g$. More concretely, we show that every such limit point must be earthquake flow invariant, absolutely continuous with respect to the Mirzakhani measure, and a probability measure. The desired conclusion follows by the ergodicity of the Mirzakhani measure with respect to the earthquake flow and the Banach-Alaoglu theorem.

\subsection*{Total mass.} Our first step towards proving Theorem \ref{theo:horoball_equid_1} will be to compute an asymptotic formula for the total mass $m_{\gamma,h}^L$ of the measures $\widehat{\nu}_{\gamma,h}^L$ as $L \to \infty$.
The main observation is that horoballs are fibered by horospheres, and Theorem \ref{theo:Do} states that the Weil-Petersson measure on deeper horospheres converges to the Kontsevich measure.
As most of the mass of the horoball is concentrated on the deeper horospheres, we get the following formula.

\begin{proposition}
    \label{prop:total_mass}
    The following limit holds:
    \[
    \lim_{L \to \infty} \frac{m_{\gamma,h}^L}{L^{6g-6}} = \frac{1}{12g-12} \int_{\mathcal{MRG}_{g-1,2}(1,1)} h(x) \thinspace d\eta_{\mathrm{Kon}}(x),
    \]
where $\eta_\mathrm{Kon}$ is the Kontsevich measure on $\mathcal{MRG}_{g-1,2}(1,1)$.
\end{proposition}

\begin{proof}
    Recall that $\mathrm{Stab}(\gamma) \subseteq \mathrm{Mod}_g$ denotes the stabilizer of $\gamma$.
    Denote by $\widetilde{\mu}_\mathrm{wp}^\gamma$ the local pushforward of the Weil-Petersson measure $\mu_\mathrm{wp}$ on $\mathcal{T}_g$ to $\mathcal{T}_g/\mathrm{Stab}(\gamma)$.
    Recall
    \[
    d\mu_{\gamma,h}^L(X) := \mathbbm{1}_{[0,L]}(\ell_\gamma(X)) \thinspace h(\mathrm{RSC}_\gamma(X)) \thinspace d\mu_\mathrm{wp}(X).
    \]
    By taking local pushforwards to $\mathcal{T}_g/\mathrm{Stab}(\gamma)$ we deduce
    \begin{equation}\label{eqn:RSC/Stab}
    d\widetilde{\mu}_{\gamma,h}^L(X) = \mathbbm{1}_{[0,L]}(\ell_\gamma(X)) \thinspace h(\mathrm{RSC}_\gamma(X)) \thinspace d\widetilde{\mu}^\gamma_\mathrm{wp}(X).
    \end{equation}
    It follows that we can write
    \begin{align*}
        m_{\gamma,h}^L &:=\widehat{\mu}_{\gamma,h}^L(\mathcal{M}_g) \\
        &\phantom{:}=\widetilde{\mu}_{\gamma,h}^L(\mathcal{T}_g/\mathrm{Stab}(\gamma))\\
        &\phantom{:}= \int_{\mathcal{T}_g/\mathrm{Stab}(\gamma)} \mathbbm{1}_{[0,L]}(\ell_\gamma(X)) \thinspace h(\mathrm{RSC}_\gamma(X)) \thinspace d\widetilde{\mu}^\gamma_\mathrm{wp}(X)\\
        &\phantom{:}= \int_{\mathcal{T}_g/\mathrm{Stab}(\gamma)} \mathbbm{1}_{\left[0,\sqrt{L}\right]}(\ell_\gamma(X)) \thinspace h(\mathrm{RSC}_\gamma(X)) \thinspace d\widetilde{\mu}^\gamma_\mathrm{wp}(X)\\
        &\phantom{:=} + \int_{\mathcal{T}_g/\mathrm{Stab}(\gamma)} \mathbbm{1}_{\left(\sqrt{L},L\right]}(\ell_\gamma(X)) \thinspace h(\mathrm{RSC}_\gamma(X)) \thinspace d\widetilde{\mu}^\gamma_\mathrm{wp}(X).
    \end{align*}
    This decomposition into two integrals will be helpful because it will allow us to take limits in the second integral, whose lower limit will converge to $\infty$, while the first integral will remain bounded in terms of $L$ in an explicit way, with a power saving with respect to the leading term.
    
    For the first term, a direct application of Mirzakhani's integration formula (see \cite[Theorem 4.1]{Mir08b} and also \S7) shows that
    \begin{align*}
        &\left| \int_{\mathcal{T}_g/\mathrm{Stab}(\gamma)} \mathbbm{1}_{\left[0,\sqrt{L}\right]}(\ell_\gamma(X)) \thinspace h(\mathrm{RSC}_\gamma(X)) \thinspace d\widetilde{\mu}^\gamma_\mathrm{wp}(X) \right| \\
        &\leq \|h\|_\infty \thinspace \int_{\mathcal{T}_g/\mathrm{Stab}(\gamma)} \mathbbm{1}_{\left[0,\sqrt{L}\right]}(\ell_\gamma(X)) \thinspace d\widetilde{\mu}^\gamma_\mathrm{wp}(X) \\
        &= \frac{\|h\|_\infty}{2} \int_0^{\sqrt{L}} \ell \cdot V_{g-1,2}(\ell,\ell) \thinspace d\ell\\
        &= \|h\|_\infty \cdot O(L^{3g-3}),
    \end{align*}
    where $V_{g-1,2}(\ell,\ell)$ denotes the total Weil-Petersson volume of $\mathcal{M}_{g-1,2}(\ell,\ell)$; the estimate follows because is $V_{g-1,2}(\ell,\ell)$ a polynomial of degree $6g-8$ (see \cite[Theorem 1.1]{Mir07c} as well as the discussion in Section \ref{sec:prelim}). 
    
    For the second term, we recall that $\widehat{\mu}_{\mathrm{wp}}^\ell$ denotes the Weil-Petersson measure on $\mathcal{M}_{g-1,2}(\ell,\ell)$. Applying Mirzakhani's integration formula again, we compute
   \begin{align*}
       &\int_{\mathcal{T}_g/\mathrm{Stab}(\gamma)} \mathbbm{1}_{\left(\sqrt{L},L\right]}(\ell_\gamma(X)) \thinspace h(\mathrm{RSC}_\gamma(X)) \thinspace d\widetilde{\mu}^\gamma_\mathrm{wp}(X)\\
       &= \frac{1}{2} \int_{\sqrt{L}}^{L}  \int_{\mathcal{M}_{g-1,2}(\ell,\ell)} \int_0^\ell h(\mathrm{RS}(Y)) \thinspace d\tau \thinspace d\widehat{\mu}_{\mathrm{wp}}^\ell (Y) \thinspace d\ell\\
       &= \frac{1}{2} \int_{\sqrt{L}}^{L} \ell \cdot \int_{\mathcal{M}_{g-1,2}(\ell,\ell)} h(\mathrm{RS}(Y)) \thinspace d\widehat{\mu}_{\mathrm{wp}}^\ell (Y) \thinspace d\ell\\
       &= \frac{1}{2} \int_{\sqrt{L}}^{L} \ell \cdot \int_{\mathcal{MRG}_{g-1,2}(1,1)} h(x) \thinspace d \mathrm{RS}_*\widehat{\mu}_{\mathrm{wp}}^\ell (x) \thinspace d\ell \\
       &= \frac{1}{2} \int_{\sqrt{L}}^{L} \ell^{6g-7} \cdot \int_{\mathcal{MRG}_{g-1,2}(1,1)} h(x) \thinspace \frac{d \mathrm{RS}_*\widehat{\mu}_{\mathrm{wp}}^\ell(x)}{\ell^{6g-8}} \thinspace d\ell.
   \end{align*}
   By Theorem \ref{theo:Do},
   \[
   \lim_{\ell \to \infty} \int_{\mathcal{MRG}_{g-1,2}(1,1)} h(x) \thinspace \frac{d \mathrm{RS}_*\widehat{\mu}_{\mathrm{wp}}^\ell(x)}{\ell^{6g-8}} = \int_{\mathcal{MRG}_{g,1,2}(1,1)} h(x) \thinspace d\eta_\mathrm{Kon}(x).
   \]
   It follows that
   \begin{gather*}
   \lim_{L \to \infty} \frac{1}{L^{6g-6}} \cdot \frac{1}{2} \int_{\sqrt{L}}^{L} \ell^{6g-7} \cdot \int_{\mathcal{MRG}_{g-1,2}(1,1)} h(x) \thinspace \frac{d \mathrm{RS}_*\widehat{\mu}_{\mathrm{wp}}^\ell(x)}{\ell^{6g-8}} \thinspace d\ell\\ = \frac{1}{12g-12}\int_{\mathcal{MRG}_{g,1,2}(1,1)} h(x) \thinspace d\eta_\mathrm{Kon}(x).
   \end{gather*}
   From the arguments above we conclude
    \[
    \lim_{L \to \infty} \frac{m_{\gamma,h}^L}{L^{6g-6}} = \frac{1}{12g-12} \int_{\mathcal{MRG}_{g-1,2}(1,1)} h(x) \thinspace d\eta_{\mathrm{Kon}}(x). \qedhere
    \]
\end{proof}

\subsection*{Earthquake flow invariance}
Notice that applying the earthquake flow to the lifted measures $\smash{\nu_{\gamma,h}^L}$ on $\mathcal{P}^1\mathcal{T}_g$ corresponds to applying twist deformations to the base measures $\smash{\mu_{\gamma,h}^L}$ on $\mathcal{T}_g$ along $\gamma$. These measures are clearly invariant under such twists. As a consequence, the lifted measures $\smash{\nu_{\gamma,h}^L}$ on $\mathcal{P}^1\mathcal{T}_g$ are earthquake-flow invariant, a property which descends to the measures $\smash{\widehat{\nu}_{\gamma,h}^L}$ on the quotient $\smash{\mathcal{P}^1\mathcal{M}_g}$.
In particular, we deduce the following result; compare to \cite[Proposition 3.2]{Ara19b}.

\begin{lemma}
    \label{lem:invariance}
    Every weak-$\star$ limit point of $(\smash{\widehat{\nu}_{\gamma,h}^L})_{L > 0}$ is earthquake flow--invariant.
\end{lemma}

\subsection*{Absolute continuity.} Our next goal is to show that every weak-$\star$ limit point of the sequence of measures $(\widehat{\nu}_{\gamma,h}^L)_{L>0}$ on $\mathcal{P}^1\mathcal{M}_g$ is absolutely continuous with respect to the Mirzakhani measure $\widehat{\nu}_\mathrm{Mir}$; this is the most technical part of the proof, and will require us to compare RSC-horoball measures to the usual horoball measures considered in \cite{Ara19b}; see Lemma \ref{lemma:hor_meas_comp}. To this end we first introduce some extra notation and review some of the results from \cite{Ara19b}.

For every $L > 0$ the horoball measure $\mu_{\gamma}^L$ on $\mathcal{T}_g$ corresponds to the RSC-horoball measure weighted by the constant function $h \equiv 1$. Let $\widetilde{\mu}_{\gamma}^L$ be the local pushforward of $\mu_{\gamma}^L$ to $\mathcal{T}_g/\mathrm{Stab}(\gamma)$ and $\widehat{\mu}_{\gamma}^L$ be the pushworward of $\widetilde{\mu}_{\gamma}^L$ to $\mathcal{M}_g$. These measures can be lifted to the corresponding bundles of unit length measures geodesic laminations in the usual way. Let $\nu_\gamma^L$ be the corresponding lift to $\mathcal{P}^1\mathcal{T}_g$, $\widetilde{\nu}_{\gamma}^L$ be the local pushforward of $\nu_{\gamma}^L$ to $\mathcal{P}^1\mathcal{T}_g/\mathrm{Stab}(\gamma)$, and $\widehat{\nu}_{\gamma}^L$ be the pushforward of $\widetilde{\nu}_{\gamma}^L$ to $P^1\mathcal{M}_g$.
Let $m_{\gamma}^L := \widehat{\nu}_{\gamma}^L(\mathcal{P}^1\mathcal{M}_g) = \widehat{\mu}_{\gamma}^L(\mathcal{M}_g)$ denote the total mass of these measures.

Recall that $\eta_\mathrm{Kon}$ denotes the Kontsevich measure on $\mathcal{MRG}_{g-1,2}(1,1)$ and that $c_g := \eta_\mathrm{Kon}(\mathcal{MRG}_{g-1,2}(1,1))> 0$ denotes its total mass.
Setting $h \equiv 1$ in the proof of Proposition \ref{prop:total_mass} yields the following identity relating the mass of horoballs in $\M_{g}$ with the Kontsevich volume of the moduli space $\MRG_{g-1, 2}(1,1)$; alternatively, this identity can be deduced from (\ref{eqn:Konts_int}) and the corresponding equation for the Weil-Petersson symplectic form.

\begin{proposition}
\label{prop:total_mass_2}
    Let $\gamma$ be a non-separating simple closed curve on $S_g$. Then,
    \[
    \lim_{L \to \infty} \frac{m_\gamma^L}{L^{6g-6}} = \frac{c_g}{12g-12}.
    \]
\end{proposition}

An alternative computation in \cite{Ara19b} relates the mass of horoballs in $\M_{g}$ with the relative frequency of non-separating simple closed curves among all geodesics:

\begin{proposition}\cite[Proposition 3.1]{Ara19b}
   \label{prop:total_mass_2_AH}
    Let $\gamma$ be a non-separating simple closed curve on $S_g$. Then, the following limit holds:
    \[
    \lim_{L \to \infty} \frac{m_\gamma^L}{L^{6g-6}} = c(\gamma).
    \]
\end{proposition}

In particular, we deduce the following identity, which will be useful later on.

\begin{corollary}
    \label{cor:identity}
    Let $\gamma$ be a non-separating simple closed curve on $S_g$. Then,
    \[
    c(\gamma) = \frac{c_g}{12g-12}.
    \]
\end{corollary}

The following bound will allow us to apply technical results from \cite{Ara19b} without the need of reproving them in our context.

\begin{lemma}
	\label{lemma:hor_meas_comp}
	There exists a constant $C = C(h) > 0$ such that for every Borel measurable subset $A \subseteq \mathcal{P}^1 \mathcal{M}_{g}$, the following bound holds:
	\[
	\limsup_{L \to \infty} \frac{\widehat{\nu}_{\gamma,h}^{L}(A)}{m_{\gamma,h}^{L}} \leq C \cdot \limsup_{L \to \infty} \frac{\widehat{\nu}_{\gamma}^{L}(A)}{m_{\gamma}^{L}}.
	\]
\end{lemma}

\begin{proof}
	It follows directly from the definitions that
	\[
	\widehat{\nu}_{\gamma,h}^{L} \leq \|h\|_\infty \cdot \widehat{\nu}_{\gamma}^{L}.
	\]
	As a consequence of Propositions \ref{prop:total_mass} and \ref{prop:total_mass_2} there exists $c > 0$ such that
	\[
	\lim_{L \to \infty} \frac{m_{\gamma,h}^{L}}{m_\gamma^{L}} = c.
	\]
	It follows that, for every Borel measurable subset $A \subseteq P^1\mathcal{M}_{g}$,
	\[
	\limsup_{L \to \infty} \frac{\widehat{\nu}_{\gamma,h}^{L}(A)}{m_{\gamma,h}^{L}} \leq \|h\|_\infty \cdot 	\limsup_{L \to \infty} \frac{\widehat{\nu}_{\gamma}^{L}(A)}{m_{\gamma,h}^{L}} \leq \|h\|_\infty \cdot c \cdot 	\limsup_{L \to \infty} \frac{\widehat{\nu}_{\gamma,}^{L}(A)}{m_{\gamma}^{L}}. \qedhere
	\]
\end{proof}

Denote by $d_\mathrm{Thu}$ the symmetric Thurston metric on $\mathcal{T}_g$; the precise definition of this metric will not be important in what follows but the standard reference is \cite{Thu98}. Denote by $U_X(\varepsilon) \subseteq \mathcal{T}_{g}$ the open ball of radius $\varepsilon > 0$ centered at $X \in \mathcal{T}_{g}$ with respect to $d_\mathrm{Thu}$. Denote by $\Pi \colon \mathcal{P}^1\mathcal{T}_g \to \mathcal{P}^1\mathcal{M}_g$ the natural quotient map. 

Denote by $\mathcal{PML}_g$ the space of projective measured geodesic laminations on $S_g$ and by $[\lambda] \in \mathcal{P}\mathcal{ML}_{g}$ the projective class of $\lambda \in \mathcal{ML}_g$. A Borel measurable subset $V \subseteq \mathcal{P}\mathcal{ML}_{g}$ is said to be a continuity subset of the Thurston measure class if
\[
\mu_{\mathrm{Thu}}(\{\lambda \in \mathcal{ML}_{g} \ | \ [\lambda] \in \partial V \}) = 0.
\]

The following estimate is the main technical tool that will be used in the ensuing discussion; the purpose of this estimate is to compare horoball measures to the Mirzakhani measure.

\begin{proposition}
    \cite[Proposition 3.5]{Ara19b}
	\label{prop:ac_key_estimate}
	Let $K \subseteq \mathcal{T}_{g}$ be a compact subset and $\varepsilon_0 > 0$ be fixed. Then, there exists a constant $C = C(K,\varepsilon_0) > 0$ such that for every $X \in K$, every $0 < \varepsilon < \varepsilon_0$, and every open continuity subset $V \subseteq P \mathcal{ML}_{g}$ of the Thurston measure class, the following estimate holds,
	\[
	\limsup_{L \to \infty} \frac{\widehat{\nu}_{\gamma}^{L}(\Pi(U_X(\varepsilon) \times V))}{m_\gamma^{L}} \leq C \cdot \nu_{\mathrm{Mir}}(U_X(\varepsilon) \times V).
	\]
\end{proposition}

The following lemma allows us to control the behavior of sets of zero $\nu_\mathrm{Mir}$ measure in terms of sets whose horoball measure is controlled by Proposition \ref{prop:ac_key_estimate}.

\begin{lemma}\cite[Lemma 3.4]{Ara19b}
	\label{lemma:approx}
	Let $K \subseteq \mathcal{P}^1\mathcal{T}_{g}$ be a compact subset. Then, there exists a constant $\varepsilon_0 = \varepsilon_0(K) > 0$ such that for every Borel measurable subset $A \subseteq K$ with $\nu_{\mathrm{Mir}}(A) = 0$ and every $\delta> 0$, there exists a countable cover $\{W_i\}_{i\in \mathbb{N}}$ of $A$ such that each $W_i$ is a product set
	\[W_i = U_{X_i}(\varepsilon_i) \times V_i\]
	where each $X_i \in K$, each $\varepsilon_i \in (0, \varepsilon_0)$, and each $V_i \subseteq \mathcal{P}\mathcal{ML}_{g}$ is an open continuity subset of the Thurston measure class. Moreover, this cover may be taken such that
	\[	\sum_{i\in \mathbf{N}} \nu_{\mathrm{Mir}}(W_i)< \delta.\]
\end{lemma}

We are now ready to study the absolute continuity of limit measures.

\begin{proposition}
    \label{prop:abs_cont}
    Every weak-$\star$ limit of the sequence of measures $(\smash{\widehat{\nu}_{\gamma,h}^L}/ m_{\gamma,h}^{L}
    )_{L > 0}$ on $\mathcal{P}^1\mathcal{M}_g$ is absolutely continuous with respect to the Mirzakhani measure  $\smash{\widehat{\nu}_\mathrm{Mir}}$.
\end{proposition}

\begin{proof}
	Let $\widehat{\nu}$ be some weak-$\star$ limit point of the sequence and let $L_j \nearrow +\infty$ be an increasing sequence of positive real numbers such that
	\[
	\lim_{j \to \infty} \frac{\widehat{\nu}_{\gamma,h}^{L_j}}{m_{\gamma,h}^{L_j}} = \widehat{\nu}.
	\]
	Let $\smash{\widehat{A}} \subseteq \mathcal{P}^1 \mathcal{M}_{g}$ be a Borel measurable subset such that $\widehat{\nu}_{\mathrm{Mir}}(\smash{\widehat{A}})=0$.
	Our goal is to show that $\widehat{\nu}(\smash{\widehat{A}}) = 0$.
	As $\mathcal{P}^1 \mathcal{M}_{g}$ admits a countable exhaustion by compact sets and as the limit measure $\widehat{\nu}$ is continuous with respect to increasing limits of sets, it suffices to consider those $\smash{\widehat{A}} \subseteq \smash{\widehat{K}}$ for some compact subset $\smash{\widehat{K}} \subseteq \mathcal{P}^1 \mathcal{M}_{g}$.
	Let $K \subseteq \mathcal{P}^1 \mathcal{T}_{g}$ be a compact subset covering $\smash{\widehat{K}}$ and $A \subseteq K$ be a subset covering $\smash{\widehat{A}}$. Notice that $\nu_{\text{Mir}}(A) = 0$.\\
	
	Let $\delta > 0$ be arbitrary and consider the countable cover $\{W_i\}_{i\in \mathbb{N}}$ of $A$ guaranteed by Lemma \ref{lemma:approx}.
	The monotonicity property of measures ensures that
	\begin{equation}\label{eqn:mon}
	\widehat{\nu}(\widehat{A}) \leq \widehat{\nu} \left(\bigcup_{i\in \mathbb{N}} \Pi(W_i)\right).
	\end{equation}
	On the other hand, as the limit measure $\widehat{\nu}$ is finite and continuous with respect to increasing limits of sets, we can find a finite subset $I \subseteq\mathbb{N}$ such that
	\begin{equation}\label{eqn:coverdeltaerror}
	\widehat{\nu} \left(\bigcup_{i\in \mathbb{N}} \Pi(W_i)\right)
	\leq \delta + \widehat{\nu} \left(\bigcup_{i\in I} \Pi(W_i)\right)
	\leq \delta +  \sum_{i\in I} {\widehat{\nu}}\left( \Pi(W_i)\right) 
	\end{equation}
	where the second inequality follows from the subadditivity property of measures.
	
	Now Portmanteau's theorem applied to open sets with compact closure ensures
	\begin{equation}\label{eqn:coverPort}
	\widehat{\nu} \left( \Pi(W_i)\right)  \leq \liminf_{j \to \infty} \ \frac{\widehat{\nu}_{\gamma,h}^{L_j}}{m_\gamma^{f,L_j}} \left(\Pi(W_i)\right)
	\end{equation}
	for every $i \in I$.
	Lemma \ref{lemma:hor_meas_comp} bounds the right-hand side in terms of the usual horoball measures $\widehat{\nu}_{\gamma}^{L}$, and thanks to the structure of our choice of cover, we know by Proposition \ref{prop:ac_key_estimate} that the $\widehat{\nu}_{\gamma}^{L}/m_\gamma^L$ measures of the $W_i$ sets are small.
	More concretely, taken together these statements provide a constant $C > 0$ such that
	\[
	\limsup_{j \to \infty} \frac{\widehat{\nu}_{\gamma,h}^{L_j}}{m_{\gamma,h}^{L_j}} \left( \Pi(W_i )\right) \leq C \cdot \nu_{\mathrm{Mir}}(W_i)
	\]
	for every $i \in I$. In particular, we have that
	\begin{equation}\label{eqn:smallmeaslimit}
	\limsup_{j \to \infty} \sum_{i\in I} \frac{\widehat{\nu}_{\gamma,h}^{L_j}}{m_{\gamma,h}^{L_j}} \left( \Pi(W_i)\right)  \leq C \cdot \sum_{i \in I} \nu_{\mathrm{Mir}}(W_i) \le C \cdot \delta,
	\end{equation}
	where the last inequality follows from our choice of the cover $\{W_i\}_{i \in \mathbb{N}}$ in Lemma \ref{lemma:approx}.
	Putting together \eqref{eqn:mon}, \eqref{eqn:coverdeltaerror}, \eqref{eqn:coverPort}, and \eqref{eqn:smallmeaslimit}, we get that
	\[
	\widehat{\nu}(\widehat{A}) \leq (1 + C) \cdot \delta.
	\]
	As $\delta  > 0$ was arbitrary, we see that $\widehat{\nu}(\widehat{A})  = 0$, completing the proof.
\end{proof}

\subsection*{No escape of mass.} We now show that every weak-$\star$ limit point of the measures $\widehat{\nu}_{\gamma,h}^L/m_{\gamma, h}^L$ on $\mathcal{P}^1\mathcal{M}_g$ is a probability measure; this does not follow automatically because the space $\PoM_g$ is not compact.

For every $s > 0$ denote by $\mathcal{K}_s \subseteq \mathcal{M}_{g}$ the $s$-thick part of moduli space and by $\mathcal{P}^1 \mathcal{K}_s \subseteq \mathcal{P}^1\mathcal{M}_{g}$ the natural lift of this subset to the bundle of unit length measured geodesic laminations over moduli space.
By Mumford's compactness criterion, these subsets are both compact.

We begin by recording the desired statement for the measures $\widehat{\nu}_{\gamma}^L$.

\begin{proposition}\cite[Proposition 3.9]{Ara19b}
    \label{prop:nem_1}
    For every $\delta > 0$ there exists an $s > 0$ such that the following bound holds:
    \[
    \liminf_{L \to \infty} \frac{\widehat{\nu}_{\gamma}^L(\mathcal{P}^1\mathcal{K}_s)}{m_{\gamma}^L} \geq 1 - \delta.
    \] 
\end{proposition}

Invoking the comparison result of Lemma \ref{lemma:hor_meas_comp} allows us to quickly deduce the corresponding statement for RSC-horoballs.

\begin{corollary}
    \label{cor:nem_2}
    For every $\delta > 0$ there exists an $s > 0$ such that
    \[
    \liminf_{L \to \infty} \frac{\widehat{\nu}_{\gamma,h}^L(\mathcal{P}^1\mathcal{K}_s)}{m_{\gamma,h}^L} \geq 1 - \delta.
    \]
\end{corollary}

\begin{proof}
    Let $C = C(h) > 0$ be as in Lemma \ref{lemma:hor_meas_comp}. Fix $\delta > 0$. Proposition \ref{prop:nem_1} ensures there exists $s > 0$ such that
    \[
    \liminf_{L \to \infty} \frac{\widehat{\nu}_{\gamma}^L(\mathcal{P}^1\mathcal{K}_s)}{m_{\gamma}^L} \geq 1 - \delta/C.
    \] 
    Taking complements we deduce
    \[
    \limsup_{L \to \infty} \frac{\widehat{\nu}_{\gamma}^L(\mathcal{P}^1\mathcal{M}_g \setminus \mathcal{P}^1\mathcal{K}_s)}{m_{\gamma}^L} \leq \delta/C.
    \]
    It follows from Lemma \ref{lemma:hor_meas_comp} that
    \[
        \limsup_{L \to \infty} \frac{\widehat{\nu}_{\gamma,h}^L(\mathcal{P}^1\mathcal{M}_g \setminus \mathcal{P}^1\mathcal{K}_s)}{m_{\gamma,h}^L} \leq C \cdot \limsup_{L \to \infty} \frac{\widehat{\nu}_{\gamma}^L(\mathcal{P}^1\mathcal{M}_g \setminus \mathcal{P}^1\mathcal{K}_s)}{m_{\gamma,h}^L} \leq \delta.
    \]
    Taking complements we conclude
    \[
    \liminf_{L \to \infty} \frac{\widehat{\nu}_{\gamma,h}^L(\mathcal{P}^1\mathcal{K}_s)}{m_{\gamma,h}^L} \geq 1 - \delta. \qedhere
    \]
\end{proof}

As such, we see that no mass can escape out the cusp of $\PoM_g$, hence standard arguments allow us to conclude the following.

\begin{corollary}\label{cor:prob}
Every weak-$\star$ limit of $(\widehat{\nu}_{\gamma,h}^L/ m_{\gamma}^L)_{L > 0}$ is a probability measure.
\end{corollary}

\subsection*{Endgame.} We are now ready to prove Theorem \ref{theo:horoball_equid_1}.

\begin{proof}[Proof of Theorem \ref{theo:horoball_equid_1}.]
    The Banach-Alaoglu theorem ensures that the sequence of measures $\smash{(\widehat{\nu}_{\gamma,h}^L/m_{\gamma}^L)_{L>0}}$ has a weak-$\star$ limit point in the space of regular measures on $\mathcal{P}^1\mathcal{M}_g$. By Lemma \ref{lem:invariance}, Proposition \ref{prop:abs_cont}, and Corollary \ref{cor:prob}, every such limit point is invariant under the earthquake flow, absolutely continuous with respect to $\widehat{\nu}_\mathrm{Mir}$, and a probability measure. As the earthquake flow is ergodic  with respect to $\widehat{\nu}_\mathrm{Mir}$ (Theorem \ref{theo:ergodic}), every such limit point must be equal to $\widehat{\nu}_\mathrm{Mir}/b_g$.
\end{proof}

\section{Equidistribution of complementary subsurfaces: the first case}\label{sec:avg_unfold}

\subsection*{Outline of this section.} In this section we prove Theorem \ref{theo:main}, the main result of this paper in the special case of non-separating simple closed curves.
We begin by reducing the original equidistribution question to a counting problem for metric ribbon graphs. Using work of Luo (Theorem \ref{thm:Luo}), this is equivalent to a counting problem for surfaces in a certain region of moduli space.
Averaging and unfolding techniques allow us to further reduce this counting problem to an equidistribution question for RSC-horoball measures;
Propositions \ref{prop:asymp_triv} and \ref{prop:weightsvary} will play an important role at this stage of the proof.
The results of \S\ref{sec:RSCequi} (which rely on the ergodicity of the earthquake flow) guarantee these measures equidistribute.
The relationship between the total mass of RSC-horoball measures and the Kontsevich measure (explained in Proposition \ref{prop:total_mass}) plays a key role in the averaging and unfolding argument.

\subsection*{Reducing equidistribution to counting.}
For the rest of this section fix an integer $g \geq 2$ and a connected, oriented, closed surface $S_g$ of genus $g$. Let $\gamma$ be a non-separating simple closed curve on $S_g$ and fix a marked hyperbolic structure $X \in \mathcal{T}_g$.
Recall that for every $L > 0$ we are considering the counting function
\[
s(X,\gamma,L) := \# \{\alpha \in \mathrm{Mod}_g \cdot \gamma \ | \ \ell_\alpha(X) \leq L \}
\]
which does not depend on the marking of $X \in \mathcal{T}_g$ but only on its underlying hyperbolic structure $X \in \mathcal{M}_g$.
For the convenience of the reader, we restate Mirzakhani's asymptotic count:
\begin{equation}
\label{eq:mir}
\lim_{L \to \infty} \frac{s(X,\gamma,L)}{L^{6g-6}} = \frac{c(\gamma) \cdot B(X)}{b_g}
\end{equation}
where the constants $c(\gamma) > 0$, $B(X) > 0$, and $b_g > 0$ were introduced in \S1 (and further discussed in the beginning of Section \ref{sec:RSCequi}).

We recall that on $\mathcal{MRG}_{g-1,2}(1,1)$ we want to study the asymptotic distribution of the counting measure
\[
\eta_{X,\gamma}^L := \sum_{\alpha \in \mathrm{Mod}_g \cdot \gamma} \mathbbm{1}_{[0,L]}(\ell_\alpha(X)) \cdot \delta_{\mathrm{RSC}_\alpha(X)}.
\]
Recall that $\eta_\mathrm{Kon}$ denotes the Kontsevich measure on $\mathcal{MRG}_{g-1,2}(1,1)$ and $c_g := \eta_\mathrm{Kon}(\mathcal{MRG}_{g-1,2}(1,1))> 0$ its total mass.
For the convenience of the reader we restate Theorem \ref{theo:main} here; this is the main result of this paper for non-separating simple closed curves and its proof will occupy the bulk of this section.

\begin{theorem}
\label{theo:main_2}
Let $\gamma$ be a non-separating simple closed curve on $S_g$ and $X \in \mathcal{M}_g$. Then, with respect to the weak-$\star$ topology for measures on $\mathcal{MRG}_{g-1,2}(1,1)$,
\[
\lim_{L \to \infty} \frac{\eta_{X,\gamma}^L}{s(X,\gamma,L)} = \frac{\eta_{\mathrm{Kon}}}{c_g}.
\]
\end{theorem}

As explained in the introduction, to prove Theorem \ref{theo:main_2} we first reduce to a counting problem for metric ribbon graphs. More concretely, it is enough to show that for every non-zero, non-negative, continuous, compactly supported $f \colon \mathcal{MRG}_{g-1,2}(1,1) \to \mathbf{R}_{\geq 0}$,
\begin{gather*}
\lim_{L \to \infty} \frac{1}{s(X,\gamma,L)} \int_{\mathcal{MRG}_{g-1,2}(1,1)} f(x) \thinspace d\eta_{X,\gamma}^L(x) \\
= \frac{1}{c_g}\int_{\mathcal{MRG}_{g-1,2}(1,1)} f(x) \thinspace d\eta_\mathrm{Kon}(x).
\end{gather*}

For the rest of this section we fix such a function $f$.
For every $L > 0$ consider the $f$-weighted counting function
\begin{align}
\label{eq:count}
c(X,\gamma,f,L) &:= \int_{\mathcal{MRG}_{g-1,2}(1,1)} f(x) \thinspace d\eta_{X,\gamma}^L(x)\\
&\phantom{:}= \sum_{\alpha \in \mathrm{Mod}_g \cdot \gamma} \mathbbm{1}_{[0,L]}(\ell_\alpha(X)) \cdot f(\mathrm{RSC}_\alpha(X)). \nonumber
\end{align}
The rest of this section is devoted to proving the ensuing result, from which Theorem \ref{theo:main_2} follows directly by the above discussion. This appears as Theorem \ref{theo:red_intro} in the introduction.

\begin{theorem}
\label{theo:red}
With all notation as above,
\[
\lim_{L \to \infty} \frac{c(X,\gamma,f,L)}{s(X,\gamma,L)} = \frac{1}{c_g} \int_{\mathcal{MRG}_{g-1,2}(1,1)} f(x) \thinspace d\eta_\mathrm{Kon}(x).
\]
\end{theorem}

\subsection*{Averaging counts.} Our next goal is to average the counting functions introduced in (\ref{eq:count}) over small neighborhoods of moduli space. Using the results from \S \ref{sec:weightsvary}, we first study how these counting functions vary in such neighborhoods. 

Recall that Proposition \ref{prop:weightsvary} states that for any $\varepsilon$-bi-Lipschitz close hyperbolic structures $X, Y \in \T_g$ and any $\alpha$ so that all of the edges of $SC_\alpha(X)$ are long, the edges of $SC_\alpha(X)$ and $SC_\alpha(Y)$ differ by a multiplicative and additive constant. We now define analogous neighborhoods in the moduli space of ribbon graphs.

Given $\mrg \in \mathcal{MRG}_{g-1,2}(1,1)$ and positive constants $\varepsilon, \delta > 0$, denote by $N_{\varepsilon,\delta}(\mrg)$ the set of all $\mrgy \in \mathcal{MRG}_{g-1,2}(1,1)$ in the same facet as $\mrg$, i.e., with the same topological type underlying ribbon graph as $\mrg$, such that for every edge $e$ of $\mrg$ and $\mrgy$,
\[
e^{-\varepsilon} \cdot |e|_\mrg - \delta \leq |e|_\mrgy \leq e^\varepsilon \cdot |e|_\mrg + \delta
\]
(here we have implicitly fixed a local marking so we can compare the weights of specific edges).
For every $\varepsilon > 0$ and $\delta > 0$ consider the averaged functions
\[
f_{\varepsilon,\delta}^{\min}, f_{\varepsilon,\delta}^{\max} \colon \mathcal{MRG}_{g-1,2}(1,1) \to \mathbf{R}_{\geq 0}
\]
given by
\[
f_{\varepsilon,\delta}^{\min}(\mrg) :=\min_{\mrgy \in N_{\varepsilon,\delta}(\mrg)} f(\mrgy),\quad
f_{\varepsilon,\delta}^{\max}(\mrg) := \max_{\mrgy \in N_{\varepsilon,\delta}(\mrg)} f(\mrgy).
\]

Denote by $o_X(L^{6g-6})$ any function $o_X \colon \mathbf{R}_{>0} \to \mathbf{R}$ depending on $X \in \mathcal{M}_g$ with
\[
\lim_{L \to \infty} \frac{o_X(L)}{L^{6g-6}} = 0.
\]

We begin our proof of Theorem \ref{theo:red} with the following estimate, which allows us to compare counts of curves on $X$ with counts on nearby surfaces.
Proposition \ref{prop:asymp_triv}, Lemma \ref{lem:edgespersist}, and Proposition \ref{prop:weightsvary} play a crucial role in the proof of this result.
 
\begin{proposition}
\label{prop:comp}
There exists $K = K(X) > 0$ such that for every non-zero, continuous, compactly supported function $f \colon \mathcal{MRG}_{g-1,2}(1,1) \to \mathbf{R}_{\geq 0}$, every $\varepsilon \in (0,1)$, every $Y \in \mathcal{T}_g$ that is $e^\varepsilon$-bi-Lipschitz close to $X$, and every $L > 0$,
\begin{gather}
    c\left(Y,\gamma,f_{2\varepsilon,3K/\sqrt{L}}^{\min},e^{-\varepsilon} L\right) + \| f \|_\infty \cdot o_{X}\left(L^{6g-6}\right) \leq c(X,\gamma,f,L), \label{eq:bd_1}\\
    c(X,\gamma,f,L) \leq c\left(Y,\gamma,f_{2\varepsilon,3K/\sqrt{L}}^{\max},e^\varepsilon L\right) + \|f\|_\infty \cdot o_{X}\left(L^{6g-6}\right). \label{eq:bd_2}
\end{gather}
\end{proposition}
 
\begin{proof}
We prove (\ref{eq:bd_2}). A proof of (\ref{eq:bd_1}) can be obtained following similar arguments.

We first isolate the bottom of the count by length of curves in the mapping class group orbit of $\gamma$, i.e., we fix $L > 0$ and consider the truncated counting function
\[
c^\dagger(X,\gamma,f,L) := \sum_{\alpha \in \mathrm{Mod}_g \cdot \gamma} \mathbbm{1}_{[\sqrt{L},L]}(\ell_\alpha(X)) \cdot f(\mathrm{RSC}_\alpha(X)).
\]
The asymptotic estimate in (\ref{eq:mir}) implies
\begin{equation}
\label{eq:a_1}
c(X,\gamma,f,L) = c^\dagger(X,\gamma,f,L) + \|f\|_\infty \cdot o_X\left(L^{6g-6}\right).
\end{equation}

We now apply Proposition \ref{prop:asymp_triv} to further focus our attention on curves in the mapping class group orbit of $\gamma$ whose complements are deep in maximal facets of $\MRG_{g-1,2}(1,1)$.
Let $s > 0$ be such that all surfaces $Y \in \mathcal{M}_g$ that are $e$-bi-Lipschitz close to $X$ are $s$-thick and set $K = K_2(s) > 0$ to be the constant coming from Proposition \ref{prop:weightsvary}. 

Recall from Section \ref{sec:asymptriv} that for $K > 0,$ the set of $\alpha \in \Mod_g \cdot \gamma$ for which $SC_\alpha(X)$ lies in a $K$-neighborhood of the lower-dimensional facet of $\mathbb{R}_{>0} \cdot \MRG_{g-1,2}(1,1)$ is denoted by $F(X,\gamma,K)$.
Consider now the further truncated counting function
\begin{gather*}
c^\ddagger(X,\gamma,f,L) 
:= \sum_{\alpha \in \mathrm{Mod}_g \cdot \gamma \setminus F(X, \gamma, K)}
\mathbbm{1}_{[\sqrt{L},L]}(\ell_\alpha(X)) \cdot f(\mathrm{RSC}_\alpha(X)).
\end{gather*}
By Proposition \ref{prop:asymp_triv} it follows that
\begin{equation}
\label{eq:a_2}
c^\dagger(X,\gamma,f,L) = c^\ddagger(X,\gamma,f,L) + \|f\|_\infty \cdot o_X\left(L^{6g-6}\right).
\end{equation}

We can now invoke the geometric comparison results of Section \ref{sec:weightsvary}.
Fix $\varepsilon \in (0,1)$ and consider any $Y \in \mathcal{M}_g$ such that $Y$ is $e^\varepsilon$-bi-Lipschitz close to $X$.
Now for any $\alpha \in \mathrm{Mod}_g \cdot \gamma \setminus F(X, \gamma, K)$,
it follows from Lemma \ref{lem:edgespersist} and Proposition \ref{prop:weightsvary} that $\mathrm{SC}_\alpha(Y)$ is in the same maximal facet as $\mathrm{SC}_\alpha(X)$ and that for every edge $e$ of $\mathrm{SC}_\alpha(X)$ and $\mathrm{SC}_\alpha(Y)$,
\begin{equation}
\label{eq:bd_3}
e^{-\varepsilon} \cdot |e|_X - K \leq |e|_Y \leq e^\varepsilon \cdot |e|_X + K
\end{equation}
and hence
\begin{equation}
\label{eq:bd_XYflip}
e^{-\varepsilon} \cdot |e|_Y - e^{-\varepsilon} K \leq |e|_X \leq e^\varepsilon \cdot |e|_Y + e^\varepsilon K.
\end{equation}
Dividing \eqref{eq:bd_XYflip} by $\ell_\alpha(X)$ and using the bounds $\ell_\alpha(X) \geq \sqrt{L}$ and $e^{-\varepsilon} \leq e^\varepsilon \leq 3$, we get that
\begin{equation}
\label{eq:bd_intermediate}
e^{-\varepsilon} \frac{|e|_Y}{\ell_\alpha(X)} - 3K/\sqrt{L} \leq
|e|_{\mathrm{RSC}_\alpha(X)}  \leq
e^\varepsilon \frac{|e|_Y}{\ell_\alpha(X)} + 3K/\sqrt{L}.
\end{equation}
As $Y$ is $e^\varepsilon$-bi-Lipschitz close to $X$ we have
\begin{equation}
\label{eq:bd_thu}
e^{-\varepsilon} \cdot \ell_\alpha(X) \leq \ell_\alpha(Y) \leq e^\varepsilon \cdot \ell_\alpha(X),
\end{equation}
so combining \eqref{eq:bd_intermediate} and (\ref{eq:bd_thu}) we deduce
\[
e^{-2\varepsilon} \cdot |e|_{\mathrm{RSC}_\alpha(Y)} - 3 K / \sqrt{L} \leq |e|_{\mathrm{RSC}_\alpha(X)} \leq e^{2\varepsilon} \cdot |e|_{\mathrm{RSC}_\alpha(Y)} + 3 K / \sqrt{L}.
\]
It follows that $\mathrm{RSC}_\alpha(X) \in N_{2\varepsilon,3K/\sqrt{L}}(\mathrm{RSC}_\alpha(Y))$, so by definition
\[
f(\mathrm{RSC}_\alpha(X)) \leq f^{\max}_{2\varepsilon,3K/\sqrt{L}}(\mathrm{RSC}_\alpha(Y)).
\]
From this bound and (\ref{eq:bd_thu}) we get the comparison of counting functions
\begin{equation}
\label{eq:a_3}
c^\ddagger(X,\gamma,f,L) \leq c(Y,\gamma,f^{\max}_{2\varepsilon,3K/\sqrt{L}},e^\varepsilon L).
\end{equation}
Putting together (\ref{eq:a_1}), (\ref{eq:a_2}), and (\ref{eq:a_3}) we conclude (\ref{eq:bd_2}) holds, i.e., 
\[
c(X,\gamma,f,L) \leq c\left(Y,\gamma,f_{2\varepsilon,3K/\sqrt{L}}^{\max},e^\varepsilon L\right) + \|f\|_\infty \cdot o_{X}\left(L^{6g-6}\right). \qedhere
\]
\end{proof}

Since $K$ depends only on $X$, taking $L$ arbitrarily large allows us to make the additive error of the neighborhood $\mathcal{N}_{\varepsilon, \delta}(\mrg)$ over which we are averaging as small as we wish, directly yielding the following corollary.

\begin{corollary}
    \label{cor:comp}
    For every $\delta > 0$ there exists a constant $L_0 = L_0(X,\delta) > 0$ such that for 
    every $f, \varepsilon$, and $Y$ as above, and every $L \geq L_0$,
    \begin{gather}
    c\left(Y,\gamma,f_{2\varepsilon,\delta}^{\min},e^{-\varepsilon} L\right) + \| f \|_\infty \cdot o_{X}\left(L^{6g-6}\right) \leq c(X,\gamma,f,L), \label{eq:bd_1b}\\
    c(X,\gamma,f,L) \leq c\left(Y,\gamma,f_{2\varepsilon,\delta}^{\max},e^\varepsilon L\right) + \|f\|_\infty \cdot o_{X}\left(L^{6g-6}\right). \label{eq:bd_2b}
    \end{gather}
\end{corollary}

Denote by $o_{X,f}(L^{6g-6})$ any function $o_{X,f} \colon \mathbf{R}_{>0} \to \mathbf{R}$ depending on $X \in \mathcal{M}_g$ and a function $f \colon \mathcal{MRG}_{g-1,2}(1,1) \to \mathbf{R}_{\geq 0}$ such that
\[
\lim_{L \to \infty} \frac{o_{X,f}(L)}{L^{6g-6}} = 0.
\]

Recall that $\widehat{\mu}_\mathrm{wp}$ denotes the Weil-Petersson measure on $\mathcal{M}_g$. 
For every $\varepsilon \in (0,1)$ denote by $U_X(\varepsilon) \subseteq \mathcal{M}_g$ the neighborhood of hyperbolic surfaces that are $e^\varepsilon$-bi-Lipschitz close to $X \in \mathcal{M}_g$ and let $\beta_{X,\varepsilon} \colon \mathcal{M}_g \to \mathbf{R}_{\geq 0}$ be any bump function supported on $U_X(\varepsilon)$ of total $\widehat{\mu}_\mathrm{wp}$ mass $1$. 
Integrating Corollary \ref{cor:comp} therefore gives the following result.

\begin{corollary}
    \label{cor:comp_2}
    For every $L \geq L_0(X, \delta)$ we have
    \begin{gather}
    \int_{\mathcal{M}_g} \beta_{X,\varepsilon}(Y) \cdot c\left(Y,\gamma,f_{2\varepsilon,\delta}^{\min},e^{-\varepsilon} L\right) \thinspace d\widehat{\mu}_\mathrm{wp}(Y) + o_{X,f}\left(L^{6g-6}\right) \leq c(X,\gamma,f,L), \label{eq:bd_4}\\
    c(X,\gamma,f,L) \leq \int_{\mathcal{M}_g} \beta_{X,\varepsilon}(Y) \cdot c\left(Y,\gamma,f_{2\varepsilon,\delta}^{\max},e^\varepsilon L\right) \thinspace d\widehat{\mu}_\mathrm{wp}(Y) + o_{X,f}\left(L^{6g-6}\right) \label{eq:bd_5}.
    \end{gather}
\end{corollary}

\subsection*{Unfolding averaged counts.}
Unfolding the integrals in (\ref{eq:bd_4}) and (\ref{eq:bd_5}) over $\mathcal{T}_{g} / \text{Stab}(\gamma)$ and pushing them back down to $\mathcal{M}_{g}$ in a suitable way will reduce the proof of Theorem \ref{theo:red} to an applicaton of Corollary \ref{cor:equid}.
The following proposition describes this principle; 
the reader should also compare \cite[Proposition 3.3]{Ara20a}.

\begin{proposition}
	\label{prop:pull_push}
	Fix a non-negative, continuous, compactly supported function $h \colon \mathcal{MRG}_{g-1,2}(1,1) \to \mathbf{R}_{\geq 0}$. Then, for every $\varepsilon > 0$ and every $L > 0$,
	\begin{equation}
	\label{eq:unfolding}
	\int_{\mathcal{M}_{g}} \beta_{X,\varepsilon}(Y) \cdot c(Y,\gamma,h,L) \thinspace d\widehat{\mu}_{\mathrm{wp}}(Y) = \int_{\mathcal{M}_{g}} \beta_{X,\varepsilon}(Y) \thinspace d\widehat{\mu}_{\gamma,h}^{L}(Y).
	\end{equation}
\end{proposition}

\begin{remark}
Notice that our weight function has changed names; this is because we eventually apply Proposition \ref{prop:pull_push} with $h$ equal to the functions $f_{2\varepsilon,\delta}^{\max}$ and $f_{2\varepsilon,\delta}^{\min}$.
\end{remark} 

\begin{proof}
Let $\varepsilon > 0$ and $L > 0$ be arbitrary. For every $Y \in \mathcal{M}_{g}$ one can rewrite the counting function $c(Y,\gamma,h,L)$ as follows:
	\begin{align*}
	c(Y,\gamma,h,L)  &= \sum_{\alpha \in \text{Mod}_{g} \cdot \gamma} \mathbbm{1}_{[0,L]}(\ell_\alpha(Y)) \cdot h(\mathrm{RSC}_\alpha(Y))\\
	&= \sum_{[\phi] \in \text{Mod}_{g}/\text{Stab}(\gamma)} \mathbbm{1}_{[0,L]}(\ell_{\phi.\gamma}(Y)) \cdot h(\mathrm{RSC}_{\phi.\gamma}(Y)) \\
	&= \sum_{[\phi] \in \text{Mod}_{g}/\text{Stab}(\gamma)} \mathbbm{1}_{[0,L]}(\ell_\gamma(\phi^{-1}.Y)) \cdot h(\mathrm{RSC}_{\gamma}(\phi^{-1}.Y)) \\
	&= \sum_{[\phi] \in \text{Stab}(\gamma) \backslash \text{Mod}_{g}} \mathbbm{1}_{[0,L]}(\ell_\gamma(\phi.Y)) \cdot h(\mathrm{RSC}_{\gamma}(\phi.Y)). \\
	\end{align*}
	Let us record this fact as
	\begin{equation}
	\label{eq:count_mas}
	c(Y,\gamma,h,L)  = \sum_{[\phi] \in \text{Stab}(\gamma) \backslash \text{Mod}_{g}} \mathbbm{1}_{[0,L]}(\ell_\gamma(\phi.Y)) \cdot h(\mathrm{RSC}_{\gamma}(\phi.Y)).
	\end{equation}
	Let $p_\gamma \colon \mathcal{T}_{g}/\text{Stab}(\gamma) \to \mathcal{M}_{g}$ be the quotient map and $\smash{\widetilde{\beta}_{X,\varepsilon}^\gamma} \colon \mathcal{T}_{g}/\text{Stab}(\gamma) \to \mathbf{R}_{\geq 0}$ be the lift of $ \beta_{X,\varepsilon}$ given by $\smash{\widetilde{\beta}_{X,\varepsilon}^\gamma} := \beta_{X,\varepsilon} \circ p_\gamma$. Unfolding the integral on the left hand side of (\ref{eq:unfolding}) using (\ref{eq:count_mas}) it follows that
	
	\begin{align*}
	\int_{\mathcal{M}_{g}} \beta_{X,\varepsilon}(Y) \cdot & c(Y,\gamma,h,L) \thinspace d \widehat{\mu}_\mathrm{wp}(Y) \\
	& = \int_{\mathcal{T}_{g}/\text{Stab}(\gamma)} \widetilde{\beta}_{X,\varepsilon}^\gamma(Y) \cdot \mathbbm{1}_{[0,L]}(\ell_\gamma(Y)) \cdot h(\mathrm{RSC}_{\gamma}(Y)) \thinspace
	d \widetilde{\mu}_{\mathrm{wp}}^\gamma(Y)  \\
	& = \int_{\mathcal{T}_{g}/\text{Stab}(\gamma)} \widetilde{\beta}_{X,\varepsilon}^\gamma(Y)
	\thinspace d \widetilde{\mu}_{\gamma,h}^{L}(Y) \\
	& = \int_{\mathcal{M}_{g}} \beta_{X,\varepsilon}(Y)
	\thinspace d \widehat{\mu}_{\gamma,h}^{L}(Y),
	\end{align*}
where the second equality follows from the expression for $\widetilde{\mu}_{\gamma,h}^{L}$ appearing in \eqref{eqn:RSC/Stab} and the third from the fact that $\widehat{\mu}_{\gamma,h}^{L}$ is the pushforward of $\widetilde{\mu}_{\gamma,h}^{L}$ to $\mathcal{M}_{g}$.
\end{proof}

\subsection*{Reducing counting to equidistribution.}
We are now ready to prove Theorem \ref{theo:red} by applying the equidistribution results proved in \S \ref{sec:RSCequi}.
Our strategy is to relate the $f$-weighted counted function $c(X,\gamma,f,L)$ with the mass of certain RSC-horoball measures, which we can then compare with the count $s(X,\gamma,L)$ of all geodesics in the $\Mod_g$-orbit of $\gamma$.

\begin{proof}[Proof of Theorem \ref{theo:red}]
Recall that we are aiming to prove that
\[
\lim_{L \to \infty} \frac{c(X,\gamma,f,L)}{s(X,\gamma,L)} = \frac{1}{c_g} \int_{\mathcal{MRG}_{g-1,2}(1,1)} f(x) \thinspace d\eta_\mathrm{Kon}(x).
\]

By Proposition \ref{prop:total_mass}, the following limit exists for any non-negative, continuous, compactly supported function $h \colon \mathcal{MRG}_{g-1,2}(1,1) \to \mathbb{R}_{\geq 0}$,
	\begin{equation}
	\label{eq:r(gamma,h)}
	\lim_{L \to \infty} \frac{m_{\gamma,h}^{L}}{L^{6g-6}} = \frac{1}{12g-12} \int_{\mathcal{MRG}_{g-1,2}(1,1)} h(x) \thinspace d\eta_{\mathrm{Kon}}(x).
	\end{equation}
	For the rest of the proof, denote this limit by $r(\gamma, h)$.
	
	By (\ref{eq:r(gamma,h)}) and Corollary 
    \ref{cor:identity}, proving Theorem \ref{theo:red} is equivalent to showing that
	\begin{equation}
		\label{eq:count_lbd}
		\frac{r(\gamma, f)}{c(\gamma)} \leq \liminf_{L \to \infty} \frac{c(X,\gamma,f,L)}{s(X,\gamma,L)},
	\end{equation}
	\begin{equation}
		\label{eq:count_ubd}
		\limsup_{L \to \infty} \frac{c(X,\gamma,f,L)}{s(X,\gamma,L)} \leq \frac{r(\gamma,f)}{c(\gamma)}.
	\end{equation}
	
	We verify (\ref{eq:count_ubd}); a proof of (\ref{eq:count_lbd}) can be obtained following similar arguments.
	Let $\delta > 0$ be arbitrary and take $L_0 = L_0(X,\delta) > 0$ as in Corollary \ref{cor:comp_2}.
	Fix $\varepsilon \in (0,1)$ and $L \geq L_0$. Set $h := f_{2\varepsilon,\delta}^{\max}$. Using Corollary \ref{cor:comp_2} (specifically \eqref{eq:bd_5}), we can average our counting function to get
	\[
	c(X,\gamma,f,L) \leq \int_{\mathcal{M}_g} \beta_{X,\varepsilon}(Y) \cdot c\left(Y,\gamma,h,e^\varepsilon L\right) \thinspace d\widehat{\mu}_\mathrm{wp}(Y) + o_{X,f}\left(L^{6g-6}\right).
	\]
	Unfolding this integral (Proposition \ref{prop:pull_push}) then implies that
	\[
	c(X,\gamma,f,L) \leq \int_{\mathcal{M}_{g}} \beta_{X,\varepsilon}(Y) \thinspace d\widehat{\mu}_{\gamma,h}^{L}(Y) + o_{X,f}\left(L^{6g-6}\right).
	\]
	Dividing this inequality by $m_{\gamma,h}^L$ (which is non-zero since $f$ is non-zero) we get
	\[
	\frac{c(X,\gamma,f,L)}{m_{\gamma,h}^L} \leq \int_{\mathcal{M}_{g}} \beta_{X,\varepsilon}(Y) \thinspace \frac{d\widehat{\mu}_{\gamma,h}^{L}(Y)}{m_{\gamma,h}^L} + \frac{o_{X,f}\left(L^{6g-6}\right)}{m_{\gamma,h}^L}.
	\]
	Taking $\limsup_{L \to \infty}$ on both sides of this inequality and using Corollary \ref{cor:equid} and Proposition \ref{prop:total_mass} we deduce that
\begin{equation}\label{eq:fcount_vs_RSCmass}
	\limsup_{L \to \infty} \frac{c(X,\gamma,f,L)}{m_{\gamma,h}^L} \leq \int_{\mathcal{M}_{g}} \beta_{X,\varepsilon}(Y) \thinspace \frac{B(Y) \cdot d\widehat{\mu}_\mathrm{wp}(Y)}{b_g}.
\end{equation}

Now we know that the RSC-horoball mass $m_{\gamma, h}^L$ and the counting function $s(X,\gamma,L)$ both grow polynomially with degree $L^{6g-6}$, so combining \eqref{eq:fcount_vs_RSCmass} with the definition of $r(\gamma, h)$ appearing in \eqref{eq:r(gamma,h)} and Mirzakhani's asymptotic formula \eqref{eq:mir}
we arrive at the following estimate:
	\begin{equation}
	\label{eq:coun_ineq_1}
	\limsup_{L \to \infty} \frac{c(X,\gamma,f,L)}{s(X,\gamma,L)} \leq \frac{r(\gamma,h)}{c(\gamma) \cdot B(X)} \cdot \int_{\mathcal{M}_{g,n}} \beta_{X,\varepsilon}(Y) \cdot B(Y) \thinspace d\widehat{\mu}_{\text{wp}}(Y).
	\end{equation}
	
We now shrink our approximating neighborhoods and consider the limiting behavior of the right-hand side of \eqref{eq:coun_ineq_1}.
	By definition, $h=f_{2\varepsilon,\delta}^{\max} \searrow f$ pointwise as $\varepsilon,\delta \searrow 0$. In particular, by the monotone convergence theorem,
	\begin{align*}
	\lim_{\varepsilon,\delta \searrow 0}  r(\gamma,f_{2\varepsilon,\delta}^{\max}) &= \lim_{\varepsilon,\delta \searrow 0}  \frac{1}{12g-12} \int_{\mathcal{MRG}_{g-1,2}(1,1)} f_{2\varepsilon,\delta}^{\max}(x) \thinspace d\eta_{\mathrm{Kon}}(x) \\
	&= \frac{1}{12g-12} \int_{\mathcal{MRG}_{g-1,2}(1,1)} f(x) \thinspace d\eta_{\mathrm{Kon}}(x)\\
	&= r(\gamma,f).
	\end{align*}
	Directly from the definition of $\beta_{X,\varepsilon} \colon \mathcal{M}_{g} \to \mathbf{R}_{\geq0}$ one checks that
	\[
	\lim_{\varepsilon \to 0}  \int_{\mathcal{M}_{g}} \beta_{X,\varepsilon}(Y) \cdot B(Y) = B(X).
	\]
	Taking $\varepsilon,\delta \searrow 0$ in (\ref{eq:coun_ineq_1}) we therefore deduce
	\[
	\limsup_{L \to \infty} \frac{c(X,\gamma,f,L)}{s(X,\gamma,L)} \leq \frac{r(\gamma,f)}{c(\gamma)}. \qedhere
	\]
\end{proof}

This completes the proof of the counting result (Theorem \ref{theo:red}), hence the proof of our main equidistribution result (Theorem \ref{theo:main_2}).

\section{Equidistribution of complementary subsurfaces: the general case.}\label{sec:general}

\subsection*{Outline of this section.}
In this section we state the main result of this paper for general simple closed multi-curves; see Theorem \ref{theo:main_2_new}. 
As its proof mirrors that of Theorem \ref{theo:main}, we discuss only a few of its most important aspects; the rest of the argument can be quoted {\em mutatis mutandis}.
We begin by introducing appropriate terminology concerning simple closed multi-curves and their corresponding cut-and-glue fibrations.
We also introduce the fibered Kontsevich measure, a key player in the statement of Theorem \ref{theo:main_2_new}.
We finish this section with a discussion on simultaneous equidistribution; see Theorem \ref{theo:main_3_new}.

\subsection*{Multi-curves and subsurfaces.}
For the rest of this section fix an integer $g \geq 2$ and a connected, oriented, closed surface $S_g$ of genus $g$.
A {\em simple closed multi-curve}
$\gamma = \{\gamma_1, \ldots, \gamma_k\}$ on $S_g$ with $1 \leq k \leq 3g-3$ components is a $k$-tuple of pairwise non-isotopic, pairwise non-intersecting, essential (isotopy classes of) simple closed curves.
Given a marked hyperbolic structure $X \in \mathcal{T}_g$, the total length of $\gamma$ is given by
\[
 \ell_{\gamma}(X) := \ell_{\gamma_1}(X) + \cdots + \ell_{\gamma_k}(X).
\]
Throughout this section, we assume all of our multi-curves are simple and consist of closed curves, so we refer to them simply as multi-curves.

For any $X \in \mathcal{T}_g$, cutting along the geodesic realization of $\gamma$ on $X$ results in a hyperbolic structure on the (possibly disconnected) surface with boundary $S_g \setminus \gamma$ together with an induced marking.
Label the components $(\Sigma_j)_{j=1}^c$ of $S_g \setminus \gamma$, and for each $j \in \{1,\dots,c\}$ let $g_j,b_j \geq 0$ to be the pair of non-negative integers such that $\Sigma_j$ is homeomorphic to $\smash{S_{g_j,b_j}}$.
Cutting along $\gamma$ therefore yields a map
\[\mathrm{C}_{\gamma}:\T_g \to \prod_{j=1}^c \T_{g_j, b_j}\]
whose fibers are homeomorphic to $\mathbb{R}^k$, representing all the possible Fenchel-Nielsen twists along the components of $\gamma$.

\subsection*{The cut-and-glue fibration}
We would like to push $\mathrm{C}_\gamma$ down to a map from the moduli space of pairs (hyperbolic surface, multi-curve) to the moduli space of hyperbolic structures on the complementary subsurfaces.
However, a difficulty arises in that simply marking a multi-curve is not enough to consistently identify its complementary subsurfaces.
At fault is the fact that the mapping class group can permute the components of $\gamma$, and can stabilize a component while reversing its orientation.

To address this, define an ordered, oriented multi-curve $\agamma:= (\agamma_1,\dots,\agamma_k)$ to be a multi-curve together with an ordering of its components and a choice of orientation on each.
Now, given any hyperbolic surface $X \in \M_g$ and a $\agamma$ on $X$, the subsurface ``to the right of $\gamma_i$'' is well-defined, and so we can distinguish between the components of $X \setminus \agamma$.
As before, we use $\Sigma_j$ to denote the (labeled) components of $S_g \setminus \agamma$ and $g_j, b_j$ to denote the genus and number of boundary components of $\Sigma_j$.

\begin{remark}
 Given an ordered, oriented multi-curve $\agamma:= (\agamma_1,\dots,\agamma_k)$ we will sometimes denote by $\gamma:= (\gamma_1,\dots,\gamma_k)$ the underlying (unoriented) ordered multi-curve.
\end{remark}

Given a simple closed curve $\gamma$ on $S_{g}$, denote by
$\text{Stab}_0(\gamma) < \text{Mod}_{g}$
the subgroup of all mapping classes that fix $\gamma$ up to isotopy together with its orientations. Since $\gamma$ only admits two orientations, this is clearly the same as the subgroup of $\Mod_g$ that fixes either of those orientations.

For an ordered multi-curve $\gamma := (\gamma_1,\dots,\gamma_k)$ on $S_{g}$ with $1 \leq k \leq 3g-3$ components, denote by
\[
\Stab_0(\gamma) := \bigcap_{i=1}^k \Stab_0(\gamma_i) < \text{Mod}_{g}
\]
the stabilizer of each component of $\gamma$ together with their respective orientations.
As above, this is equal to the stabilizer of the ordered, oriented multi-curve $\agamma$ for any choice of orientation of the components of $\gamma$.
In particular, $\Stab_0(\gamma)$ preserves each complementary subsurface of $S_g \setminus \agamma$ and fixes each of its boundaries setwise.

\begin{remark}
There is a natural isomorphism
\[
\Stab_0(\gamma)
\cong
\langle T_{\gamma_1}, \ldots, T_{\gamma_k} \rangle
\times 
\prod_{j=1}^c \Mod_{g_j, b_j} 
\]
where $T_{\gamma_i}$ denotes the Dehn twist in $\gamma_i$.
This follows from our convention that $\Mod_{g,b}$ is considered up to isotopies {\em setwise} fixing each boundary component.
\end{remark}

Each point of $\T_g / \Stab_0(\gamma)$ is now a hyperbolic structure together with a choice of ordered, oriented multi-curve in the $\Mod_g$ orbit of $\agamma$. Equivalently, this quotient is the moduli space of hyperbolic structures on $S_g$ equipped with an unoriented ordered multi-curve, together with an labeling of its complementary subsurfaces.

Cutting a hyperbolic surface $X \in \T_g / \Stab_0(\gamma)$ along the geodesic representative of the specified multi-curve gives (unmarked) hyperbolic structures on the complementary subsurfaces $\{\Sigma_j\}_{j=1}^c$, but not all structures are represented.
In particular, boundary components glued along a curve of $\gamma$ must have equal length.

We therefore introduce the following notation: if
$\mathbf{L} = \smash{(L_i)_{i=1}^k }\in \mathbb{R}_{>0}^k$,
then for each $j \in \{1,\dots,c\}$ we let
$\smash{\mathbf{L}^{(j)}} \in \mathbb{R}_{>0}^{b_j}$
be the subvector of $\mathbf{L}$ whose entries correspond to the boundary components of $\Sigma_j$.
Now define the slice $\M(S_g \setminus \agamma; \bfL)$ to be the moduli space of complementary subsurfaces where the components of $\gamma$ have fixed lengths, i.e., 
\[
\M(S_g \setminus \agamma; \bfL) :=
\prod_{j=1}^c \M_{g_j,b_j} \left(\bfL^{(j)}\right).
\]

These slices piece together into a fibration over the space of possible lengths, and we define the moduli space $\M(S_g \setminus \agamma)$ of complementary subsurfaces as the total space of this fibration:
\[
\M(S_g \setminus \agamma) :=
\left\{ (X_1, \ldots, X_c) \in \M(S_g \setminus \agamma; \bfL)
\mid 
\bfL \in \mathbb{R}_{>0}^k \right\}.
\]

With all of this notation established, we now observe that there is a fibration
\[\mathrm{C}_\gamma: \T_g / \Stab_0(\gamma) \to \M(S_g \setminus \agamma),\]
whose fiber over $(Y_1, \ldots, Y_c) \in \prod_{j=1}^c \M_{g_j,b_j} \left(\bfL^{(j)}\right)$ is a $k$-dimensional torus representing the different gluings of the complementary subsurfaces $(X_1, \ldots, X_c)$. This is often referred to as the {\em cut-and-glue fibration}.

We note that the exact shape of the fibers is somewhat subtle: one would initially expect the fiber to be equal to 
\[\RR^k / (L_1 \mathbb{Z} \oplus \ldots \oplus L_k \mathbb{Z}),\]
representing that one can twist along each $\gamma_i$ and get different hyperbolic structures on the glued surface.
However, if any of the $\gamma_i$ separate off a torus with one boundary component then the gluings with twist parameters $t$ and $t + L_i/2$ are isometric; this is a consequence of the presence of the elliptic involution.
In general, then, the fiber is equal to 
\[\RR^k / (2^{-\varepsilon_1} L_1 \mathbb{Z} \oplus \ldots \oplus 2^{-\varepsilon_k} L_k \mathbb{Z}),\]
where $\varepsilon_i =1$ if $\gamma_i$ separates off a subsurface of genus 1 with a single boundary component and $\varepsilon_i=0$ otherwise.

\subsection*{Disintegrating the Weil--Petersson measure}
Using Wolpert's formula for the Weil--Petersson symplectic form
\cite{Wol85},
we can express the Weil--Petersson volume form $\mu_{\WP}$ on $\T_g$ as
\begin{equation}\label{eqn:WPcutandglue_Teich}
d\mu_{\WP} = d\mu_{\WP}^{\agamma, \bfL} \,
d \tau_1 \ldots d \tau_k \thinspace
d L_1 \ldots d L_k,
\end{equation}
where $\tau_i$ are the twist parameters associated with the components of $\gamma$ and $\mu_{\WP}^{\agamma,\bfL}$ is the product of the Weil--Petersson measures on the Teichm{\"u}ller spaces of the complementary subsurfaces, i.e.,
\[\mu_{\WP}^{\agamma,\bfL} = \mu_{\WP}^{\bfL^{(1)}} \otimes \ldots \otimes \mu_{\WP}^{\bfL^{(c)}}.\]

We now wish to push equation \eqref{eqn:WPcutandglue_Teich} forward by the cut-and-glue fibration and down to the moduli spaces considered above.
However, there are number of technical subtleties we must take into account to get the correct constant factor.
For one, the size of the toral fibers of $\mathrm{C}_\gamma$ depends on both the length $\bfL$ of the curves of $\gamma$ as well as the number $\rho_{g}(\gamma)$ of components of $\gamma$ that bound a torus with one boundary component.
An even more delicate issue arises from our definition of the Weil--Petersson measures; since we are taking a local pushforward measure to an orbifold, we must divide each measure by the size of the generic stabilizer.
For a more thorough discussion of these issues see \cite{Ara19b,Ara20a}.

Therefore, let $\sigma_{g}(\gamma) > 0$ be the rational number given by
\[
\sigma_{g}(\gamma) := \frac{\prod_{j=1}^c |K_{g_j,b_j}|}{|\text{Stab}_0(\gamma)\cap K_{g}|},
\]
where $K_{g_j,b_j} \triangleleft \text{Mod}_{g_j,b_j}$ is the kernel of the mapping class group action on $\mathcal{T}_{g_j,b_j}$ and $K_{g} \triangleleft \text{Mod}_{g}$ is the kernel of the mapping class group action on $\mathcal{T}_{g}$.
These kernels are non-trivial only in the low complexity cases where special involutions arise. More specifically, $|K_{g_j,b_j}| = 1$ unless $g_j = b_j = 1$ in which case $|K_{1,1}| = 2$, and $|K_g|=1$ unless $g=2$ in which case $|K_2|=2$.

Once we have dealt with these issues, we can push forward formula \eqref{eqn:WPcutandglue_Teich}.
Let $\widetilde{\mu}_{\WP}$ denote the local pushforward of $\mu_{\WP}$ to $\T_g / \Stab_0(\gamma)$, and define the Weil--Petersson measure $\widehat{\mu}^{\agamma,\bfL}_{\WP}$ on $\M(S_g \setminus \agamma; \bfL)$ as the product of the Weil--Petersson measures $\widehat{\mu}^{\bfL^{(j)}}_{\WP}$ on the moduli spaces of complementary subsurfaces.
Then there is the following relationship between measures on $\M(S_g \setminus \agamma)$:
\begin{equation}\label{eqn:WPcutandglue}
d (C_\gamma)_* \widetilde{\mu}_{\WP} = 
\frac{\sigma_{g}(\gamma)}
{ 2^{\rho_{g}(\gamma)}} \thinspace
L_1 \cdots L_k \, d\widehat{\mu}_{\WP}^{\bfL} \, d L_1 \ldots d L_k.
\end{equation}

\subsection*{The fibered Kontsevich measure}
We now mimic the above constructions for the space of spines of complementary subsurfaces.
Therefore, for any length vector $\bfL \in \mathbb{R}_{>0}^k$, we set
\[
\MRG(S_g \setminus \agamma; \bfL) :=
\prod_{j=1}^c \MRG_{g_j,b_j} \left(\bfL^{(j)}\right).
\]
As these are simply products of moduli spaces with fixed boundary lengths, they come equipped with a natural (product) Kontsevich measure $\eta_\mathrm{Kon}^{\agamma, \bfL}$.
As in the case of hyperbolic structures on $S_g \setminus \agamma$, these slices fit together into a larger moduli space $\MRG(S_g \setminus \agamma)$ which can be topologized through its natural embedding into a product of combinatorial moduli spaces with variable boundary lengths.
By Theorem \ref{thm:Luo} the spine map induces a homeomorphism
\[\mathrm{S}: \M(S_g \setminus \agamma) \xrightarrow{\sim} \MRG(S_g \setminus \agamma)\]
that restricts to a homeomorphism on each slice.

Integrating against boundary lengths, the Kontsevich measures on each slice also fit together into a measure on the total space, defined by
\[
\eta_{\mathrm{Kon}}^{\agamma}(A) := 
\int_{\mathbb{R}_{>0}^{k}} 
\eta_\mathrm{Kon}^{\agamma, \bfL}
\left( A \cap \MRG(S_g \setminus \agamma; \bfL) \right)
\,
dL_1 \ldots dL_k 
\]
for any measurable set $A$ of $\MRG(S_g \setminus \agamma)$.

Denote by $\Delta \subseteq \mathbb{R}_{>0}^k$ the (open) standard simplex
\[
\Delta := \{\bfL \in \mathbb{R}_{>0}^k \colon L_1 + \dots + L_k = 1\}
\]
and let $\MRG(S_g \setminus \agamma;\Delta)$ denote the total space of the fibration over $\Delta$ whose fiber above $\bfL$ is $\MRG(S_g \setminus \agamma; \bfL)$; this can be thought of as the ``projectivization'' of $\MRG(S_g \setminus \agamma)$.
To that point, we can define a map
\[
\mathrm{RSC}_\gamma :
\T_g/ \Stab_0(\gamma) \to \MRG(S_g \setminus \agamma;\Delta)
\]
which to every hyperbolic structure assigns the point obtained by cutting $X$ along the geodesic representatives of the components of $\agamma$, finding the metric spines of the complementary hyperbolic surfaces, and rescaling the corresponding metric ribbon graphs to lie in $\MRG(S_g \setminus \agamma;\Delta)$.

In analogy with formula \eqref{eqn:WPcutandglue}, we now define a measure on $\mathcal{MRG}(S_g \setminus \agamma;\Delta)$ as follows.
For a subset $A \subset \mathcal{MRG}(S_g \setminus \agamma; \Delta)$, define 
\[
\text{cone}(A) :=
\{ (\Gamma, t\mathbf{x}): (\Gamma, \mathbf{x}) \in A, t \in (0,1]\}
\subset \MRG(S_g \setminus \agamma).
\]

\begin{definition}\label{def:fib_Kont}
The {\em fibered Kontsevich measure} of $A \subset \mathcal{MRG}(S_g \setminus \agamma; \Delta)$ is
\[
\mathring{\eta}_\mathrm{Kon}^{\Delta}(A) := 
\frac{\sigma_{g}(\gamma)}
{ 2^{\rho_{g}(\gamma)}}
\int_{\text{cone}(A)} L_1 \cdots L_k \, d \eta_{\mathrm{Kon}}^{\agamma}(\Gamma, \mathbf{x}).
\]
\end{definition}

The ``fibered'' qualifier and the circle accent are both meant to emphasize that we are taking into account the mass of the toral fibers coming from different gluings along the curves of $\gamma$.
As we will see shortly, the fibered Kontsevich measure reflects the asymptotic behavior of the $\mathrm{RSC}_\gamma$-pushforward of Weil--Petersson measure to $\M(S_g \setminus \agamma; \Delta)$ as the length of $\gamma$ grows.
Compare to Proposition \ref{prop:total_mass_comp}.

\begin{remark}\label{rmk:K_normalization}
By \eqref{eqn:Konts=Euclid}, the measure $\mathring{\eta}_\mathrm{Kon}^{\Delta}$ can also be expressed in terms of the Lebesgue measure on $\prod_{j=1}^c \MRG_{g_j,b_j}$ in edge length coordinates, restricted to the subspace $\MRG(S_g \setminus \agamma)$.
This requires lifting to the Teichm{\"u}ller space of complementary subsurfaces (equivalently, the filling arc complex $|{\mathscr{A}_\text{fill}}(\Sigma)|_{\RR}$) and then taking a local pushforward.
Since $\M_{g,b}$ and $\MRG_{g,b}$ have the same symmetries, this strategy also naturally recovers the weighting factor $\sigma_g(\gamma) / 2^{\rho_g(\gamma)}$.
\end{remark}

\subsection*{Counting multi-curves.} Let $\gamma := (\gamma_1,\dots,\gamma_k)$ be an ordered multi-curve on $S_g$ with $1 \leq k \leq 3g-3$ components and $X \in \mathcal{T}_g$ be a marked hyperbolic structure on $S_g$. For every $L > 0$ consider the counting function
\[
s(X,\gamma,L) := \# \{\alpha \in \mathrm{Mod}_g \cdot \gamma \ | \ \ell_\alpha(X) \leq L \}.
\]
By work of Mirzakhani \cite{Mir08b}, an asymptotic formula analogous to \eqref{eq:mir} holds for this counting function as well; compare also \cite{Liu19,Ara20a}.

To state this result, we first define the following constant,
\[c(\gamma) := \frac{\mathring{\eta}_\mathrm{Kon}^{\Delta}
\left( \MRG(S_g \setminus \agamma; \Delta) \right)
}   {[\Stab(\gamma): \Stab_0(\gamma)]} > 0.
\]
Recall from Section \ref{sec:RSCequi} that $B(X) > 0$ denotes the Thurston volume of the ball of unit length measured geodesic laminations on $X$ and that $b_g > 0$ denotes the total mass of the Mirzakhani measure on $\PoM_g$.

\begin{theorem}
    \cite[Theorems 1.1, 1.2]{Mir08b} \cite[Theorem 1.2]{Liu19} \cite[Theorem 1.6]{Ara20a} 
    Let $\gamma := (\gamma_1,\dots,\gamma_k)$ be an ordered multi-curve on $S_g$ with $1 \leq k \leq 3g-3$ components and $X \in \mathcal{M}_g$ be a hyperbolic structure on $S_g$. Then,
    \[
    \lim_{L \to \infty} \frac{s(X,\gamma,L)}{L^{6g-6}} = \frac{c(\gamma) \cdot B(X)}{b_g}.
    \]
\end{theorem}

\begin{remark}
In \cite{Mir08b}, \cite{Liu19}, and \cite{Ara20a}, the constant $c(\gamma)$ is defined in terms of polynomials $W_g(\gamma, \bfL)$
that record the top degree part $\smash{V_{g_j, b_j}^{\text{top}}}$ of the Weil--Petersson volume polynomials of the complementary subsurfaces.
In particular, once one fully unravels our definition of $c(\gamma)$ and the fibered Kontsevich measure, it becomes clear our definition is the same as the one that appears in those papers but with $\smash{V_{g_j, b_j}^{\text{top}}}$ replaced by the Kontsevich volumes of complementary moduli spaces.
Corollary \ref{cor:WPtop=Kont} implies that our definition is equivalent to the original one.
\end{remark}

\subsection*{Equidistribution of complementary subsurfaces.}
We can now state the general version of our main theorem for arbitrary multi-curves. 
Fix an ordered oriented multi-curve $\agamma := (\agamma_1,\dots,\agamma_k)$ and a marked hyperbolic structure $X \in \mathcal{T}_g$.
On $\MRG(S_g \setminus \agamma;\Delta)$ consider the counting measure
\[
\eta_{X,\agamma}^L := \sum_{\vec{\alpha} \in \mathrm{Mod}_g \cdot \agamma} \mathbbm{1}_{[0,L]}(\ell_{\vec{\alpha}}(X)) \cdot \delta_{\mathrm{RSC}_{\vec{\alpha}}(X)}.
\]
This measure does not depend on the marking of $X \in \mathcal{T}_g$ but only on its underlying hyperbolic strucure $X \in \mathcal{M}_g$.
Recall that $\mathring{\eta}_\mathrm{Kon}^{\Delta}$ denotes the fibered Kontsevich measure on $\mathcal{MRG}(S_g \setminus \agamma; \Delta)$ and denote by 
$m_{\agamma} > 0$
its total mass.

The following generalization of Theorem \ref{theo:main} can be proved using the same strategy adopted in previous sections; for the sake of brevity, we will only discuss some of the most important aspects of the proof.

\begin{theorem}
\label{theo:main_2_new}
Let $\agamma:=(\agamma_1,\dots,\agamma_k)$ be an ordered oriented multi-curve on $S_g$ and $X \in \mathcal{M}_g$ be a hyperbolic structure on $S_g$. Then, with respect to the weak-$\star$ topology for measures on $\mathcal{MRG}(S_g \setminus \agamma;\Delta)$,
\[
\lim_{L \to \infty} \frac{\eta_{X,\agamma}^L}{s(X,\gamma,L)} = \frac{\mathring{\eta}_{\mathrm{Kon}}^{\Delta}}{m_{\agamma}}.
\]
\end{theorem}

\begin{remark}
The results in \cite{Liu19} and \cite{Ara20a} also explain how to count multi-curves in the $\Mod_g$ orbit of $\agamma$ with respect to other notions of length, e.g., with non-uniform weights on the components or by the maximum length of the components, instead of the total length.
Analogues of Theorem \ref{theo:main_2_new} also hold true in these settings; we leave it to the reader to formulate precise statements.
\end{remark}

\subsection*{RSC-horoball measures.} We begin by disussing how to define appropriate RSC-horoball measures in this context. For the rest of this section fix a non-zero, non-negative, continuous, compactly supported function $h \colon \mathcal{MRG}_{g}(\agamma,\Delta) \to \mathbb{R}_{>0}$. For every $L > 0$ consider the {\em RSC-horoball measure} $\mu_{\vec{\gamma},h}^L$ on $\mathcal{T}_g$ given by restricting the Weil-Petersson measure to the horoball of depth $L$ at $\vg$ and weighting it by $h$:
\begin{equation*}
d\mu_{\vg,h}^L(X) := \mathbbm{1}_{[0,L]}(\ell_{\vg}(X)) \thinspace h(\mathrm{RSC}_{\vg}(X)) \thinspace d\mu_\mathrm{wp}(X).
\end{equation*}
As a consequence of the cut-and-glue fibration, this measure is invariant by  $\mathrm{Stab}_0(\vg)$ as well as by twisting in the curves of $\gamma$.
Consider the sequence of covers
\[
\mathcal{T}_g \to \mathcal{T}_g/\mathrm{Stab}_0(\vg) \to \mathcal{M}_g.
\]
Let $\widetilde{\mu}_{\vg,h}^L$ be the local pushforward of $\mu_{\vg,h}^L$ to $\mathcal{T}_g/\mathrm{Stab}_0(\vg)$ and $\widehat{\mu}_{\vg,h}^L$ be the pushworward of $\widetilde{\mu}_{\vg,h}^L$ to $\mathcal{M}_g$.

\begin{remark}\label{rmk:nonsepRSC}
    For an oriented non-separating simple closed curve $\agamma$ on $S_g$,
    \[
    \widehat{\mu}_{\vg,h}^L = [\mathrm{Stab}(\gamma):\mathrm{Stab}_0(\gamma)] \cdot \widehat{\mu}_{\gamma,h}^L = 2\widehat{\mu}_{\gamma,h}^L
    \]
where the measure $\widehat{\mu}_{\gamma,h}^L$ is as defined in \eqref{eq:RSC_meas}.
\end{remark}

These measures can be lifted to the corresponding bundles of unit length measured geodesic laminations in the following way. Let $\mathbf{1} \cdot \vg$ be the weighted multi-curve $\gamma_1 + \cdots + \gamma_k$ considered as an element of $\mathcal{ML}_g$. On $\mathcal{P}^1\mathcal{T}_g$ consider the measures $\nu_{\vg,h}^L$ given by
\[
d\nu_{\gamma,h}^L(X,\lambda) :=  d \delta_{\mathbf{1}\cdot \vg/\ell_{\vg}(X)}(\lambda) \thinspace d\mu_{\vg,h}^L(X).
\]
This measure is similarly $\mathrm{Stab}_0(\vg)$ invariant, and as a consequence of the twist-invariance of $\mu_{\vg,h}^L(X)$, is invariant under the earthquake flow. Let $\widetilde{\nu}_{\vg,h}^L$ be the local pushforward of $\nu_{\vg,h}^L$ to $\mathcal{P}^1\mathcal{T}_g/\mathrm{Stab}_0(\vg)$ and $\widehat{\nu}_{\vg,h}^L$ be the pushforward of $\widetilde{\nu}_{\vg,h}^L$ to $\mathcal{P}^1\mathcal{M}_g$. 
Denote by
\[m_{\vg,h}^L := \widehat{\nu}_{\vg,h}^L(\mathcal{P}^1\mathcal{M}_g) = \widehat{\mu}_{\vg,h}^L(\mathcal{M}_g)\]
the total mass of these measures. 

To prove Theorem \ref{theo:main_2_new} using the methods discussed in previous sections one needs to ensure the measures $\widehat{\mu}_{\vg,h}^L$ equidistribute over $\mathcal{M}_g$ with respect to $B(X) \thinspace d\widehat{\mu}_\mathrm{wp}(X)$ as $L \to \infty$. Following the same arguments as in the proof of Theorem \ref{theo:horoball_equid_1} one first deduces the following result.

\begin{theorem}
\label{theo:equid_new}
    With respect to the weak-$\star$ topology for measures on $\mathcal{P}^1\mathcal{M}_g$,
    \[
    \lim_{L \to \infty} \frac{\widehat{\nu}_{\vec{\gamma},h}^L}{m_{\vec{\gamma},h}^L} = \frac{\widehat{\nu}_\mathrm{Mir}}{b_g}.
    \]
\end{theorem}

Pushing down to moduli space, one immediately concludes the following.

\begin{theorem}
\label{theo:equid_new_Mg}
    With respect to the weak-$\star$ topology for measures on $\mathcal{M}_g$,
    \[
    \lim_{L \to \infty} \frac{\widehat{\mu}_{\vec{\gamma},h}^L}{m_{\vec{\gamma},h}^L} = \frac{B(X) \thinspace d\widehat{\mu}_\mathrm{wp}(X)}{b_g}.
    \]
\end{theorem}

\subsection*{Total mass.} As in the case of Theorem \ref{theo:horoball_equid_1} an important step in the proof of Theorem \ref{theo:equid_new} is to compute an asymptotic formula for the total mass $\smash{m_{\vg,h}^L}$ of the measures $\smash{\widehat{\nu}_{\vg,h}^L}$ as $L \to \infty$. Using the same methods as in the proof of Proposition \ref{prop:total_mass} together with the cut-and-glue fibration yields the following result.

\begin{proposition}
    \label{prop:total_mass_comp}
    The following limit holds,
    \[
    \lim_{L \to \infty} \frac{m_{\vec{\gamma},h}^L}{L^{6g-6}} = 
    \int_{\mathcal{MRG}(S_g \setminus \agamma;\Delta)} h(x) \thinspace d\mathring{\eta}_{\mathrm{Kon}}^{\Delta}(x).
    \]
\end{proposition}

\begin{remark}
A priori, in the case $\vec{\gamma}$ is an oriented non-separating simple closed curve, the limit in Proposition \ref{prop:total_mass_comp} seems to differ from the limit in Proposition \ref{prop:total_mass}. This difference can be reconciled in the following way. Notice that
\[\mathring{\eta}_\mathrm{Kon}^\Delta :=
\int_0^1 \ell \thinspace \eta_\mathrm{Kon}^{(\ell,\ell)} d \ell.\]
By quasi-invariance of the Kontsevich measure under rescaling \eqref{eqn:Konts_rescale},
\[\eta_\mathrm{Kon}^{(\ell,\ell)} = \ell^{6g-8} \thinspace \eta_\mathrm{Kon}^{(1,1)} \]
and thus we get that
\[
\mathring{\eta}_\mathrm{Kon}^\Delta =
\int_0^1 \ell^{6g-7} \thinspace \eta_\mathrm{Kon}^{(1,1)} d \ell
= \frac{1}{6g-6} \eta_\mathrm{Kon}^{(1,1)}.
\]
The last missing factor of 2 comes from the fact that we have defined our measures by pushing down to $\M_g$ from $\mathcal{T}_g/\mathrm{Stab}_0(\gamma)$ instead of from $\mathcal{T}_g/\mathrm{\Stab}(\gamma)$, as in Remark \ref{rmk:nonsepRSC}.
\end{remark}

\subsection*{Simultaneous equidistribution.} Inspired by \cite{AES16a,AES16b,ERW19}, we now discuss the issue of simultaneous equidistribution, proving that the placement of a curve in the space of measured laminations is independent from the shape of its complementary subsurface.

Recall that $\mathcal{PML}_g$ denotes the space of projective measured geodesic laminations on $S_g$ and that $[\lambda] \in \mathcal{PML}_g$ denotes the projective class of $\lambda \in \mathcal{ML}_g$. Given $X \in \mathcal{T}_g$ consider the coned-off Thurston measure $\mu_\mathrm{Thu}^X$ which to every measurable subset $A \subseteq \mathcal{PML}_g$ assigns the value
\[
\mu_\mathrm{Thu}^X(A) := \mu_\mathrm{Thu}(\{\lambda \in \mathcal{ML}_g \ | \ \ell_\lambda(X) \leq 1, \ [\lambda] \in A\});
\]
up to identifying $\mathcal{P}_X^1\mathcal{T}_g \cong \PML_g$, this is just the measure $\mu_\mathrm{Thu}^X$ introduced in Section \ref{sec:RSCequi}.
Observe that its total mass is $B(X)$.

Fix an ordered multi-curve $\gamma := (\gamma_1,\dots,\gamma_k)$ on $S_g$ and a marked hyperbolic structure $X \in \mathcal{T}_g$. Recall that $\mathbf{1} \cdot \gamma := \gamma_1 + \dots + \gamma_k \in \mathcal{ML}_g$. For every $L > 0$ consider the counting measure on $\mathcal{PML}_g$ given by
\[
\zeta_{\gamma,X}^L := \sum_{\alpha \in \mathrm{Mod}_g \cdot \gamma} \mathbbm{1}_{[0,L]}(\ell_{\alpha}(X)) \cdot \delta_{[\mathbf{1} \cdot \alpha]}.
\]
This measure depends only on the underlying hyperbolic structure of $X \in \mathcal{T}_g$ and not on its marking. The following result can be deduced directly from Mirzakhani's work \cite[Theorem 6.4]{Mir08b}.

\begin{theorem}
\label{theo:mir}
In the weak-$\star$ topology for measures on $\mathcal{PML}_g$,
\[
\lim_{L \to \infty} \frac{\zeta_{\gamma,X}^L}{s(X,\gamma,L)} = \frac{\mu_\mathrm{Thu}^X}{B(X)}.
\]
\end{theorem}

It is natural to consider the question of simultaneous equidistribution for the limits in Theorems \ref{theo:main_2_new} and \ref{theo:mir}. More precisely, fix an ordered oriented multi-curve $\vg := (\vg_1,\dots,\vg_k)$ on $S_g$ and $X \in \mathcal{T}_g$. For every $L > 0$ consider the counting measure on $\mathcal{PML}_g \times \mathcal{MRG}(S_g \setminus \agamma;\Delta)$ given by
\[
\xi_{\gamma,X}^L := \sum_{{\vec{\alpha}} \in \mathrm{Mod}_g \cdot {\vg}} \mathbbm{1}_{[0,L]}(\ell_{\alpha}(X)) \cdot \delta_{[\mathbf{1} \cdot \alpha]} \otimes \delta_{\mathrm{RSC}_{\vec{\alpha}}(X)}.
\]
As always, this measure depends only on the underlying hyperbolic structure of $X \in \mathcal{T}_g$ and not on its marking.
The following general result can be proved following similar arguments as in the proof of Theorem \ref{theo:main_2_new} but working over the bundle $\mathcal{P}^1\mathcal{M}_g$; compare to \cite[Proof of Theorem 3.5]{Ara20a}.

\begin{theorem}
\label{theo:main_3_new}
Let $\agamma:=(\agamma_1,\dots,\agamma_k)$ be an ordered oriented multi-curve on $S_g$ with $1 \leq k \leq 3g-3$ components and $X \in \mathcal{M}_g$. Then, with respect to the weak-$\star$ topology for measures on $\mathcal{PML}_g \times \mathcal{MRG}(S_g \setminus \agamma;\Delta)$,
\[
\lim_{L \to \infty} \frac{\xi_{\gamma,X}^L}{s(X,\gamma,L)} = \frac{\mu_\mathrm{Thu}^X}{B(X)} \otimes \frac{\mathring{\eta}_{\mathrm{Kon}}^{\Delta}}{m_{\agamma}}.
\]
\end{theorem}

As a consequence, we see that even when prescribing how a set of curves coarsely wraps $X$ (for example, by fixing a maximal train track chart for $\ML_X$), the complementary subsurfaces to those curves remain uniformly distributed.

\begin{appendices}
\section{A Morse Lemma with shrinking constants}\label{appendix}

The usual Morse Lemma in hyperbolic geometry states that quasigeodesics remain close to geodesics.
Recall that a map $c: \mathbb{R} \to X$ to a metric space $X$ is an {\em $(L,K)$-quasigeodesic} if for all $t, t' \in \mathbb{R}$ we have that
\[\frac{1}{L}|t-t'| - K
\le 
d_X(c(t), c(t'))
\le 
L|t - t'| + K.\]

\begin{lemma}\label{Morse Lemma}
For any geodesic, Gromov-hyperbolic metric space $X$ and any constants $L,K >0$, there exists $R>0$ such that any $(L, K)$-quasigeodesic $c$ is contained within the $R$ neighborhood of $g$, the geodesic with the same endpoints as $c$.
\end{lemma}

See \cite[Theorem III.H.1.7]{BH} for a proof.
In this appendix, we show that in the case where we restrict to $X = \mathbb{H}^2$ and set $K=0$, we can take $R \to 0$ as $L \to 1$.

\begin{proposition}\label{prop:shrinkMorse}
For any $\delta >0$ there exists a constant $L >1$ so that  any $L$-bi-Lipschitz curve $c: \mathbb{R} \to \mathbb{H}^2$ is contained within the $\delta$ neighborhood of $g$, the geodesic with the same endpoints as $c$.
\end{proposition}

We first record one useful fact about the hyperbolic geometry of certain quadrilaterals; see
\cite[Theorem 2.3.1, (iv)]{Bus92}.

\begin{lemma}\label{lem:quadgeo}
Given a hyperbolic quadrilateral with three right angles whose sides have length as labeled as in Figure \ref{fig:hypquad}, then
\[\cosh(b) \tanh(h)= \tanh(R).\]
In particular, if $R$ remains fixed, then $h \to 0$ as $b \to \infty.$
\end{lemma}

\begin{figure}[ht]
    \centering
    \includegraphics[scale=.7]{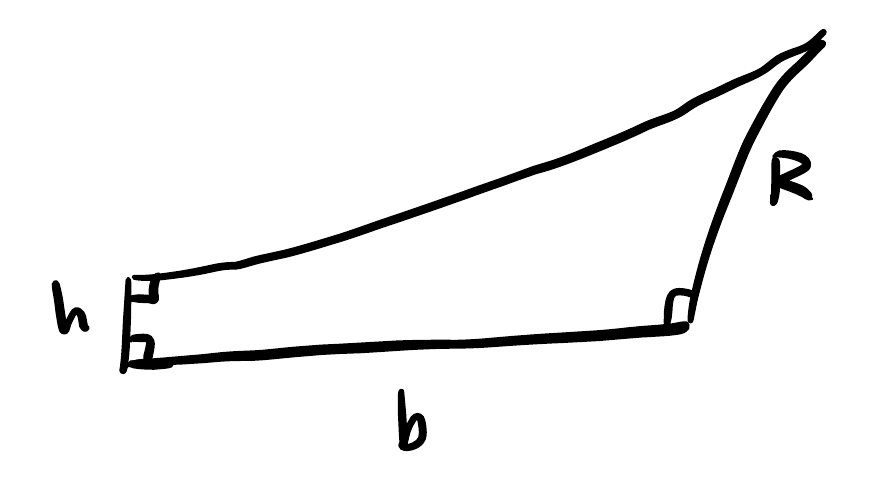}
    \caption{A hyperbolic quadrilateral with three right angles.}
    \label{fig:hypquad}
\end{figure}

\begin{proof}[Proof of Proposition \ref{prop:shrinkMorse}]
We argue by contradiction. Suppose the desired statement did not hold; then there would exist a sequence of $L_n$-bi-Lipschitz curves $c_n$ with $L_n \to 0$ such that each $c_n$ has the same endpoints as $g$ but is not contained in its $\delta$-neighborhood.
Without loss of generality, assume all $L_n < 2$.
Then, by the Morse Lemma, there exists a uniform $R > 0$ so that the curves $c_n$ are all contained within the $R$ neighborhood of $g$.

Translating each $c_n$ along $g$ as necessary (and relabeling the sequence), we can find points $p_n \in c_n$ all of which lie along a common orthogeodesic emanating from a point $o \in g$ and so that
\[d_{\mathbb{H}^2}(p_n, o) \in [\delta, R].\]
See the left-hand side of Figure \ref{fig:Morse_lemma}.
Take a subsequence $c_{n_k}$ so that the $p_{n_k}$ converge to a point $p_\infty$, necessarily lying along the orthogeodesic through $o$ and distance at least $\delta$ from $g$.

\begin{figure}[h]
    \centering
    \includegraphics[scale=.6]{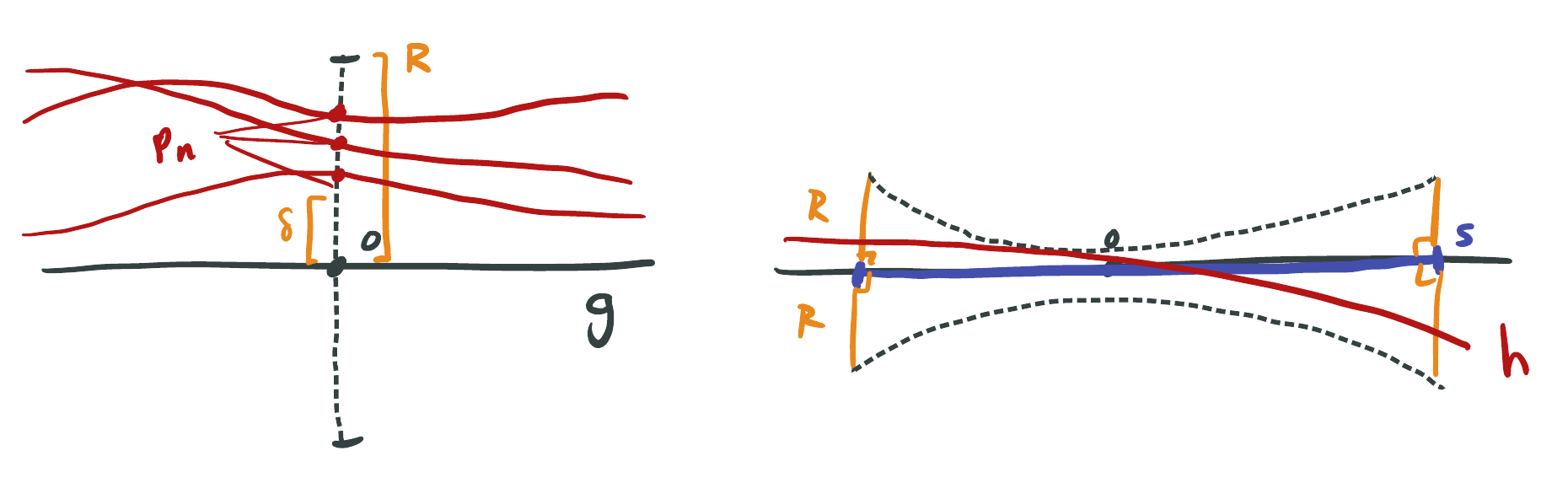}
    \caption{The limiting behavior of bi-Lipschitz paths with improving Lipschitz constant.}
    \label{fig:Morse_lemma}
\end{figure}

Now consider the subsegment $s$ of $g$ of some large length $2b$ centered at $o$ and let $N$ denote the uniform $R$-neighborhood of $s$.
Because the family of $2$-Lipschitz maps is equicontinuous, and $N$ is compact, the Arzela--Ascoli theorem implies that there is a further subsequence (which we also denote $c_{n_k}$) that converges on $N$.
Since $L_n \to 0$, the limit must be a geodesic $h$.

This limiting geodesic $h$ must meet the length $R$ orthogeodesic segments through each of the endpoints of $s$; see the right-hand side of Figure \ref{fig:Morse_lemma}.
Therefore, Lemma \ref{lem:quadgeo} implies that $h$ has distance at most 
$\arctanh \left( \tanh(R) \sech(b) \right)$
from $o$. In particular, for $b$ large this is arbitrarily small.

But now $h$ must also pass through $p_\infty$, which was at least $\delta$ from $o$, a contradiction.
Thus no such sequence $c_n$ can exist, i.e., there is a uniform bound on the distance from an $L$-bi-Lipschitz curve to the geodesic connecting its endpoints.
\end{proof}

Given this, we can immediately deduce the proof of Proposition \ref{prop:geos_to_geos}, which we repeat below for the reader's convenience.

\begin{proposition}\label{prop:geos_to_geos2}
For any small enough $\delta>0$ there exists an $\varepsilon>0$ so that for any $e^\varepsilon$-bi-Lipschitz map $f: \mathbb{H}^2 \rightarrow \mathbb{H}^2$ and any geodesic $g \subset \mathbb{H}^2$, we have
\[d_{\mathbb{H}^2}^H(f(g), g') \le \delta\]
where $g'$ denotes the geodesic with the same endpoints as $f(g)$.
\end{proposition}
\begin{proof}
Using Proposition \ref{prop:shrinkMorse}, take $\varepsilon < \log 2$ so that any $e^\varepsilon$-bi-Lipschitz curve is within $\delta/2$ of the geodesic with the same endpoints.

Now because $g$ is a geodesic, $f(g)$ is an $e^{\varepsilon}$-bi-Lipschitz curve, so by Proposition \ref{prop:shrinkMorse} it must lie in the $\delta/2$ neighborhood of $g'$.
Since $f^{-1}$ is also $e^\varepsilon$-bi-Lipschitz, we also get that $f^{-1}(g')$ is $\delta/2$ close to $g$.
Once more applying the fact that $f$ is bi-Lipschitz, this implies that $g'$ is $e^{\varepsilon} \delta/2 < \delta$ close to $f(g)$, completing the proof.
\end{proof}
\end{appendices}

\bibliographystyle{amsalpha}

\bibliography{bibliography}

\newcommand{\etalchar}[1]{$^{#1}$}
\providecommand{\bysame}{\leavevmode\hbox to3em{\hrulefill}\thinspace}
\providecommand{\MR}{\relax\ifhmode\unskip\space\fi MR }
\providecommand{\MRhref}[2]{%
  \href{http://www.ams.org/mathscinet-getitem?mr=#1}{#2}
}
\providecommand{\href}[2]{#2}
\begin{thebibliography}{ABC{\etalchar{+}}20}

\bibitem[ABC{\etalchar{+}}20]{ABC}
J.~E. Andersen, G.~Borot, S.~Charbonnier, A.~Giacchetto, D.~Lewa{\'n}ski, and
  C.~Wheeler, \emph{On the {K}ontsevich geometry of the combinatorial
  {T}eichm{\"u}ller space}, Preprint, {arXiv:2010.11806}, 2020.

\bibitem[AES16a]{AES16a}
M.~Aka, M.~Einsiedler, and U.~Shapira, \emph{Integer points on spheres and
  their orthogonal grids}, J. Lond. Math. Soc. (2) \textbf{93} (2016), no.~1,
  143--158.

\bibitem[AES16b]{AES16b}
\bysame, \emph{Integer points on spheres and their orthogonal lattices},
  Invent. Math. \textbf{206} (2016), no.~2, 379--396.

\bibitem[{Ara}20]{Ara20a}
F.~{Arana-Herrera}, \emph{Counting hyperbolic multi-geodesics with respect to
  the lengths of individual components and asymptotics of {W}eil--{P}etersson
  volumes}, Geom. Top., {\em (to appear)}, 2020.

\bibitem[{Ara}21]{Ara19b}
\bysame, \emph{Equidistribution of families of expanding horospheres on moduli
  spaces of hyperbolic surfaces}, Geometriae Dedicata \textbf{210} (2021),
  65--102.

\bibitem[BE88]{BowdEpst}
B.~Bowditch and D.~B.~A. Epstein, \emph{{Natural triangulations associated to a
  surface}}, Topology \textbf{27} (1988), no.~1, 91--117.

\bibitem[BH99]{BH}
M.~R. Bridson and A.~Haefliger, \emph{{Metric Spaces of Non-Positive
  Curvature}}, Grundlehren der mathematischen Wissenschaften [Fundamental
  Principles of Mathematical Sciences], vol. 319, Springer-Verlag, Berlin,
  1999.

\bibitem[BS85]{BS85}
J.~S. Birman and C.~Series, \emph{Geodesics with bounded intersection number on
  surfaces are sparsely distributed}, Topology \textbf{24} (1985), no.~2,
  217--225.

\bibitem[Bus92]{Bus92}
P.~Buser, \emph{Geometry and spectra of compact {R}iemann surfaces}, Progress
  in Mathematics, vol. 106, Birkh\"{a}user Boston, Inc., Boston, MA, 1992.

\bibitem[CF]{CF2}
A.~Calderon and J.~Farre, \emph{Continuity of the orthogeodesic foliation and
  the geometry of train tracks}, In preparation.

\bibitem[CF21]{CF}
\bysame, \emph{Shear-shape cocycles for measured laminations and ergodic theory
  of the earthquake flow}, Preprint, {arXiv:2102.13124}, 2021.

\bibitem[CMS11]{CMS_Ksympred}
K.~Chapman, M.~Mulase, and B.~Safnuk, \emph{The {K}ontsevich constants for the
  volume of the moduli of curves and topological recursion}, Commun. Number
  Theory Phys. \textbf{5} (2011), no.~3, 643--698.

\bibitem[DE86]{DE_extensions}
A.~Douady and C.~J. Earle, \emph{Conformally natural extension of
  homeomorphisms of the circle}, Acta Math. \textbf{157} (1986), no.~1-2,
  23--48.

\bibitem[{Do}10]{Do10}
N.~{Do}, \emph{The asymptotic {W}eil--{P}etersson form and intersection theory
  on $\mathcal{M}_{g,n}$}, Preprint, {arXiv:1010.4126}, 2010.

\bibitem[Duk88]{Duke}
W.~Duke, \emph{Hyperbolic distribution problems and half-integral weight
  {M}aass forms}, Invent. Math. \textbf{92} (1988), no.~1, 73--90.

\bibitem[EM18]{EM}
A.~Eskin and M.~Mirzakhani, \emph{Invariant and stationary measures for the
  {${\rm SL}(2,\Bbb R)$} action on moduli space}, Publ. Math. Inst. Hautes
  \'{E}tudes Sci. \textbf{127} (2018), 95--324.

\bibitem[EMM15]{EMM}
A.~Eskin, M.~Mirzakhani, and A.~Mohammadi, \emph{Isolation, equidistribution,
  and orbit closures for the {${\rm SL}(2,\Bbb R)$} action on moduli space},
  Ann. of Math. (2) \textbf{182} (2015), no.~2, 673--721.

\bibitem[ERW19]{ERW19}
M.~Einsiedler, R.~R\"{u}hr, and P.~Wirth, \emph{Distribution of shapes of
  orthogonal lattices}, Ergodic Theory Dynam. Systems \textbf{39} (2019),
  no.~6, 1531--1607.

\bibitem[ES20]{ES20}
V.~Erlandsson and J.~Souto, \emph{Geodesic currents and {M}irzakhani's curve
  counting}, In preparation, 2020.

\bibitem[ES22]{ES21}
\bysame, \emph{Distribution in the unit tangent bundle of the geodesics of
  given type}, Ergodic Theory and Dynamical Systems (2022), 1–17.

\bibitem[GT21]{GT_residue}
Q.~Gendron and G.~Tahar, \emph{Quadratic differentials with prescribed
  singularities}, Preprint, {arXiv:2111.12653}, 2021.

\bibitem[HOP21]{HOP}
Y.~Huang, K.~Ohshika, and A.~Papadopoulos, \emph{The infinitesimal and global
  thurston geometry of teichm{\"u}ller space}, Preprint, {arXiv:2111.13381},
  2021.

\bibitem[Hub16]{Hub16}
J.~H. Hubbard, \emph{Teichm\"uller theory and applications to geometry,
  topology, and dynamics. {V}ol. 2}, Matrix Editions, Ithaca, NY, 2016, Surface
  homeomorphisms and rational functions.

\bibitem[Jen57]{Jenkins}
J.~A. Jenkins, \emph{On the existence of certain general extremal metrics},
  Ann. of Math. (2) \textbf{66} (1957), 440--453.

\bibitem[Ker83]{Ker83}
S.~P. Kerckhoff, \emph{The {N}ielsen realization problem}, Ann. of Math. (2)
  \textbf{117} (1983), no.~2, 235--265.

\bibitem[Kon92]{Kon92}
M.~Kontsevich, \emph{Intersection theory on the moduli space of curves and the
  matrix {A}iry function}, Comm. Math. Phys. \textbf{147} (1992), no.~1, 1--23.

\bibitem[Lin68]{Linnik}
Y.~V. Linnik, \emph{Ergodic properties of algebraic fields}, Ergebnisse der
  Mathematik und ihrer Grenzgebiete, Band 45, Springer-Verlag New York, Inc.,
  New York, 1968, Translated from the Russian by M. S. Keane.

\bibitem[{Liu}19]{Liu19}
M.~{Liu}, \emph{{Length statistics of random multicurves on closed hyperbolic
  surfaces}}, Preprint, {arXiv:1912.11155}, 2019.

\bibitem[Luo07]{Luo}
F.~Luo, \emph{On {T}eichm\"{u}ller spaces of surfaces with boundary}, Duke
  Math. J. \textbf{139} (2007), no.~3, 463--482.

\bibitem[Mar70]{Mar04}
G.~A. Margulis, \emph{On some aspects of the theory of {A}nosov systems}, Ph.D.
  Thesis, 1970, Springer-Verlag, Berlin, 2003.

\bibitem[Mas85]{Mas85}
H.~Masur, \emph{Ergodic actions of the mapping class group}, Proc. Amer. Math.
  Soc. \textbf{94} (1985), no.~3, 455--459.

\bibitem[Mir07]{Mir07c}
M.~Mirzakhani, \emph{Weil--{P}etersson volumes and intersection theory on the
  moduli space of curves}, J. Amer. Math. Soc. \textbf{20} (2007), no.~1,
  1--23.

\bibitem[Mir08a]{Mir08a}
\bysame, \emph{Ergodic theory of the earthquake flow}, Int. Math. Res. Not.
  IMRN (2008), no.~3, Art. ID rnm116, 39.

\bibitem[Mir08b]{Mir08b}
\bysame, \emph{Growth of the number of simple closed geodesics on hyperbolic
  surfaces}, Ann. of Math. (2) \textbf{168} (2008), no.~1, 97--125.

\bibitem[Mir16]{Mir16}
\bysame, \emph{{Counting Mapping Class group orbits on hyperbolic surfaces}},
  Preprint, {arXiv:1601.03342}, 2016.

\bibitem[Mon09a]{Mond_handbook}
G.~Mondello, \emph{Riemann surfaces, ribbon graphs and combinatorial classes},
  Handbook of {T}eichm\"{u}ller theory. {V}ol. {II}, IRMA Lect. Math. Theor.
  Phys., vol.~13, Eur. Math. Soc., Z\"{u}rich, 2009, pp.~151--215.

\bibitem[Mon09b]{Mo09}
\bysame, \emph{Triangulated {R}iemann surfaces with boundary and the
  {W}eil--{P}etersson {P}oisson structure}, J. Differential Geom. \textbf{81}
  (2009), no.~2, 391--436.

\bibitem[MW02]{MW02}
Y.~Minsky and B.~Weiss, \emph{Nondivergence of horocyclic flows on moduli
  space}, J. Reine Angew. Math. \textbf{552} (2002), 131--177.

\bibitem[OP19]{OP_bij}
K.~Ohshika and A.~Papadopoulos, \emph{Bijections of geodesic lamination space
  preserving left {H}ausdorff convergence}, Monatsh. Math. \textbf{189} (2019),
  no.~3, 507--521.

\bibitem[Pen87]{Pen}
R.~C. Penner, \emph{The decorated {T}eichm\"{u}ller space of punctured
  surfaces}, Comm. Math. Phys. \textbf{113} (1987), no.~2, 299--339.

\bibitem[PH92]{PH92}
R.~C. Penner and J.~L. Harer, \emph{Combinatorics of train tracks}, Annals of
  Mathematics Studies, vol. 125, Princeton University Press, Princeton, NJ,
  1992.

\bibitem[Str67]{Strebel}
K.~Strebel, \emph{On quadratic differentials with closed trajectories and
  second order poles}, J. Analyse Math. \textbf{19} (1967), 373--382.

\bibitem[Thu80]{T80}
W.~P. Thurston, \emph{Geometry and topology of three-manifolds}, Lecture notes,
  Princeton University, 1980.

\bibitem[Thu97]{Th_book}
\bysame, \emph{Three-dimensional geometry and topology. {V}ol. 1}, Princeton
  Mathematical Series, vol.~35, Princeton University Press, Princeton, NJ,
  1997, Edited by Silvio Levy.

\bibitem[Thu98]{Thu98}
\bysame, \emph{{Minimal stretch maps between hyperbolic surfaces}}, Preprint,
  {arXiv:9801039}, 1998.

\bibitem[Wol83]{Wol83}
S.~Wolpert, \emph{On the symplectic geometry of deformations of a hyperbolic
  surface}, Ann. of Math. (2) \textbf{117} (1983), no.~2, 207--234.

\bibitem[Wol85]{Wol85}
\bysame, \emph{On the {W}eil--{P}etersson geometry of the moduli space of
  curves}, Amer. J. Math. \textbf{107} (1985), no.~4, 969--997.

\bibitem[Wol86]{Wolp_WPclass}
\bysame, \emph{Chern forms and the {R}iemann tensor for the moduli space of
  curves}, Invent. Math. \textbf{85} (1986), no.~1, 119--145.

\end{thebibliography}
\end{document}